\documentclass[11pt, a4paper,leqno]{amsart}
\usepackage{amsmath,amsthm,amscd,amssymb,amsfonts, amsbsy}
\usepackage{latexsym}
\usepackage{txfonts}
\usepackage{exscale}
\usepackage{graphicx}
\usepackage{bbm}
\usepackage{enumerate}
\usepackage[utf8]{inputenc}

\usepackage[colorlinks,citecolor=red,pagebackref,hypertexnames=false]{hyperref}
\usepackage{color}

\calclayout
\allowdisplaybreaks

\theoremstyle{plain}
\newtheorem{theorem}[equation]{Theorem}
\newtheorem{lemma}[equation]{Lemma}

\newtheorem{proposition}[equation]{Proposition}

\theoremstyle{definition}
\newtheorem{definition}[equation]{Definition}

\theoremstyle{remark}
\newtheorem{remark}[equation]{Remark}

\numberwithin{equation}{section}

\newcommand{\NN}{{\mathbb{N}}}

\newcommand{\eps}{\varepsilon}

\newcommand{\dist}{\operatorname{dist}}

\newcommand{\re}{\mathbb{R}}
\newcommand{\rn}{\mathbb{R}^n}
\newcommand{\ren}{\mathbb{R}^n}

\newcommand{\ree}{\mathbb{R}^{n+1}}

\newcommand{\cA}{\mathcal{A}}

\newcommand{\om}{\Omega}

\newcommand{\cH}{\mathcal{H}}

\newcommand{\W}{\mathcal{W}}

\newcommand{\B}{\mathcal{B}}

\newcommand{\s}{\mathcal{S}}

\newcommand{\pom}{\partial\Omega}

\newcommand{\hm}{\omega}

\newcommand{\bfpsi}{\mbf{\Psi}}

\newcommand{\xbf}{\mathbf{X}}
\newcommand{\mbf}[1]{{\mathbf{#1}}}

\renewcommand{\emptyset}{\text{\textup{\O}}}

\DeclareMathOperator{\supp}{supp}

\DeclareMathOperator{\diam}{diam}

\DeclareMathOperator{\partialp}{\partial_P}

\def\Lip{\mathop{\operatorname{Lip}}\nolimits}
\def\BMO{\mathop{\operatorname{BMO_P}}\nolimits}

\newcommand{\vertiii}[1]{{\left\vert\kern-0.15ex\left\vert\kern-0.15ex\left\vert #1
		\right\vert\kern-0.15ex\right\vert\kern-0.15ex\right\vert}}

\def\Xint#1{\mathchoice
{\XXint\displaystyle\textstyle{#1}}%
{\XXint\textstyle\scriptstyle{#1}}%
{\XXint\scriptstyle\scriptscriptstyle{#1}}%
{\XXint\scriptscriptstyle%
\scriptscriptstyle{#1}}%
\!\int}
\def\XXint#1#2#3{{\setbox0=\hbox{$#1{#2#3}{%
\int}$ }
\vcenter{\hbox{$#2#3$ }}\kern-.6\wd0}}
\def\barint{\,\Xint -} 
\def\bariint{\barint_{} \kern-.4em \barint}
\def\bariiint{\bariint_{} \kern-.4em \barint}
\renewcommand{\iint}{\int_{}\kern-.34em \int} 
\renewcommand{\iiint}{\iint_{}\kern-.34em \int} 

\renewcommand{\d}{\, \mathrm{d}}

\begin{document}
\allowdisplaybreaks

\title[
$\mathrm{L}^p$ solvability implies regularity for graphs]
{Solvability of the $\mathrm{L}^p$ Dirichlet problem for the heat equation is equivalent to parabolic uniform rectifiability in the
case of a parabolic Lipschitz graph.
}

\author{Simon Bortz}
\address{Department of Mathematics
\\
University of Alabama
\\
Tuscaloosa, AL, 35487, USA}
\email{sbortz@ua.edu}

\author{Steve Hofmann}
\address{
Department of Mathematics
\\
University of Missouri
\\
Columbia, MO 65211, USA}
\email{hofmanns@missouri.edu}

\author{Jos\'e Mar{\'\i}a Martell}

\address{
Instituto de Ciencias Matem\'aticas CSIC-UAM-UC3M-UCM\\
Consejo Superior de Investigaciones Cient{\'\i}ficas\\
C/ Nicol\'as Cabrera, 13-15\\
E-28049 Madrid, Spain} \email{chema.martell@icmat.es}

\author{Kaj Nystr\"om}
\address{Department of Mathematics, Uppsala University, S-751 06 Uppsala, Sweden}
\email{kaj.nystrom@math.uu.se}

\thanks{S.B.~was supported by the Simons Foundation’s Travel Support for Mathematicians program.  S.H.~was supported by NSF grants DMS 2000048 and DMS 2349846.
J.M.M.~acknowledges financial support from 
MCIN/AEI/ 10.13039/501100011033 grants CEX2023-001347-S and PID2022-141354NB-I00.
K.N. was partially supported by grant  2022-03106 from the 
Swedish research council (VR)}

\date{August 19, 2023. \textit{Revised}: \today.} 



\begin{abstract}
We prove that if a parabolic Lipschitz (i.e., Lip(1,1/2)) graph domain has the property that its caloric measure is parabolic $A_\infty$ with respect to surface measure (which property is in turn equivalent to $\mathrm{L}^p$ solvability of the Dirichlet problem for some finite $p$),  then the function defining the graph has a half-order time derivative in the space of (parabolic) bounded mean oscillation. Equivalently, we prove that the $A_\infty$ property of caloric measure implies, in this case, 
that the boundary is parabolic uniformly rectifiable. Consequently, by combining our result with the work of Lewis and Murray we resolve a long standing open problem in the field by characterizing those
parabolic Lipschitz graph domains for which one has $\mathrm{L}^p$ solvability (for some $p <\infty$) of the Dirichlet problem for the heat equation.  The 
key idea of our proof is to view the level sets of the
Green function as extensions of the original boundary graph for 
which we can prove (local) square function
estimates of Littlewood-Paley type.
\end{abstract}

\maketitle
\tableofcontents

\section{Introduction}

In this paper we resolve a long standing open problem, in domains defined as regions above  graphs of parabolic Lipschitz functions (Lip(1,1/2) functions), concerning necessary and sufficient conditions for $\mathrm{L}^p$ solvability (for some $p <\infty $) of the Dirichlet problem for the heat equation. To be precise, we prove that $\mathrm{L}^p$ solvability is equivalent to  the function defining the boundary having a half-order time derivative in the space of (parabolic) bounded mean oscillation. In the setting of parabolic Lipschitz graphs, the latter is equivalent to the boundary of the domain being parabolic uniformly rectifiable.  We emphasize that in general, parabolic Lipschitz graphs do {\em not} have this property; we shall return to this point momentarily.

To put our result into context, we recall that in 1977, Dahlberg \cite{DahlbergRH2} proved 
that, for a Lipschitz domain $\Omega\subset \mathbb{R}^n$, 
harmonic measure $\omega$ is mutually absolutely continuous with respect to the surface measure $\sigma$ on $\pom$, and more precisely that the Poisson kernel $d\omega/d\sigma$
is an $A_\infty$ weight with respect to surface 
measure. In fact, Dahlberg proved more, as he established that the Poisson kernel satisfies a scale-invariant reverse H\"older estimate in $\mathrm{L}^2$, and that the $\mathrm{L}^p$ Dirichlet for the Laplace equation in a bounded Lipschitz domain is solvable for all
$p\in (2-\epsilon,\infty)$. At the time,  the problem of finding  the analogue
of Dahlberg’s result for the heat equation, in domains whose boundaries are given locally as graphs
of functions which are Lipschitz in the space variable, was proposed. It was conjectured by Hunt (see \cite[p 2]{KWu}), on the basis of natural homogeneity, that a sufficient regularity condition in the time variable should be Lipschitz of order $1/2$, and hence that the appropriate geometric setting for the parabolic analogue of Dahlberg’s result should be that of Lip(1,1/2) domains. However, subsequent counterexamples of Kaufman and Wu \cite{KWu} showed that the Lip(1,1/2) condition does not suffice even for (qualitative) mutual absolute continuity of caloric measure and parabolic surface measure.

A major breakthrough in the field occurred in 1995, when  Lewis and Murray \cite{LewMur} proved  that if the function defining the graph domain is Lip(1,1/2), and in addition has a half-order time derivative in the space of parabolic BMO, then the caloric measure is parabolic $A_\infty$ (in a local, scale invariant way) with respect to parabolic surface measure on the boundary.  Consequently, Lewis and Murray \cite{LewMur} obtained solvability of the Dirichlet problem for the heat equation with data in $\mathrm{L}^p$, for $p<\infty$ sufficiently large,
but unspecified. Hence, Lewis and Murray \cite{LewMur} established a sufficient condition, in the context of Lip(1,1/2) graph domains,  for the 
$\mathrm{L}^p$ solvability, for some $p <\infty $, of the Dirichlet problem for the heat equation (such solvability is equivalent to caloric measure being parabolic $A_\infty$ with respect to surface measure, in an appropriate scale invariant local sense). We will frequently refer to a Lip(1,1/2) function having this additional regularity as a {\em regular} Lip(1,1/2) function.

Subsequently, in 1996 Hofmann and Lewis \cite{HL96} were able to prove
the solvability of the $\mathrm{L}^2$ Dirichlet
problem (and of the $\mathrm{L}^2$ Neumann and regularity problems)  for the heat equation by the way of layer potentials, in domains given by the region above a regular Lip(1,1/2) graph. They established the $\mathrm{L}^2$ results under the restriction that the half-order time derivative (measured in BMO) of the function defining the graph is small. The smallness is sharp in the
sense that in \cite{HL96} it is proved that there are regular Lip(1,1/2) graph domains for which the $\mathrm{L}^2$ Dirichlet problem is not solvable.

The works of Lewis and Murray \cite{LewMur}  and Hofmann and Lewis \cite{HL96} jointly give the parabolic analogue of the result of \cite{DahlbergRH2} by establishing sufficient conditions on the defining graph for the conclusions. The main result of this paper is that we prove that the condition found by Lewis and Murray \cite{LewMur}, i.e., that the defining function for the domain is a regular Lip(1,1/2) function, is not only {\it sufficient} for the conclusion that caloric measure is parabolic $A_\infty$ (locally) with respect to surface measure, but also {\it necessary}.  Equivalently, we characterize those parabolic Lipschitz domains for which one has L$^p$ solvability of the Dirichlet problem, for some $p<\infty$; thus a necessary and sufficient criterion for solvability with singular data.
In particular, we prove the following theorem. We refer to the sequel for precise definitions, and explanations of notation; especially, as regards the notions of parabolic uniform rectifiability, and of the $A_\infty$ property of caloric measure, 
see Subsection \ref{sspUR} and Remark \ref{PUR} for the former, and
Definition \ref{defainfty} for the latter.

\begin{theorem}\label{mainthrm.thrm}
Suppose $n \ge 2$, let $\psi(x,t): \mathbb{R}^{n-1} \times \re \to \re$ be a Lip(1,1/2) function and let
\[\Omega:= \{\mbf{X} = (x_0,x,t) \in \re \times \mathbb{R}^{n-1} \times \re : x_0 > \psi(x,t)\}. \]
Let $\hm$ denote the caloric measure for $\Omega$ and let $\sigma = \cH_p^{n+1}|_{\partial\Omega}$ be the parabolic surface measure on $\pom$. If $\hm$ is parabolic $A_\infty$ with respect to $\sigma$, then $\psi$ is regular Lip(1,1/2), i.e.,
it has a half-order time derivative in BMO, with norm bounded by a constant depending only on the dimension, the Lip(1,1/2) constant of $\psi$, and the $A_\infty$ constants of $\hm$. In particular, $\partial\Omega$ is parabolic uniformly rectifiable.
\end{theorem}

We note that Theorem \ref{mainthrm.thrm} treats a version of a classical 
1-phase caloric free boundary problem.  We shall return to this point in more detail momentarily.

The theorem has important implications which we summarize as follows.

\begin{theorem}\label{mainconsq.thrm}
Suppose $n \ge 2$, let $\psi(x,t): \mathbb{R}^{n-1} \times \re \to \re$ be a Lip(1,1/2) function and let
\[\Omega:= \{\mbf{X} = (x_0,x,t) \in \re \times \mathbb{R}^{n-1} \times \re : x_0 > \psi(x,t)\}. \]
Let $\sigma = \cH_p^{n+1}|_{\partial\Omega}$ denote  
the parabolic surface measure on $\pom$. The following are equivalent.
\begin{align*}
\mathrm{(i)}&\ \mbox{ The caloric measure for $\Omega$ is a parabolic $A_\infty$ weight with respect to $\sigma$.}\\
\mathrm{(ii)}&\ \mbox{ The function $\psi$ is {\em \bf regular} Lip(1,1/2), i.e., it
has a half-order time derivative in $BMO$.}\\
\mathrm{(iii)}&\ \mbox{ The adjoint caloric measure for $\Omega$ is a parabolic $A_\infty$ weight with respect to  $\sigma$.}\\
\mathrm{(iv)}&\ \mbox{ The $\mathrm{L}^p$ Dirichlet problem for the heat equation is solvable in $\Omega$, for some $p<\infty$.}\\
\mathrm{(v)}&\ \mbox{ The $\mathrm{L}^p$ Dirichlet problem for the adjoint heat equation is solvable in $\Omega$, for some $p<\infty$.}\\
\mathrm{(vi)}&\ \mbox{ $\partial\Omega$ is parabolic uniformly rectifiable.}
\end{align*}
\end{theorem}

We remark that our results in Theorems \ref{mainthrm.thrm} and \ref{mainconsq.thrm} continue to hold in the case $n=1$, as the interested reader may verify, making the natural adjustments.

Our work proves that  the existence of a half-order time derivative in BMO for the graph function is {\it precisely} the extra ingredient needed to obtain $\mathrm{L}^p$ estimates for solutions to the Dirichlet problem, and (equivalently) quantitative absolute continuity for the caloric measure. 
As noted above, in the context of Theorem \ref{mainconsq.thrm},
the fact that (ii) implies (i) (and hence also (iii), by the 
change of variable $t\mapsto - t$), is due to Lewis and Murray \cite{LewMur}.
Our new contribution is that $\mathrm{(i)}$
(or $\mathrm{(iii)}$) implies $\mathrm{(ii)}$, and hence also that 
$\mathrm{(iii)}\!\!\iff\!\!\mathrm{(ii)}$. 
In the context of Lip(1,1/2) domains, the equivalences between $\mathrm{(i)}$ and $\mathrm{(iv)}$, and between $\mathrm{(iii)}$ and $\mathrm{(v)}$, are 
standard and well-known, and may be derived as consequences of the theory of Muckenhoupt weights and boundary estimates for non-negative solutions and caloric measure, see, e.g., \cite{LewMur,N1997} for details, and also \cite[Theorem 2.10]{GH} for a more general result. Concerning $\mathrm{(vi)}$, the notion of parabolic uniformly rectifiable sets was introduced by Hofmann, Lewis, and Nyström 
in \cite{HLN1}, \cite{HLN2}, and 
is the dynamic counterpart of the notion of uniform rectifiability developed 
in the  monumental works of G. David and S. Semmes \cite{DS1}, \cite{DS2}.  
The parabolic version of this 
theory concerns time-varying boundaries 
which are locally not necessarily given by graphs, and
which are minimally smooth from the point of view of, e.g., 
parabolic singular integrals\footnote{That all ``sufficiently nice" parabolic
SIOs are $L^2$ bounded on any parabolic uniformly rectifiable set is observed in
\cite[Corollary 4.9]{BHHLN-BP}. The converse to that result, in general, 
remains open for now, 
but in the present context (the case of Lip(1,1/2) graphs), 
boundedness of parabolic SIOs implies parabolic uniform rectifiability, i.e., that 
the graph is regular Lip(1,1/2).  We include a proof of this observation in an appendix to
the present paper (Appendix \ref{appendixB}).}. 
As in the classical (elliptic, or steady-state) case treated in 
\cite{DS1}, \cite{DS2},
geometry is controlled by a local geometric square function
(the parabolic analogue of the ``$\beta$-numbers" of P. Jones), from which key geometric information and structure can be extracted.
The notions of parabolic uniformly rectifiable sets and parabolic uniform rectifiability extract the geometrical theoretical essence of the (time-dependent) (regular) parabolic Lipschitz graphs introduced in \cite{H95}, \cite{H}, \cite{HL96}, \cite{LewMur}\footnote{We remark that
the regular Lip(1,1/2) condition in \cite{LewMur} appears slightly different to the one considered in the present paper (which is the same as that in
\cite{H95}, \cite{H}, and \cite{HL96}), owing to a different choice of half-order time derivative, but in fact, the conditions are equivalent, as shown in
\cite{HL96}.}, \cite{LS}.  
In particular, in the context of Lip(1,1/2) graphs, a 
graph being {\em regular} Lip(1,1/2) is equivalent to 
the graph being parabolic uniformly rectifiable, i.e., in the statement of Theorem \ref{mainconsq.thrm},  $\mathrm{(ii)}$ is equivalent to 
$\mathrm{(vi)}$. For recent progress on equivalent 
formulations of parabolic uniform rectifiability, 
reminiscent of the ones concerning  uniform 
rectifiability in \cite{DS1}, \cite{DS2}, we refer to \cite{BHHLN-BP,BHHLN-Corona}.
We remark that regular Lip(1,1/2) graphs may be characterized by $L^2$ boundedness of singular integral operators (see Appendix \ref{appendixB}).

As noted above, Theorem \ref{mainthrm.thrm} 
can be viewed in the context of a 
1-phase caloric free boundary problem. Indeed, consider a  
solution $u$ of the heat equation (or adjoint heat equation), 
and a domain $\om$ such that
\begin{equation}\label{upos}
\om = \{ u>0\}.
\end{equation}
Then $u$ vanishes on $\pom$ (the free boundary), hence
\[|\nabla u| = \frac{\partial u}{\partial \nu} \,\quad \text{on} \,\, \pom\,,
\]
where $\partial u/\partial \nu$ is the inward normal derivative.  Since $u$ vanishes
on $\pom$, the problem will be overdetermined if we are 
given information about
$\partial u/\partial \nu$; 
thus in principle, prescribing regularity of $|\nabla u|$ on $\pom$ should 
imply some regularity of the free boundary.   Theorem \ref{mainthrm.thrm} 
is a particular case of the adjoint caloric version of this 
free boundary problem, in 
which $u$ is a Green function $G(X_0,t_0,-,-)$
with some fixed pole 
(in this case, we interpret \eqref{upos} locally, away from the pole), 
and the assumed regularity of $|\nabla u|$ is the $A_\infty$ hypothesis
for the caloric measure.  We observe that by the comparison principle (aka boundary Harnack principle), valid in Lip(1,1/2) domains (see \cite{FGS}), there is no loss of generality in taking $u$ to be a Green function.
The regularity that we 
deduce for the free boundary is that it is a {\em regular} Lip(1,1/2) graph
(i.e., in light of the preceeding remarks, that it is parabolic uniformly rectifiable).

Some historical remarks are in order.
A ``small constant" version of the free boundary problem
described in the preceding paragraph,
may be formulated either above the continuous threshold ($\log k \in C^\alpha$, where $k=d\omega/d\sigma$), or just below that threshold
($\log k \in$ VMO); in the presence of suitable background hypotheses (e.g.,
Reifenberg flatness, and Ahlfors regularity of the boundary), one seeks, in the former case, 
to show that
$\pom \in C^{1,\alpha}$, and in the latter case, effectively that $\pom$ is uniformly rectifiable with ``vanishing constant" (see \cite[Remark, pp 383-384]{HLN1}).
The present paper, and the forthcoming work \cite{BHMN}, can be viewed as treating the  ``large constant" version of this problem:
indeed, our assumption that $\omega \in A_\infty$ is ``almost" the same as assuming 
that $\log k \in$ BMO, and we seek to establish (parabolic) uniform rectifiability of $\pom$.
In the elliptic (i.e., harmonic) setting, the small constant version of the problem has been treated above the continuous threshold in \cite{AC} and \cite{Jerison-Calpha}, and below the continuous threshold in the series of papers \cite{KT1,KT2,KT3}.  The large constant
case appears in restricted form (i.e., assuming that $k\approx 1$) in \cite{LV}, and in full generality in \cite{HLMN}; an alternative proof is given in \cite{MT}.
In the parabolic setting, small constant results were obtained as follows:
in \cite{HLN2}
(a partial result, with an extra hypothesis, below the continuous threshold); in \cite{N2012}
(below the continuous threshold, in the graph case); and in \cite{Eng} (in full generality,
 both above and below the continuous threshold).  In the large constant case, 
 only a weak version of our Theorem \ref{mainthrm.thrm}, under the much more restrictive hypothesis that $k\approx 1$, had hitherto been known \cite{LeNy,N2006}.
It is worthwhile to emphasize that the {\em conclusion} in the large constant case
(namely that $\pom$ is uniformly rectifiable), is a {\em hypothesis} in all the works treating the small constant case.  This hypothesis is imposed
implicitly in the elliptic case (where 
uniform rectifiability is a consequence of the fact that a Reifenberg flat domain with Ahlfors regular boundary is, in particular, a chord-arc domain, and thus has a uniformly rectifiable boundary by the results of \cite{DJ}) and
explicitly in the parabolic case in \cite{HLN2,N2012,Eng}, to rule out the case of a non-regular Lip(1,1/2) graph with vanishing constant (see the example in \cite[p 384]{HLN1}).
Given our results here, and in our forthcoming paper \cite{BHMN}, one expects that the
hypothesis of parabolic uniform rectifiability in \cite{HLN2,N2012,Eng} can be removed.


As noted above,
Theorem \ref{mainthrm.thrm} and Theorem \ref{mainconsq.thrm} resolve
a long-standing problem in this subject,
and are thus of stand-alone interest, 
but in addition, Theorem \ref{mainthrm.thrm} is an essential ingredient in extending this type of free boundary problem 
to more general (non-graph) settings.
Indeed, the results proved in the present work
will play a key role in our forthcoming paper 
\cite{BHMN}, in which we plan to treat similar 
problems in a space-time domain $\Omega$ satisfying an interior 
corkscrew condition, whose boundary 
$\pom=: \Sigma$ 
is a closed subset of $\mathbb R^{n+1}$ which is (only) Ahlfors-David regular
in an appropriate parabolic sense. In \cite{BHMN} we shall prove the caloric version of
\cite{HLMN}; i.e., we shall prove that if caloric measure has the weak-$A_\infty$ property (i.e., $A_\infty$ minus doubling) with respect to the surface measure on $\Sigma$, then $\Sigma$ is parabolic uniformly rectifiable.  The strategy of the proof in \cite{BHMN} is, first, to establish, via some elaborate geometric constructions exploiting the 
weak-$A_\infty$ property, a
Corona approximation in terms
of Lip(1,1/2) graph domains; then, second, 
to obtain parabolic uniform rectifiability of $\Sigma$ by showing that
the constructed approximating Lip(1,1/2) graphs are in fact {\em regular} 
Lip(1,1/2) graphs, as in Theorem \ref{mainthrm.thrm} (ii).  This is achieved
by pushing the $A_\infty$ property of the caloric measure to these graph domains, at which point Theorem \ref{mainthrm.thrm} applies, 
thus establishing
the desired regularity of the graphs. 


In the present work, the key idea of our proof, and the main novelty, 
is to view the level sets of a
(normalized) Green function (which we show are graphs, locally), 
as extensions of the original graph, for 
which we can prove (local) estimates of Littlewood-Paley type.
Using implicit differentiation, we derive
the latter from local square function estimates for the Green function, 
which are in turn a consequence of the $A_\infty$ property of caloric measure, by a refinement (due to \cite{LeNy}) of the 
standard integration by parts argument.
Of course, consideration of the level sets of a solution, per se, 
is not new in free boundary theory (see, e.g., \cite{KindNir, Jerison-Calpha, Eng} for some similar ideas)
but our work seems to be the first to exploit Littlewood-Paley theory 
for the level sets.  Finally, we use the Littlewood-Paley estimates for the level sets to establish the regularity of the function $\psi$ whose graph defines the boundary (i.e., to show that $\psi$ has half a time derivative in parabolic BMO).

The rest of the paper is organized as follows. Section \ref{Prelim} is of a preliminary nature and we here introduce notation and some of the basic terminology, in the context of Lip(1,1/2) domains, to be used in the forthcoming sections. We here also define precisely regular Lip(1,1/2) graph domains. In Section \ref{bbest} we outline and state the results/estimates concerning non-negative solutions to the heat/adjoint heat equation that we will use. All estimates stated are essentially known and can be extracted from the literature,  
although  for the reader's convenience, in Appendix \ref{appendixA}
we shall provide a proof,
 which simplifies existing arguments in the case considered here, of Lemma \ref{lemma:CS-3}. 
In Section \ref{mainthrmsetup.sect} 
we take some initial steps towards the proof of
Theorem \ref{mainthrm.thrm} by exploring the $A_\infty$ condition, by introducing sawtooths, by studying the level sets for the normalized Green function, and by introducing a regularized distance function $h$ which will be an important tool for us. 
The Littlewood-Paley estimates for the level sets, 
mentioned above, 
are proved in Section \ref{SecSq}. In Section \ref{final} all 
ingredients developed are combined and the proof of Theorem \ref{mainthrm.thrm} is completed.

\smallskip

\noindent{\bf Acknowledgements}.  The authors thank the referees for helpful suggestions to improve the exposition, and to clarify prior history of related work.

\section{Preliminaries}\label{Prelim}

\subsection{Notation}
 Throughout the paper, $n \ge 2$ is a natural number and
we let $d:=n+1$ denote the  natural parabolic homogeneous dimension of space-time $\ren$. The ambient space we work in is $\ree:=\re\times \re^{n-1}\times \re$,
\[
\ree = \big\{\xbf= (X,t)=(x_0,x,t) \in \re\times \re^{n-1}\times\re\big\}.
\]
Here we have distinguished the last coordinate as the time coordinate and the first spatial coordinate as the graph coordinate. We also work with
\[
\rn = \big\{ \mbf{x}=(x,t)  \in \re^{n-1}\times\re\big\}.
\]
To help the reader identify the nature of points used in the paper, we use the notation employed above, which we here describe in detail. We use lower case letters
(e.g. $x$, $y$, $z$) to denote spatial points in $\re^{n-1}$, and capital letters (e.g $X = (x_0, x)$, $Y=(y_0, y)$, $Z = (z_0, z)$), to denote points in $\re^{n}=\re\times\re^{n-1}$.
We also use boldface capital letters (e.g. $\mbf{X}=(X,t)$, $\mbf{Y}=(Y,s)$, $\mbf{Z}=(Z,\tau)$), to denote points in 
$\ree$ and boldface lowercase letters (e.g. $\mbf{x}=(x,t)$, $\mbf{y}=(y,s)$, $\mbf{z}=(z,\tau)$) to denote points in $n$-dimensional space-time. In accordance with this notation, given $\mbf{X}=(X,t)\in \ree$ (resp. $ \mbf{x}=(x,t)\in\ren$) we use the notation $t(\mbf{X})$ (resp. $t=t(\mbf{x})$) to denote its time component, that is,  if $\mbf{X}=(X,t)$ then $t=t(\mbf{X})$ (resp. if $\mbf{x}=(x,t)$ then $t=t(\mbf{x})$).

We denote the parabolic length by
\begin{align*}
\|\mbf{X}\|&=\|(X,t)\|:=|X|+|t|^{\frac12}, \quad \mbf{X}=(X,t)\in\ree= \ren\times \mathbb{R}\,,
\\ \|\mbf{x}\|&=\|(x,t)\|:=|x|+|t|^{\frac12}, \quad \mbf{x}=(x,t)\in\ren=\mathbb{R}^{n-1}\times \mathbb{R}\,.
\end{align*}
All distances will be measured with respect to the parabolic metric
\[
\dist(\mbf{X},\mbf{Y}):=\|\mbf{X}-\mbf{Y}\|:=|X-Y|+|t-s|^{\frac12},
\qquad \mbf{X}=(X,t),\ \mbf{Y}=(Y,s)\in\ree,
\]
and
\[
\dist(\mbf{x},\mbf{y}):=\|\mbf{x}-\mbf{y}\|:=|x-y|+|t-s|^{\frac12},
\qquad \mbf{x}=(x,t),\ \mbf{y}=(y,s)\in\rn.
\]
It is sometimes convenient to use a different (smooth) parabolic distance. Given $\mbf{x}=(x,t)\in\ren\setminus\{0\}$, 
we let $\vertiii{\mbf{x}}$ 
be defined as the unique positive solution of the equation
\begin{equation}\label{eqpdist}
\frac{|x|^2}{\vertiii{\mbf{x}}^2}+\frac{t^2}{\vertiii{\mbf{x}}^4}=1, 
\end{equation}
that is,
\begin{equation}\label{eqn:parabolic-dist-new}
	\vertiii{\mbf{x}}= 2^{-\frac12}\big(\sqrt{|x|^4+4t^2}+|x|^2\big)^{\frac12}, \quad  \mbf{x}=(x,t)\in\ren\,.
\end{equation}
Then $\vertiii{\mbf{x}}\approx \|\mbf{x}\|$ 
with implicit constants depending 
only on $n$.

Given $\mbf{x}=(x,t)\in\ren$ and $R>0$, we introduce the parabolic
cube, centered at $\mbf{x}$ and of size $R$, as
 \[
Q_R(\mbf{x})
:=
\big\{
\mbf{y}=(y,s)\in\ren: |y_i-x_i|<R, \, 1\leq i\leq n-1,\,\ |t-s|< R^2
\big\}.
\]
Given $\mbf{x}=(x,t)\in\ren$ and $R>0$, 
we denote a closed parabolic cube in $\ree$ by
\begin{equation}\label{cuba}\mathcal{J}_{R}(\mbf{X}) = \mathcal{J}_{R}(x_0, x,t):=[x_0-R, x_0+R]\times \overline{Q_R(\mbf{x})}\,.
\end{equation}
For $\mathcal{J} = \mathcal{J}_{R}(\mbf{X})$, we let $\ell(\mathcal{J}) := \ell(\mathcal{J}_{R}(\mbf{X}))= 2R$
denote the parabolic side length of $\mathcal{J}$.

\subsection{{Parabolic} Hausdorff measure}
Given $\eta \geq 0$, we let $\cH^\eta$ denote
 standard $\eta$-dimensional Hausdorff measure.
  We also define a {parabolic} Hausdorff measure of {homogeneous}
  dimension $\eta$, denoted
  $\cH_{\text p}^\eta$, in the same way that one defines standard Hausdorff measure, but instead using coverings
  by {parabolic} cubes. I.e., for $\delta>0$, and for $E\subset \mathbb R^{n+1}$, we set
  \[ \cH_{\text{p},\delta}^\eta(E):= \inf \sum_k \diam(E_k)^\eta\,,
  \]
  where the infimum runs over all countable such coverings of $E$, $\{E_k\}_k$, with $\diam(E_k)\leq \delta$ for all $k$. Of course, the diameter is measured in the parabolic metric.  We then define
  \[
  \cH_{\text p}^\eta (E) := \lim_{\delta\to 0^+} \cH_{\text{p},\delta}^\eta(E)\,.
  \]
As is the case for classical Hausdorff measure, $ \cH_{\text{p}}^\eta$ is a Borel regular measure.
We refer the reader to \cite[Chapter 2]{EG} for a discussion of the basic properties of standard
Hausdorff measure.  The arguments in \cite{EG} adapt readily to treat $  \cH_{\text{p}}^\eta$.
In particular, one obtains a measure equivalent to   $\cH_{\text{p}}^\eta$ if one defines
$\cH_{\text{p},\delta}^\eta$ in terms of coverings by arbitrary sets of parabolic diameter at most $\delta$, rather than
cubes.

\subsection{Lip(1,1/2) graph domains, and surface measure on the boundary}\label{Lipnot}  A function $\psi:\mathbb R^{n-1}\times\mathbb R\to \mathbb R$ is called Lip(1,1/2) with constant $C$, if
\begin{align}\label{1.1}
|\psi(x,t) - \psi(y,s)| \le C(|x - y| + |t - s|^{1/2}) = C\|(x,t)-(y,s)\|, \quad \forall (x,t), (y,s) \in \rn.
\end{align}
We define $\|\psi\|_{\Lip(1,1/2)}$ to be the infimum of all constants $C$ as in \eqref{1.1}.  If we set
\begin{eqnarray}\label{1.1++}
\Sigma:=
\big\{(\psi(x,t), x, t): (x,t)\in\ren\big\}
=
\big\{(\psi(\mbf{x}), \mbf{x}): \mbf{x}\in\ren\big\} =:
\big\{\mbf{\Psi}(\mbf{x}): \mbf{x}\in\ren\big\},
\end{eqnarray}
 then we say that $\Sigma$ is a Lip(1,1/2) graph. The set $\Sigma$ is the boundary of the domain
\begin{eqnarray}\label{1.1+++}
\Omega:=\big\{\mbf{X}=(x_0,x,t)\in\ree: x_0>\psi(x,t)\big\},
\end{eqnarray}
We refer to $\Omega\subset\mathbb R^{n+1}$  as an
(unbounded) Lip(1,1/2) graph
domain with constant $\|\psi\|_{\Lip(1,1/2)}$. Given the closed set $\Sigma \subset \mathbb R^{n+1}$
  of  homogeneous dimension $\cH_{\text{p}, \text{dim}}(\Sigma)=n+1$, we
define  a surface measure on $\Sigma$ as the restriction of $\cH_{\text{p}}^{n+1}$ to $\Sigma$, i.e.,
\begin{equation}\label{sigdef}
\sigma = \sigma_\Sigma:=  \cH_{\text{p}}^{n+1}|_\Sigma\,.
\end{equation}

We remark that for Lip(1,1/2) graphs, $\sigma$ as defined in \eqref{sigdef} is equivalent to $d\sigma^{\bf s}:= d\sigma_t \,dt$, where $d\sigma_t$ is 
standard $(n-1)$-dimensional Hausdorff measure $\cH^{n-1}$, 
restricted to the cross section 
$\Sigma_t:= \{x: (x,t) \in \Sigma\}$.  We refer the reader to 
\cite[Remark 2.8 and Appendix B]{BHHLN-Corona} for details.

\subsection{Surface cubes and reference points}\label{scuberef} We let
\begin{eqnarray}\label{1.1+}
M_0 := 1 + \|\psi\|_{\Lip(1,1/2)}.
\end{eqnarray}
For every $\mbf{X}=(x_0,\mbf{x})=(x_0,x,t) \in\ree$ and $R>0$, we introduce vertically elongated open ``cubes" 
\begin{equation}\label{eq3.1}
 I_R(\mbf{X}):=(x_0-3M_0\sqrt{n}R, x_0+3M_0\sqrt{n}R)\times Q_R(\mbf{x})
\end{equation}
and set
$$\Delta_R(\mbf{X}):=I_R(\mbf{X})\cap\Sigma.$$ We will refer to $\Delta_R(\mbf{X})$ as a surface box or cube of size $R>0$ and centered at $\mbf{X}$. Unless otherwise specified, we implicitly assume that the center
$\mbf{X}=(x_0,\mbf{x})=(x_0,x,t)$, of any surface box 
$\Delta_R(\mbf{X})$, is in $\Sigma$, that is, $x_0=\psi(x,t)$.

Note the crude estimate
\begin{equation}\label{crdIrest.eq}
\mbf{Y} \in  I_R(\mbf{X}) \implies \|\mbf{Y} - \mbf{X}\| \le 5M_0\sqrt{n} R,
\end{equation}
and that by construction,
\begin{equation}\label{surface-box:Lip}
	\Delta_R(\mbf{X})
=
\big\{ \bfpsi(\mbf{y}): \mbf{y}\in Q_R(\mbf{x})\big\},
\qquad
\forall\,\mbf{X}=(x_0,\mbf{x})\in\Sigma,
\end{equation}
where we recall that $\mbf{y}\mapsto \bfpsi(\mbf{y}):=(\psi(\mbf{y}),\mbf{y})$
is the graph parametrization of $\Sigma$ (see \eqref{1.1++}).
Indeed, if $\mbf{y}\in Q_R(\mbf{x})$, 
then $\|\mbf{x} - \mbf{y}\| \le (\sqrt{n}R + R)$ and hence
\[|\psi(\mbf{x}) - \psi(\mbf{y})| \le M_0 (\sqrt{n}R + R) \le 2\sqrt{n}M_0 R.\]
We also note, by the same reasoning, that if $R > 0$, $\mbf{X}=(x_0,\mbf{x})\in\Sigma$ and $\mbf{y}\in Q_R(\mbf{x})$, then
\[
(\psi(\mbf{y}) + a,\mbf{y}) \in I_R(\mbf{X}) \quad \forall a \in (-M_0\sqrt{n}R, M_0\sqrt{n}R)\,.
\]
Given $\tau>0$, we define the parabolic dilation
$\tau\Delta_R(\mbf{X}):=\Delta_{\tau R}(\mbf{X})$.

We introduce time forward and time backwards corkscrew points relative to $\Delta_R(\mbf{X})$,
\begin{equation}\label{CSpm}
\cA^\pm_R(\mbf{X})  := \left(x_0+ 2 M_0R, x, t\pm 2 R^2\right),   \quad  \mbf{X}=(x_0,x,t)= (x_0,\mbf{x})=(\psi(\mbf{x}),\mbf{x}) \in \Sigma,
\end{equation}
and we let
\begin{equation}\label{CSnopm}
\cA_R(\mbf{X})  := \left(x_0+ 2 M_0R, x, t\right),   \quad  \mbf{X}=(x_0,x,t)= (x_0,\mbf{x})=(\psi(\mbf{x}),\mbf{x}) \in \Sigma.
\end{equation}
Note that for $\mbf{X}\in\Sigma$,
\begin{equation}\label{CS2}
\cA^\pm_R(\mbf{X}),\ \cA_R(\mbf{X}) \in \Omega_{2R}
=\Omega_{2R}(\mbf{X}) := I_{2R}(\mbf{X}) \cap \Omega,
\end{equation}
and that $$\dist\left(\cA^\pm_R(\mbf{X}), \partial \Omega_{2R}\right)\approx \dist\left(\cA^\pm_R(\mbf{X}), \partial \Omega_{3R}\right) \approx R,$$
where the implicit constants depend 
only on $n$ and $\|\psi\|_{\Lip(1,1/2)}$.  Furthermore,
\[
|\psi(x_0, x, t\pm 2 R^2) - \psi(x_0, x, t)| \le \sqrt{2}\|\psi\|_{\Lip(1,1/2)}\,R\,,
\]
and hence
\begin{equation}\label{morecslocinfo.eq}
\dist(\cA^\pm_R(\mbf{X}), \partial \Omega) \,\ge\, 2M_0R - 
\sqrt{2}\|\psi\|_{\Lip(1,1/2)}\,R\, \ge\, 2R.
\end{equation}
The same argument and conclusions apply to $\cA_R(\mbf{X})$.


We frequently use (sometimes without mention) the following elementary consequence of our definitions.
\begin{lemma}\label{distgraphlem.lem} Assume that $\mbf{X}=(x_0,x,t)\in \Omega$. Then
\[(x_0 - \psi(x,t))M_0^{-1} \le \dist(\mbf{X}, \Sigma) \le x_0 - \psi(x,t).\]
\end{lemma}
\begin{proof}
Let $L = \|\psi\|_{\Lip(1,1/2)}$. We assume, 
without loss of generality, that $(x,t)= 0$ and $\psi(0,0) = 0$ 
so that $x_0 = (x_0 - \psi(x,t))$. 
Let $\mbf{Y} = (y_0, \mbf{y})\in\Sigma$ 
be such that $\dist(\mbf{X}, \Sigma) = \|\mbf{Y} - \mbf{X}\|$. If $\|\mbf{y}\| \ge x_0(1 + L)^{-1}$ then $\|\mbf{Y} - \mbf{X}\| \ge \|\mbf{y}\| \ge (x_0 - \psi(x,t))(1 + L)^{-1}$ and we are done. So we may assume that $\|\mbf{y}\| \le (1 + L)^{-1}x_0$. Since $\psi$ is Lip(1,1/2) it holds $|y_0| \le \tfrac{L}{1 + L} x_0$. Thus,
\[\|\mbf{Y} - \mbf{X}\| \ge |x_0 - y_0| \ge x_0\left(1 - \frac{L}{1 +L} \right) = \frac{x_0}{1 + L} = \frac{x_0 - \psi(x,t)}{1 +L} .\]
As $M_0=1+L$ this proves the lemma as the inequality $\dist(\mbf{X}, \Sigma) \le x_0 - \psi(x,t)$ is trivial.
\end{proof}

\subsection{(Parabolic) $BMO$ and fractional integral operators} Given a function $f:\ren\to \re$, which is locally integrable with respect to the $n$-dimensional Lebesgue measure, we say that $f\in\BMO(\ren)$, the parabolic BMO-space, if
\[
\|f\|_{\BMO(\ren)}
:=
\sup_{Q\subset\re^n}
\bariint_{Q} |f(x,t)-f_Q|\,\d x\d t<\infty,
\]
where the supremum runs over all parabolic cubes $Q=Q_R(\mbf{x})$, with $\mbf{x}\in\ren$ and $R>0$. Here $f_Q$ denotes the average of $f$ on $Q$.

We introduce $\mathrm{I_P}$, the fractional integral operator of parabolic order 1 on $\ren$,
by means of the Fourier transform,
\begin{equation}\label{def-IP:hat}
	(\mathrm{I_P} \psi)^{\,\widehat{}}\, (\xi,\tau):= \vertiii{(\xi,\tau)}^{-1}\,\widehat{\psi}(\xi,\tau),
	\qquad
	(\xi,\tau)\in\re^n.
	\end{equation}
Then
\begin{equation}\label{def-IP}
\mathrm{I_P}\psi(\mbf{x})=\iint_{\re^n} V(\mbf{x}-\mbf{y})\,\psi(\mbf{y})\,\d\mbf{y},
\qquad
\mbf{x}\in\re^n,
\end{equation}
for a kernel $V$ satisfying $0\le V(\mbf{y})\lesssim \|\mbf{y}\|^{1-d}$, where 
$d=n+1$ is the homogeneous dimension of parabolic $\rn$.

Using $\mathrm{I_P}$ we introduce a (parabolic) half-order time derivative as
\[
\mathcal{D}_t \psi(\mbf{x})
:=
\partial_t \circ\mathrm{I_P} \psi(\mbf{x})
=
\iint_{\re^n} \partial_t\big(V(\mbf{x}-\mbf{y})\big)\,\psi(\mbf{y})\,\d\mbf{y},
\qquad
\mbf{x}=(x,t)\in\re^n.
\]
This operator should be viewed as a principal value operator, or one should consider $\partial_t$ in the weak sense. Note that
the Fourier symbol for $\mathcal{D}_t $ is $2\pi i\tau/\vertiii{(\xi,\tau)}$.

Another half-order time derivative, $D_{1/2}^t$, can be introduced through the Fourier multiplier $|\tau|^{1/2}$ or by
\begin{eqnarray} \label{1.8gg}
 D_{1/2}^t  \psi (\mbf{x})=D_{1/2}^t  \psi (x,t):=c \int_{ \mathbb R }
\, \frac{ \psi ( x, s ) - \psi ( x, t ) }{ | s - t |^{3/2} } \, \d s,
\end{eqnarray} for properly chosen $c$.

\subsection{Regular Lip(1,1/2) graph domains}\label{Lipnotreg} Let $\psi:\mathbb R^{n-1}\times\mathbb R\to \mathbb R$ be a Lip(1,1/2) function
with norm $\|\psi\|_{\Lip(1,1/2)}$ in the sense of \eqref{1.1}.  Such a function is said to be a {\em regular} Lip(1,1/2) function if, in addition,
 \begin{eqnarray} \label{1.7}
 \mathcal{D}_t \psi\in \BMO(\ren).
\end{eqnarray}
If $\psi$ is a {regular} Lip(1,1/2)  function, and if  we define $\Sigma$ and $\Omega$ as in \eqref{1.1++} and \eqref{1.1+++}, then we say that
$\Sigma$ is a {\em regular} Lip(1,1/2) graph, and  that $\Omega\subset\mathbb R^{n+1}$  is an
(unbounded) regular Lip(1,1/2)  graph
domain. In both cases the regularity is determined by $\|\psi\|_{\Lip(1,1/2)}$ and $\|\mathcal{D}_t \psi\|_{\BMO(\ren)}$. In \cite{HL96} it is proved that
\begin{equation}\label{eqnormequiv}
\|\psi\|_{\text{R-Lip}}:=
\|\mathcal{D}_t \psi\|_{\BMO(\ren)}+\|\nabla_x\psi\|_{\infty}\approx \|D_{1/2}^t \psi\|_{\BMO(\ren)}+\|\nabla_x\psi\|_{\infty}\,,
\end{equation}
where $D_{1/2}^t$ was introduced in \eqref{1.8gg}. In particular, that a function is a {regular} Lip(1,1/2)  function can be equivalently
formulated using $D_{1/2}^t$ instead of $\mathcal{D}_t$, but the latter 
will be considerably more convenient for us to work with in this paper.

\begin{remark}  One can prove that in general, the class of regular Lip(1,1/2) functions is strictly contained in Lip(1,1/2), i.e., there are examples of
functions $\psi$ which are  Lip(1,1/2) but not regular Lip(1,1/2), see \cite{LS},
\cite{KWu}.  Moreover, it follows from the arguments of Strichartz
\cite{Stz} that for a regular Lip(1,1/2) function $\psi$, the assumption that 
$\psi$ is Lip(1/2) in the 
time variable is redundant: it 
follows from the finiteness of the R-Lip norm in \eqref{eqnormequiv}.
\end{remark}

\subsection{Parabolic uniform rectifiability} \label{sspUR}
Let $\Sigma \subset \mathbb R^{n+1}$ be a closed set, and define surface measure $\sigma$ on $\Sigma$ as in \eqref{sigdef}.
Assume that $\Sigma$ is parabolic Ahlfors-David regular, i.e., for all $\mbf{X}\in \Sigma$, and for $0<r<\diam(\Sigma)$ (with diameter measured in the parabolic metric), 
$\sigma\left({\tt D}_r(\mbf{X})\right) \approx r^{n+1}$, with uniform implicit constants, where
${\tt D}_r(\mbf{X}):=\{\mbf{Y}\in\Sigma: \|\mbf{Y}-\mbf{X}\|<r\}$.
For $\mbf{X}, r$ as above, set
$$
	\widetilde{\beta}(\mbf{X}, r  ):=\inf_{P \in \mathcal{P}}  \biggl ( \, \bariint_{{\tt D}_r(\mbf{X})}  \, \biggl (\frac {\dist ( \mbf{Y}, P )}{r}\biggr )^2  \d \sigma (\mbf{Y} )\biggr )^{1/2}, 
$$
 where  $\mathcal{P}$  is the set of all $ n $-dimensional hyperplanes $ P $ containing a line
parallel to the $ t $ axis (that is, $t$-independent planes).
Define 
$$
\d \widetilde{\nu} ( \mbf{X},r  )  := \widetilde{\beta}^{\,2}( \mbf{X}, r)\, \d \sigma (  \mbf{X}) \, r^{ - 1 }
\d r\,.
$$  
We then say that $\Sigma$ is {\em parabolic uniformly rectifiable} if $\widetilde{\nu}$
is a (parabolic) Carleson measure on  
$\Sigma\times(0,\diam(\Sigma))$, i.e., if 
\begin{align*}
\|\widetilde{\nu}\|_{\mathcal{C}}:=\sup_{\mbf{X}\in\Sigma,\, 0<r<\diam(\Sigma) }
  r^{ -(n+1) }\,
\widetilde{\nu} \big({\tt D}_r(\mbf{X}) \times ( 0, r) \big) <\infty
\end{align*}

\begin{remark}\label{PUR} One can prove, in the context of Lip(1,1/2) graphs, that the graph being regular Lip(1,1/2) is equivalent to the graph being parabolic uniform rectifiable, see \cite{HLN2} for a proof. In particular, 
let $\Sigma$ be the graph of a function $\psi(x,t)$, and set
\[\beta(r,x,t) := \inf_{L} \left[\bariint_{Q_r(x,t)} \left(\frac{\psi(y,s) - L(y)}{r} \right)^2 \,\d\sigma(y,s)\right]^{1/2}, \quad (r,x,t) \in \ree_+ \]
where the infimum is taken over all affine functions $L$ of $y$ only.
 If we let
\[\d\nu:= \d\nu_\psi := \beta^2(r,x,t) \frac{\, \d r\, \d x\, \d t}{r},\]
and assume that $\psi(x,t)$ is Lipschitz in the space variable
$x$, uniformly in $t$, then
 the condition that $\mathcal{D}_t \psi\in \BMO(\ren)$ is 
equivalent to saying that $\d\nu$ is a Carleson measure on $\ree_+$, i.e., 
\begin{equation}\label{Carlestforh.eq}
\|\nu\|:=\sup_{(z,\tau)\in\mathbb R^n,\, R>0}\int_0^R\bariint_{Q_R(z,\tau)}\beta^2(r,x,t) \frac{\, \d r\, \d x\, \d t}{r}<\infty\,.
\end{equation}
Moreover, since $\psi$ is Lipschitz in $x$, uniformly in $t$, it is not hard to see that 
(suitably interpreted) $\beta \approx \widetilde{\beta}$ 
and $\nu\approx \widetilde{\nu}$ on such graphs.
Since the further property that
$\mathcal{D}_t \psi\in \BMO(\ren)$ implies in particular that
 $\psi(x,t)$ is Lip(1/2) in $t$, uniformly in $x$, we see that
in the context of Lip(1,1/2) graphs, \eqref{Carlestforh.eq} with $\|\nu\|$ finite, is the very definition of $\Sigma$ being parabolic uniform rectifiable.
\end{remark}

\subsection{Convention concerning constants}\label{coo} We refer to $n$, the $\Lip(1,1/2)$ constant of the function defining the boundary of our domain, and the constants $C_\ast$ and $q  > 1$ appearing in Definition \ref{defainfty} below, as the {\it structural constants}. For all constants $A,B\in \mathbb R_+$, the notation $A\lesssim B$  means, unless otherwise stated, that $A/B$ is bounded from above by a positive constant depending at most on the structural constants; $A\gtrsim B$ of course means $B\lesssim A$. We write $A\approx B$  if $A\lesssim B$ and  $B\lesssim A$, while for a given constant $\eta$,
$A\lesssim_\eta B$  means, unless otherwise stated, that $A/B$ is bounded from above by a positive constant depending at most on the structural constants and 
$\eta$.

\section{Boundary estimates in $\Lip(1,1/2)$ domains}\label{bbest} The Dirichlet problem, parabolic measure, and the boundary behaviour of non-negative solutions,  for the heat equation but also for more general linear uniformly parabolic equations with space and time dependent coefficients,  have been studied intensively in Lipschitz cylinders and in Lip(1,1/2) domains over the years, see \cite{ 
FGS,FS,FSY,LewMur,N1997}.
Results include Carleson type estimates, the relation between the associate parabolic measure and the Green function, the backward in time Harnack inequality, the doubling of parabolic measure, boundary Harnack principles (local and global) and  H\"older continuity up to the boundary of quotients of non-negative solutions vanishing on the lateral boundary. We here only state  the results/estimates for the heat equation that we will use. All estimates stated are known and also apply for solutions to the adjoint heat equation subject to the appropriate changes (typically just exchanging $\cA_R^+$ with $\cA_R^-$) induced by the change of variables $t \to -t$.

In the following we consider a  Lip(1,1/2) domain $\Omega$ as in \eqref{1.1+++} with boundary $\Sigma$.  It is well known that the bounded continuous Dirichlet
problem for the heat
equation always has a unique solution in $\Omega$. Given $\mbf{Y}\in \Omega$ we let $G(\cdot)=G ( \cdot, \mbf{Y})$
 denote Green's function for the heat
equation in $ \Omega$ with pole at $ \mbf{Y}$, i.e.
 \begin{eqnarray}\label{1.4a}
(\partial_t-\Delta)G (\mbf{X}, \mbf{Y})   = \delta_{\mbf{Y}} (  \mbf{X} ) \mbox{ in }
\Omega \mbox{ and } G \equiv 0 \mbox{ on }
\Sigma.
\end{eqnarray}
Here
$\delta_{\mbf{Y}}$ is the Dirac  delta function at $\mbf{Y}$ and $ \Delta
$ is the Laplacian in $ X$. Furthermore, we note that $ G
( \mbf{Y}, \cdot ) $ is the Green's function for the adjoint heat
equation with pole at $ \mbf{Y} \in \Omega $, i.e.
 \begin{eqnarray}\label{1.4b} (-\partial_t - \Delta)G (\mbf{Y},\mbf{X})  = \delta_{\mbf{Y}} (  \mbf{X} ) \mbox{ in }
\Omega \mbox{ and } G \equiv 0 \mbox{ on }
\Sigma.
\end{eqnarray}
We  let $ \omega^{\mbf{Y}}(\cdot)$ and $\widetilde{\omega}^{\mbf{Y}}( \cdot) $ be the caloric and adjoint caloric measures,
at $\mbf{Y}\in\Omega$, associated to the heat and  adjoint heat
equation in $ \Omega$. Given $\mbf{Y}\in\Omega$ we let $G(\mbf{X},\mbf{Y})\equiv 0$ whenever
$\mbf{X}\in (\mathbb R^{n}\times (s,\infty))\setminus\Omega$ and $G(\mbf{Y},\mbf{X})\equiv 0$ whenever
$\mbf{Y}\in (\mathbb R^{n}\times (-\infty,s))\setminus\Omega$. Then,
 \begin{align}\label{1.5}
 \iiint G ( \mbf{Y}, \mbf{X} ) (  \Delta -\partial_t)\phi  \, \d \mbf{X} &= \iint\phi  \, \d \omega^{\mbf{Y}},\notag\\
 \iiint G (\mbf{X},\mbf{Y}) (  \Delta +\partial_t)\phi   \, \d \mbf{X} &= \iint \phi  \, \d \widetilde{\omega}^{\mbf{Y}},
\end{align}
for all $ \phi \in
C_0^\infty ( \mathbb R^{ n + 1 } \setminus \{  \mbf{Y} \} ). $

 For the proofs of Lemmas \ref{HCatBdry.lem}, \ref{carlest.lem}, \ref{strongharnack4gf.lem}, \ref{CFMS},  and  \ref{calisdoubpre.lem} below  we refer to \cite{FGS,FS,FSY,LewMur,N1997}.  In \cite{N1997}
all relevant estimates are stated and proved in the general setting of second order
parabolic equations in divergence form in Lip(1,1/2) domains. We also note that all of these lemmas remain valid in more general settings, see e.g., \cite{HLN2}.   To give specific references in the case of the heat/adjoint heat equation, 
 we refer the reader to \cite[Chapter 3, Section 6]{LewMur} for  
 Lemmas \ref{HCatBdry.lem} and \ref{carlest.lem},
 and to \cite{FGS} 
 for Lemmas  \ref{strongharnack4gf.lem}, 
 \ref{CFMS},  and  \ref{calisdoubpre.lem}. 

In this section all implicit constants depend only on $n$ and $\|\psi\|_{\Lip(1,1/2)}$. Given $Y\in\Omega$ we let $\delta(\mbf{Y}) := \inf_{\mbf{Z} \in \Sigma}\|\mbf{Y} - \mbf{Z}\|$.

\begin{lemma}\label{HCatBdry.lem}
Let $\mbf{X}\in\Sigma$ and $R>0$. Assume that
	$0 \le u\in C(\overline{I_{2R}(\mbf{X})\cap\Omega})$  satisfies
	$\partial_tu -\Delta u=0$ in $I_{2R}(\mbf{X})\cap\Omega$, with
	$u=0$ in $\Delta_{2R}(\mbf{X})$.
Then there exists $\alpha \in (0, 1/2)$, depending only on $n$ and $\|\psi\|_{\Lip(1,1/2)}$, such that
\[u(\mbf{Y}) \lesssim \left(\frac{\delta(\mbf{Y})}{R}\right)^{\alpha} \sup_{\mbf{Z} \in I_{2R}(\mbf{X})\cap\Omega} u(\mbf{Z}),\]
whenever $Y \in  I_{R}(\mbf{X}) \cap \Omega$.
\end{lemma}

\begin{lemma}\label{carlest.lem} Let $\mbf{X}$, $R$, $u$, and $\alpha$ be as in the statement of
Lemma \ref{HCatBdry.lem}. Then
\begin{equation}\label{carlest.eq}
u(\mbf{Y}) \lesssim u(\cA^+_R(\mbf{X})),
\end{equation}
whenever $Y \in  I_{R}(\mbf{X}) \cap \Omega$. In particular,
\begin{equation}\label{carlhcest.eq}
u(\mbf{Y}) \lesssim \left(\frac{\delta(\mbf{Y})}{R}\right)^{\alpha} u(\cA^+_R(\mbf{X})),
\end{equation}
whenever $Y \in  I_{R/2}(\mbf{X}) \cap \Omega$.
\end{lemma}

For $\mbf{X}= (x_0, x, t) \in \Sigma$,  $r > 0$, and
 $\kappa =\kappa(n,M_0)$, a sufficiently large constant to be fixed, we introduce the space-time parabolas
\[T^\pm_{\kappa, r}(\mbf{X}) = \{(y_0,y,s) \in \Omega: |(x_0,x) - (y_0,y)| \le \kappa |t-s|^{1/2}, \, \pm(s-t)\ge 16r^{2}\}.\]
Recall that $M_0 = 1 + \|\psi\|_{\Lip(1,1/2)}$. The parabola $T^+_{\kappa, r}$ is the forward in time parabola and $T^-_{\kappa, r}$ is the backward in time parabola. Note that
if $\mbf{Y} \in T^\pm_{\kappa, r}(\mbf{X})$, then $\mbf{Y} \in T^\pm_{\kappa, r'}(\mbf{X})$ for all $r' \in (0,r)$. Concerning $\kappa$,
we may take this constant as large as we like but we will choose
\begin{equation}\label{kappadef}\kappa:=  40M_0\sqrt{n}.
\end{equation}
Hence $\kappa$ only depends on  $n$ and $\|\psi\|_{\Lip(1,1/2)}$.

Lemma \ref{strongharnack4gf.lem} below states the strong Harnack inequality, or backwards in time Harnack inequality, for the Green function. 
Lemma \ref{CFMS} gives the relation between the Green function and caloric/ad\-joint caloric measure, and Lemma \ref{calisdoubpre.lem} formulates the doubling property
of the latter measures.

\begin{lemma}[\cite{FGS}]\label{strongharnack4gf.lem} If $\mbf{X} \in \Sigma$ and $\mbf{Y} \in T^+_{\kappa, R}(\mbf{X})$, then
\[G(\mbf{Y}, \cA_R^-(\mbf{X})) \approx G(\mbf{Y}, \cA_R(\mbf{X}))\approx  G(\mbf{Y}, \cA_R^+(\mbf{X})).\]
Similarly, if $\mbf{X} \in \Sigma$ and  $\mbf{Y} \in T^-_{\kappa, R}(\mbf{X})$,  then
\[  G(\cA_R^+(\mbf{X}),\mbf{Y})\approx G(\cA_R(\mbf{X}),\mbf{Y})\approx G(\cA_R^-(\mbf{X}),\mbf{Y}).\]
\end{lemma}

\begin{lemma}[\cite{FGS}]\label{CFMS}
If $\mbf{X} \in \Sigma$ and $\mbf{Y} \in T^+_{\kappa, R}(\mbf{X})$, then
\[R^n G(\mbf{Y}, \cA_R^+(\mbf{X})) \approx \omega^{\mbf{Y}}(\Delta_{R}(\mbf{X})) \approx R^nG(\mbf{Y}, \cA_R^-(\mbf{X})).\]
Similarly, if $\mbf{X} \in \Sigma$ and $\mbf{Y} \in T^-_{\kappa, R}(\mbf{X})$, then
\[R^n G(\cA_R^-(\mbf{X}),\mbf{Y}) \approx \widetilde{\omega}^{\mbf{Y}}(\Delta_{R}(\mbf{X})) \approx R^nG(\cA_R^+(\mbf{X}),\mbf{Y}).\]
\end{lemma}

\begin{lemma}[\cite{FGS}]\label{calisdoubpre.lem}
If $\mbf{X} \in \Sigma$ and $\mbf{Y} \in T^+_{\kappa, R}(\mbf{X})$, then
\[\omega^{\mbf{Y}}(\Delta_{R}(\mbf{X})) \approx \omega^{\mbf{Y}}(\Delta_{R/2}(\mbf{X})).\]
Similarly, if $\mbf{X} \in \Sigma$ and $\mbf{Y} \in T^-_{\kappa, R}(\mbf{X})$, then
\[\widetilde{\omega}^{\mbf{Y}}(\Delta_{R}(\mbf{X})) \approx \widetilde{\omega}^{\mbf{Y}}(\Delta_{R/2}(\mbf{X})).\]
\end{lemma}

We will use the following variation of Lemma \ref{calisdoubpre.lem}.

\begin{lemma}\label{calisdoub.lem}
Let $\kappa$ be as in \eqref{kappadef} and consider  $\mbf{X} \in \Sigma$ and $r > 0$. Then,
\[\cA^+_{4r}(\mbf{X}) \in T^+_{\kappa, 2\rho}(\mbf{Z}),\]
for all $\mbf{Z} \in \Sigma$ and $\rho > 0$ such that $\Delta_{2\rho}(\mbf{Z}) \subseteq \Delta_r(\mbf{X})$. In particular,
\[\omega^{\cA^+_{4r}(\mbf{X})}(\Delta_{2\rho}(\mbf{Z})) \approx \omega^{\cA^+_{4r}(\mbf{X})} (\Delta_{\rho}(\mbf{Z})),\]
for all $\mbf{Z} \in \Sigma$ and $\rho > 0$ such that $\Delta_{2\rho}(\mbf{Z}) \subseteq \Delta_r(\mbf{X})$.
\end{lemma}
\begin{proof}
We only prove  that $\cA^+_{4r}(\mbf{X}) \in T^+_{\kappa, 2\rho}(\mbf{Z})$, as once this is done the statement
\[\omega^{\cA^+_{4r}(\mbf{X})}(\Delta_{2\rho}(\mbf{Z})) \approx \omega^{\cA^+_{4r}(\mbf{X})} (\Delta_{\rho}(\mbf{Z})),\]
follows from Lemma \ref{calisdoubpre.lem}.  Write $\mbf{X} = (X,t)$ and $\mbf{Z} = (Z, \tau)$. Then the inclusion $\Delta_{2\rho}(\mbf{Z}) \subseteq \Delta_r(\mbf{X})$ ensures that
\[\tau + (2\rho)^2 \le t + r^2, \quad \text{and} \quad 2\rho \le r.\]
Therefore it holds that
\begin{equation}\label{ttacptr.eq}
\tau \le \tau + 64\rho^2 \le t + 60\rho^2 + r^2 \le t + 16r^2.
\end{equation}
Noting that the $t$-coordinate of $\cA^+_{4r}(\mbf{X})$, call it $s$, is equal to $t + 2(4r)^2$ we have from \eqref{ttacptr.eq}
\begin{equation}\label{ttacptr2.eq}
(s - \tau) = t + 32r^2 - \tau \ge 16 r^2 \ge 16\rho^2.
\end{equation}
Then the time coordinate of $\cA^+_{4r}(\mbf{X})$ satisfies the condition in the definition of $T^+_{\kappa, 2\rho}(\mbf{Z})$. Since $\mbf{Z} \subseteq \Delta_r(\mbf{X})$ writing $X = (x_0,x)$, it holds that
\[|(x_0 + 2 M_0(4r),x) - Z | \le |(x_0,x) - Z | +  2 M_0(4r) \lesssim 10M_0\sqrt{n}(4r).\]
Thus, from \eqref{ttacptr2.eq} we conclude
\begin{align*}
|(x_0 + 2 M_0(4r),x) - Z| \lesssim 2 M_0(4r) &\lesssim 10M_0\sqrt{n}(4r)\\
&\lesssim 2 M_0(4r) \lesssim 10M_0\sqrt{n}(4r) (s - \tau)^{1/2}.
\end{align*}
As $(x_0 + 2 M_0(4r),x)$ is the spatial coordinate of $\cA^+_{4r}(\mbf{X})$, the previous inequality shows that for $\kappa=  40M_0\sqrt{n}$ we have $\cA^+_{4r}(\mbf{X}) \in T^+_{\kappa, 2\rho}(\mbf{Z})$.
\end{proof}

We will also need the following quantitative non-degeneracy result for $\partial_{y_0}G$. The following lemma is essentially Lemma 2.12 in \cite{N2006}.  On the other hand, the proof given in \cite{N2006}, based on the arguments in \cite{ACS}, \cite{CS}, is not 
complete, so for the reader's convenience, we give the complete argument 
in an appendix to this paper.  

\begin{lemma}\label{lemma:CS-3}
Let $\mbf{X}\in\Sigma$ and $R>0$. Then there exists $\eta\in (0,1/2)$, depending only on $n$ and $\|\psi\|_{\Lip(1,1/2)}$, such that the following holds. If $\mbf{Y}^\ast \in T^+_{\kappa, R}(\mbf{X})$, and $u(\mbf{Y}):=G(\mbf{Y}^\ast,\mbf{Y})$, then
\begin{equation}\label{eq3.25}
	|\nabla_Y u(\mbf{Y})|\approx \langle\nabla_Y u(\mbf{Y}),e_0\rangle\approx\frac{u(\mbf{Y})}{\delta(\mbf{Y})},
	\end{equation}
for every $\mbf{Y}\in I_{R/4}(\mbf{X})\cap\Omega$ with $\delta(\mbf{Y})<\eta\,R$. If $\mbf{Y}^\ast \in T^-_{\kappa, R}(\mbf{X})$, and $u(\mbf{Y}):=G(\mbf{Y},\mbf{Y}^\ast)$, then
\begin{equation}\label{eq3.25a}
	|\nabla_Y u(\mbf{Y})|\approx \langle\nabla_Y u(\mbf{Y}),e_0\rangle\approx\frac{u(\mbf{Y})}{\delta(\mbf{Y})},
	\end{equation}
for every $\mbf{Y}\in I_{R/4}(\mbf{X})\cap\Omega$ with $\delta(\mbf{Y})<\eta\,R$.
\end{lemma}

Next we give our definition of the $A_\infty$ property for caloric measure.

\begin{definition}\label{defainfty}  We say $\omega$ is in $A_\infty$, or that its density
$d\omega/d\sigma$ is a parabolic $A_\infty$ weight, if there
exist $C_\ast$ and $q  > 1$ such that the following holds. If $\mbf{X} \in \Sigma$, $r > 0$, and $\mbf{Y} \in T^+_{\kappa, 2r}(\mbf{X})$, then $\omega^{\mbf{Y}} \ll \sigma$ on $\Delta_{2r}(\mbf{X})$ and $k_{\mbf{Y}} := {\d\omega^{\mbf{Y}}}/{\d\sigma}$ satisfies the reverse-H\"older inequality
\begin{equation}\label{rhq}
\iint_{\Delta_{r}(\mbf{X})} k_\mbf{Y}^q \, \d\sigma \le C_\ast \sigma(\Delta_{r}(\mbf{X}))^{1-q} (\omega^{\mbf{Y}}(\Delta_{r}(\mbf{X})))^q.
\end{equation}
\end{definition}

\begin{remark}\label{whnweusrhq.rmk}
Consider $\mbf{X}^0 \in \Sigma$ and $R_\ast > 0$, Using  Lemma \ref{calisdoub.lem} we have that
\[\cA^+_{4R_\ast}\left(\mbf{X}^0\right) \in T^+_{\kappa, 2\rho}(\mbf{Z}),\]
for all $\mbf{Z} \in \Sigma$ and $\rho > 0$ such that $\Delta_{2\rho}(\mbf{Z}) \subseteq \Delta_{R_\ast}(\mbf{X}^0)$. As a consequence,  \eqref{rhq} is valid with $\mbf{Y} = \cA^+_{4R_\ast}(\mbf{X}^0)$, i.e.,
\[\iint_{\Delta_{r}(\mbf{Z})} k_{\cA^+_{4R_\ast}(\mbf{X}^0)} ^q \, \d\sigma \le C_\ast \sigma(\Delta_{r}(\mbf{Z}))^{1-q} (\omega^{\cA^+_{4R_\ast}(\mbf{X}^0)}(\Delta_{r}(\mbf{Z})))^q,\]
for all $\mbf{Z} \in \Delta_{R_*/2}(\mbf{X}^0)$ and $r < R_*/2$.
\end{remark}

\section{Preliminary arguments for the proof of Theorem \ref{mainthrm.thrm} }\label{mainthrmsetup.sect}

In this section we develop a number of preliminary arguments for the proof of Theorem \ref{mainthrm.thrm}. Throughout the section we will assume the hypotheses of Theorem \ref{mainthrm.thrm} and the constants appearing are allowed to depend (implicitly) on $n$, $\|\psi\|_{\Lip(1,1/2)}$ and the constants $C_\ast$ and $q$ in Definition \ref{defainfty}, i.e., on the what we have coined the structural constants. We will sometimes stress this dependence, but otherwise it can be assumed by the reader. To prove Theorem \ref{mainthrm.thrm}, our strategy is to show that there exists a constant $N_\star$, depending only on the structural constants,  such that for each parabolic cube $Q_R\subset\ree$ we have
\begin{equation}\label{JS-lemma:1}
	\inf_C \big|\big\{
	\mbf{y}\in Q_R: |\mathcal{D}_t\psi(\mbf{y})-C|>N_\star
	\big\}\big|
	\le
	(1/4)\,|Q_R|.
\end{equation}
Indeed, if this is true then the parabolic version of the John-Str{\"o}mberg lemma implies 
\[
\|\mathcal{D}_t \psi\|_{\BMO(\re^n)}\le CN_\star\,.
\] 
Thus, we will be focused on establishing \eqref{JS-lemma:1} with $N_\star$ depending only on the structural constants.

To this end, we fix $\mbf{x}^0=(x^0,t^0)\in \re^n$ and we let 
$Q_R:=Q_R(\mbf{x}^0)$. We define
\[
\mbf{X}^0:=(x_0^0, \mbf{x}^0):=(\psi(\mbf{x}^0), \mbf{x}^0)\in \Sigma\,,
\] 
and set $I_R:=I_R(\mbf{X}^0)$, 
$\Delta:=\Delta_{R}(\mbf{X}^0)=I_{R}(\mbf{X}^0)\cap\Sigma$. 
With $\mbf{X}^0$ fixed,  we will often simply write $\Delta_{cR}$
instead of $\Delta_{cR}(\mbf{X}^0)$.  We also introduce 
\begin{equation}\label{wtMdef.eq}
\Delta_\star:=M_1\Delta\,,\qquad 
\text{where }\, M_1:=32000M_0^2\eta^{-1}n,
\end{equation}
and where $\eta$ is the constant appearing in Lemma~\ref{lemma:CS-3}. Using this notation, we set 
$\mbf{X}_\star:=\cA^+_{4M_1R}(\mbf{X}^0)$, a time forward reference point defined relative to the surface box $4\Delta_\star$. We let 
\begin{equation}\label{hmnorm}
\omega(\cdot) :=\sigma(\Delta_\star)\,\omega^{\mbf{X}_\star}(\cdot)
\end{equation}
denote  normalized caloric measure, and we let
\begin{equation}\label{greennorm}
u(\cdot) =\sigma(\Delta_\star)\,G(\mbf{X}_\star,\cdot).
\end{equation}
This normalized Green function, with pole at $\mbf{X}_\star$, is a solution to the adjoint heat equation outside of its pole. On any set where $\omega\ll \sigma$,  we let $k=\d\omega/\d\sigma$. Applying Lemma \ref{HCatBdry.lem} to the function
$\mbf{X}\to 1-\omega^{\mbf{X}}(\tfrac{1}{4}\Delta_\star)$, which is a non-negative solution to the heat equation which vanishes on $\tfrac{1}{4}\Delta_\star$, and subsequently using the Harnack inequality, we deduce that
\begin{equation}\label{Bourgain:aver}
	1
	\approx
	\omega^{\mbf{X}_\star}(\tfrac{1}{4}\Delta_\star) = \omega^{\mbf{X}_\star}(\Delta_{M_1R/4}(\mbf{X}^0))
	=
	\frac{\omega(\tfrac{1}{4}\Delta_\star)}{\sigma(\Delta_\star)}
	\lesssim
	\frac{\omega(\Delta_\star)}{\sigma(\Delta_\star)}
	\approx
	\omega^{\mbf{X}_\star}(\Delta_\star)
	\le
	1.
\end{equation}
We will refer  to \eqref{Bourgain:aver} frequently.  For future reference, we note that, in fact, we have more generally that
\begin{equation}\label{Bourgain2}
\frac{\omega(\Delta_{NR}(\mbf{X}^0))}{\sigma(\Delta_{NR}(\mbf{X}^0))} \, 
\approx \, 1\,,
 \qquad  1\leq N \leq M_1 \,. 
\end{equation}
This fact is standard for Lip(1,1/2) graph domains, but see, e.g., 
\cite[Lemma 2.2]{GH} for a more general result (which yields \eqref{Bourgain2}
in our setting by the use of Harnack's inequality).

\subsection{Constructing the base for a sawtooth from  the $A_\infty$ assumption} That  $\omega$ is in $A_\infty$ (see Definition \ref{defainfty}) implies that there exists 
$p > 1$, depending only on $n$, $M_0$,
and the constants $C_\ast$ and $q$ in Definition \ref{defainfty}, such that
\begin{equation}\label{Apatsclr.eq}
\left(\bariint_{\Delta_{M_1R/4}(\mbf{X}^0)} k  \, \d\sigma \right)\left(\bariint_{\Delta_{M_1R/4}(\mbf{X}^0)} k^{1-p'} \, \d\sigma \right)^{p-1} \lesssim_p 1,
\end{equation}
which further implies that
\begin{equation}\label{quantctyrev.eq}
\frac{\sigma(E)}{\sigma(\Delta_{M_1R/4}(\mbf{X}^0))} \lesssim_p\left(\frac{\hm(E)}{\hm(\Delta_{M_1R/4}(\mbf{X}^0))}\right)^{1/p},
\end{equation}
whenever $E \subseteq \Delta_{M_1R/4}(\mbf{X}^0)$. Indeed, given
\eqref{Apatsclr.eq},
\begin{align*}
		\Big(\frac{\sigma(E)}{\sigma(\Delta_{M_1R/4}(\mbf{X}^0))}\Big)^p
		&=
		\Big(\bariint_{\Delta_{M_1R/4}(\mbf{X}^0)} \mathbf{1}_{E}\,k^{\frac1p}\,k^{-\frac1p}\,\d\sigma\Big)^p
		\\ &\le
		\Big(\bariint_{\Delta_{M_1R/4}(\mbf{X}^0)} \mathbf{1}_{E}\,k\,\,\d\sigma\Big)
		\Big(\bariint_{\Delta_{M_1R/4}(\mbf{X}^0)} k^{1-p'}\d\sigma\Big)^{p-1}
		\\
		&\lesssim_p
		\Big(\frac{\hm(E)}{\sigma(\Delta_{M_1R/4}(\mbf{X}^0))}\Big)\,\Big(\bariint_{\Delta_{M_1R/4}(\mbf{X}^0)} k\,\d\sigma\Big)^{-1}
		=
		\frac{\hm(E)}{\hm(\Delta_{M_1R/4}(\mbf{X}^0))}.
	\end{align*}

\begin{lemma}\label{CreateFstarlem.lem} Given $\epsilon > 0$, there exists a constant $M\geq 1$, depending only on the structural constants and $\epsilon$, and  a closed set $F_\star$, such that
\begin{equation}\label{Fstarprops.eq}
F_\star \subseteq \widetilde{F}_\star :=
\left\{\mbf{X} \in \Delta_{50R}(\mbf{X}^0): \frac{1}{M} \le \frac{\hm(\Delta_r(\mbf{X}))}{\sigma(\Delta_r(\mbf{X}))} \le M  , \quad \forall r \in (0, M_1R/8) \right\},
\end{equation}
and such that
\begin{equation}\label{Fstarampleeqn.eq}
\sigma( \Delta_{50R}(\mbf{X}^0)\setminus F_\star) \le \epsilon \sigma( \Delta_{50R}(\mbf{X}^0)).
\end{equation}
\end{lemma}
\begin{proof}
Fix $\epsilon> 0$, and let $M$ be a degree of freedom eventually
to be chosen depending only on the structural constants and $\epsilon$. We introduce the truncated maximal operators
	\begin{align*}
	\mathcal{M}_\sigma \hm(\mbf{X})
	&:=
	\sup_{0<r< M_1R/8} \frac{\hm(\Delta_r(\mbf{X}))}{\sigma(\Delta_r(\mbf{X}))}, \qquad \mbf{X}\in \Delta_{50R}(\mbf{X}^0),
	\\
	\mathcal{M}_\hm \sigma(\mbf{X})
	&:=
\sup_{0<r< M_1R/8} 
\frac{\sigma(\Delta_r(\mbf{X}))}{\hm(\Delta_r(\mbf{X}))},
	\qquad \mbf{X}\in \Delta_{50R}(\mbf{X}^0).
	\end{align*}
Note that if $\mbf{X}\in\Delta_{50R}(\mbf{X}^0)$ 
	and $0<r<M_1R/8$, then
\[
\Delta_r(\mbf{X})\subset \Delta_{M_1R/8}(\mbf{X}) 
\subseteq 
\Delta_{M_1R/4}(\mbf{X}^0)\,.\]
Since $\sigma$ and $\omega$ are both doubling measures, by the weak-type (1,1) bound for the maximal function (or directly, by a standard covering argument), we have
\begin{equation}\label{weaktype:M:both}
	\begin{split}
		\sigma\big(\big\{
		\mbf{X}\in\Delta_{50R}(\mbf{X}^0): 
\mathcal{M}_\sigma \hm(\mbf{X})> N 
		\big\}\big)
		&\le
		\frac{C_n}{N}\, \hm(\Delta_{M_1R/4}(\mbf{X}^0)),
		\\
		\hm\big(\big\{
		\mbf{X}\in\Delta_{50R}(\mbf{X}^0): \mathcal{M}_\hm \sigma(\mbf{X})>N
		\big\}\big)
		&\le
		\frac{C_n}{N}\, \sigma(\Delta_{M_1R/4}(\mbf{X}^0)),
	\end{split}
	\end{equation}
for all $N>0$, and where $C_n$ depends only on dimension (and in the second inequality also on the doubling constant for $\omega$, which in turn depends only on $n$ and $M_0$). 
We introduce
 \begin{equation*}
 A^1_\star :=
\{\mbf{X}\in\Delta_{50R}(\mbf{X}^0): 
	\mathcal{M}_\sigma \hm(\mbf{X})>M \}\,,\qquad
		 A^2_\star:= \{
		\mbf{X}\in\Delta_{50R}(\mbf{X}^0): \mathcal{M}_\hm \sigma(\mbf{X})>M\}\,.
		\end{equation*}
 Using \eqref{Bourgain:aver}, 
 we have $\hm(\Delta_{M_1R/4}(\mbf{X}^0)) \approx 
 \sigma(\Delta_{M_1R/4}(\mbf{X}^0))$, so from the first inequality in \eqref{weaktype:M:both} we deduce
 \begin{equation}\label{f1bdainf.eq}
 \sigma(A^1_\star) \lesssim \frac{1}{M}\sigma(\Delta_{M_1R/4}(\mbf{X}^0)) \approx \frac{1}{M} \sigma(\Delta_{50R}(\mbf{X}^0))
 \end{equation}
 (with harmless implicit dependence on the fixed constant $M_1$).
Similarly
 \begin{equation}\label{f2bdainf.eq}
 \frac{\hm(A^2_\star)}{\hm(\Delta_{M_1R/4}(\mbf{X}^0))} \lesssim \frac{1}{M},
 \end{equation}
and using \eqref{quantctyrev.eq}  we obtain 
 \begin{equation}\label{Astar2bound.eq}
\frac{\sigma(A^2_\star)}{\sigma(\Delta_{50R}(\mbf{X}^0))} \approx \frac{\sigma(A^2_\star)}{\sigma(\Delta_{M_1R/4}(\mbf{X}^0))} \lesssim (1/M)^\theta,
\end{equation}
where $\theta = 1/p < 1$. 
Observe that
\[
\Delta_{50R}(\mbf{X}^0))\setminus \widetilde{F}_\star \subset
A^1_\star \cup A^2_\star\,,.
\]
Choosing $M$ 
such that $(1/M)^\theta \ll \epsilon$, we 
deduce from \eqref{f1bdainf.eq} 
and \eqref{Astar2bound.eq} 
that
\[\sigma( \Delta_{50R}(\mbf{X}^0) \setminus \widetilde{F}_\star) 
\,\le \,\frac{\epsilon}{2} \,\sigma(\Delta_{50R}(\mbf{X}^0))\,,\]
so using the regularity of $\sigma$ we can find a closed set $F_\star \subseteq \widetilde{F}_\star$ with
\[\sigma( \Delta_{50R}(\mbf{X}^0) \setminus F_\star) \le \epsilon \sigma(\Delta_{50R}(\mbf{X}^0)).\]
\end{proof}

We define the projection operator $\Pi$ from $(n+1)$-dimensional space-time to $n$-dimensional space-time according to
\begin{align}\label{proja}\Pi(\mbf{Y}):=(0,y,s),\ \mbf{Y}=(y_0,y,s)\in\ree.
\end{align}
Observe that
\begin{align}\label{proja+}F:=\Pi(F_\star)\subset Q_{50R}.
\end{align}
Thus, if we let $|\cdot|$ denote Lebesgue measure on $n$-dimensional space-time,
\begin{equation}  \label{projsizeeqbasic.eq}
|Q_R\setminus F| \le |Q_{50R}\setminus F|
\le
|\Pi(\Delta_{50R}\setminus F_\star)|
\le
\sigma(\Delta_{50R}\setminus F_\star)
\le
\epsilon
\sigma(\Delta_{50R})
\lesssim \epsilon\,|Q_R| \approx \epsilon R^d,
\end{equation}
where the implicit constant depends only on $n$ and $\|\psi\|_{\Lip(1,1/2)}$, and we recall that $d:=n+1$. In our use of Lemma \ref{CreateFstarlem.lem}, $\epsilon$ will in the end be 
fixed, depending only on the structural constants. In particular, we will 
always insist that $\epsilon$ is small enough that
\begin{equation}\label{projsizeeq.eq}
|Q_R\setminus F| \leq |Q_{50R}\setminus F| \le (1/8) |Q_R|\,.
\end{equation}

\subsection{Level sets of the normalized Green function} 
Recall that 
$u$ is the normalized Green function (see \eqref{greennorm}).
For future reference, we collect several important observations in the following. 
\begin{remark}\label{strongharnack4gf.rmk}
By Lemma \ref{strongharnack4gf.lem}, the Green function satisfies a strong Harnack inequality, suitably interpreted. Consequently, the strong Harnack inequality holds for $u$  in the sense that
\begin{equation}\label{stharnlocu.eq}
u(\cA_r^+(\mbf{Z})) \approx u(\cA_r^-(\mbf{Z})),
\end{equation}
for all $Z \in \Delta_{M_1R/2}(\mbf{X}^0)$ and $r < M_1R/4$. Indeed, from Lemma \ref{strongharnack4gf.lem} we have
\begin{equation}\label{stharnlocgfeq.eq}
G(\mbf{X}_\star, \cA_r^+(\mbf{Z})) \approx G(\mbf{X}_\star, \cA_r^-(\mbf{Z}))
\end{equation}
provided $\mbf{X}_\star=\cA^+_{4M_1R}(\mbf{X}^0) \in T^+_{\kappa, r}(\mbf{Z})$. The latter fact can be seen from Lemma \ref{calisdoub.lem}, upon noting that $\Delta_r(\mbf{Z}) \subset \Delta_{M_1R}(\mbf{X}^0)$ so that $r$ can play the role of $2\rho$ in Lemma \ref{calisdoub.lem}. 
Clearly, \eqref{stharnlocgfeq.eq} implies \eqref{stharnlocu.eq}. 

We may therefore apply Lemma \ref{lemma:CS-3} to obtain that
\[\partial_{y_0}u(\mbf{Y})\approx \frac{u(\mbf{Y})}{\delta(\mbf{Y})},
\quad \text{for every $\mbf{Y}=(y_0,\mbf{y})\in I_{M_1R/8}(\mbf{X}^0) \cap\Omega$}, \ \delta(\mbf{Y}) \le \eta M_1R/2.\]
Recalling the definition of $M_1$ (see \eqref{wtMdef.eq}), we have $M_1R/8 \ge 4000R$ and $\eta M_1R/2 = 16000M_0^2Rn$. Hence
\begin{equation}\label{eqn:est-der-G}
	\partial_{y_0}u(\mbf{Y})\approx \frac{u(\mbf{Y})}{\delta(\mbf{Y})},
\qquad \text{for every $\mbf{Y}=(y_0,\mbf{y})\in \overline{I_{400R}(\mbf{X}^0)}\cap\Omega$}.
\end{equation}
Here we have used \eqref{crdIrest.eq} to conclude that 
$\mbf{Y} \in \overline{I_{400R}(\mbf{X}^0)}$ implies 
\[\delta(\mbf{Y)} \le \|\mbf{Y} - \mbf{X}^0\| \le 2000M_0R\sqrt{n} < 16000M_0^2Rn = \eta M_1R/2\,.
\]
\end{remark}

Next we use the estimate \eqref{eqn:est-der-G}, and the implicit function theorem, to show, for $r$ small, that the level sets $\{u = r\}$ are locally given by the graph of a Lip(1,1/2) function $\psi_r(\mbf{x})$. Given $\mbf{X}_\star=\cA^+_{4M_1R}(\mbf{X}^0)$, the pole of the Green function, we introduce, for $r>0$ small and fixed,
the level set
\[
\Sigma_r(\mbf{X}_\star):=\big\{\mbf{Y} \in \Omega:\, u(\mbf{Y})= r\big\}.
\]

\begin{lemma}\label{lemma:level-sets} For some $\Lambda_0\gg 1$ depending only on the structural constants, if $0<r<R/\Lambda_0$, then there is a 
function $\psi_r\in C^1(Q_{100R}(\mbf{x}^0))$ such that
\[
\Sigma_r(\mbf{X}_\star)\cap I_{100R}(\mbf{X}^0)
=
\big\{\big(\psi_r(\mbf{y}),\mbf{y}\big): \mbf{y}\in Q_{100R}(\mbf{x}^0)\big\}.
\]
Moreover, for every $\mbf{y}\in  Q_{100R}(\mbf{x}^0)$ and $0<r\le R/\Lambda_0$, the mapping
$r\mapsto \psi_r(\mbf{y})$ is strictly increasing and differentiable,
\begin{equation}\label{abdlevelsets.eq}
\psi(\mbf{y}) < \psi_r(\mbf{y})<x_0^0+300M_0 \sqrt{n}R,
\end{equation}
 and
 \begin{equation*} 
 \lim_{r\to 0^+} \psi_r(\mbf{y})=\psi(\mbf{y}).
 \end{equation*}
\end{lemma}

\begin{proof} We start by observing that by 
\eqref{Bourgain2} with $N=100$,
Lemma \ref{CFMS}, and Harnack's inequality,
$u(\mbf{Y})\approx R$ for every $\mbf{Y}=(y_0,\mbf{y})\in \overline{I_{100\,R}(\mbf{X}^0)}\cap\Omega$ with $y_0=x_0^0+300M_0\sqrt{n}R$ (these are the points at the ``top'' of the box). 
This means that for every $\mbf{y}\in Q_{100\,R}(\mbf{x}^0)$, we have $u(\psi(\mbf{y}),\mbf{y})=0$ and  $u(x_0^0+300M_0 \sqrt{n}R,\mbf{y})\approx R$. In particular, there exists $\Lambda_0$ such that
\[
u(x_0^0+300M_0\sqrt{n}R,\mbf{y}) > R/\Lambda_0 \quad \forall \mbf{y}\in Q_{100R}(\mbf{x}^0),
\]
and we fix $\Lambda_0$ accordingly.
Moreover, by \eqref{eqn:est-der-G} we have that
$\partial_{y_0} u(y_0,\mbf{y})>0$ for every $y_0$ satisfying 
$\psi(\mbf{y})<y_0< x_0^0+300M_0\sqrt{n}R$, hence  $u(y_0,\mbf{y})$, viewed as a function of $y_0$, is strictly increasing in the interval 
$\psi(\mbf{y})<y_0< x_0^0+300M_0\sqrt{n}R$. By the intermediate value theorem, for every $0<r\le R/\Lambda_0$, 
one can find a unique value $\psi_r(\mbf{y})$ such that $\psi(\mbf{y})<\psi_r(\mbf{y})<x_0^0+ 300M_0 \sqrt{n}R$ 
(depending implicitly on $X_\star$) so that $u(\psi_r(\mbf{y}),\mbf{y})=r$. Furthermore, if we invoke the implicit function theorem and the local smoothness of adjoint caloric functions we conclude that $\psi_r\in C^\infty(Q_{100\,R}(\mbf{x}^0))$. Furthermore, $\psi_r$ is (infinitely) differentiable as a function of the variable $r$.

Fix $\mbf{y}\in  Q_{100\,R}(\mbf{x}^0)$. Note that if $r<s < \Lambda_0\,R$, then $u(\psi_r(\mbf{y}),\mbf{y})=r<s=u(\psi_s(\mbf{y}),\mbf{y})$, and since $u(\cdot,\mbf{y})$ is strictly increasing in the interval $(\psi(\mbf{y}), x_0^0+300\,M_0\,R\, \sqrt{n})$ it follows that $\psi_r(\mbf{y})<\psi_s(\mbf{y})$. Next, using that for fixed $\mbf{y}\in Q_{100\,R}(\mbf{x}^0)$ we have $\psi_r(\mbf{y}) > \psi(\mbf{y})$, and that $\psi_r(\mbf{y})$  is increasing in $r \in (0, \Lambda_0\,R)$, it follows that $\lim_{r\to 0^+} \psi_r(\mbf{y})$ exists and we call the limit $\psi_0(\mbf{y})$. Note that $\psi_0(\mbf{y}) \ge \psi(\mbf{y})$. By the continuity of $u$ up to the boundary
\[
0
=
\lim_{r\to 0^+} r
=
\lim_{r\to 0^+} u(\psi_r(\mbf{y}),\mbf{y})
=
u(\psi_0(\mbf{y}),\mbf{y}),
\]
and this implies that $(\psi_0(\mbf{y}),\mbf{y})\in\Sigma$, i.e. $\psi_0(\mbf{y})=\psi(\mbf{y})$. This completes the proof.
\end{proof}

\subsection{A regularized distance function}

Recall the projection operator $\Pi$ introduced in \eqref{proja}. Recall also the small parameter $\epsilon >0$ (see Lemma \ref{CreateFstarlem.lem}), 
and the set $F$ introduced (see 
\eqref{proja+}-\eqref{projsizeeq.eq}). By the triangle inequality the 
function $\re^n\ni \mbf{x}\mapsto \dist(\mbf{x},F)$ is $\Lip(1,1/2)$ with
 constant at most $1$. Thus, by \cite[Lemma~3.24]{BHHLN-CME} there 
 exists a non-negative $\Lip(1,1/2)$ function $h:\re^n\to \overline{\re^+}$ 
 (a regularized distance function) with properties as stated in the 
 following lemma.
\begin{lemma}\label{regdistlem.lem} The function $h$ satisfies
\begin{align}\label{regular-dist}
h(\mbf{x})\approx \dist(\mbf{x},F),
\qquad
\|h\|_{\Lip(1,1/2)}\lesssim 1, \qquad \|\mathcal{D}_t h\|_{\BMO(\re)}
\lesssim
1,
\end{align}
and
\begin{equation}\label{regular-dist2}
\dist(\mbf{x},F)^{2\,k-1}\,|\partial_t^k h(\mbf{x})|+\dist(\mbf{x},F)^{k-1}\,|\nabla_x^k h(\mbf{x})|\lesssim_k 1,
\qquad \forall\,\mbf{x}\notin F,\, k\in\NN,
\end{equation}
where the implicit constants depend on dimension (and on $k$ for the last estimate). By construction $h\equiv 0$ in $F$.
\end{lemma}

\begin{remark}\label{remdist}
Note that \eqref{regular-dist} states that $h$ is a non-negative regular $\Lip(1,1/2)$ function with constants of the order 1. In particular, 
the surface $$\{(h(x,t),x,t):\ (x,t)\in\mathbb R^n\}$$ is parabolic uniformly rectifiable, see Remark \ref{PUR}.
\end{remark}

In the sequel, we will often employ the notation
\begin{equation}\label{asemicoldef.eq}
\psi(r;\cdot\,):=\psi_r(\cdot)\,, \qquad \psi(0;\cdot\,):=\psi(\cdot)\,,
\end{equation}
since we will often plug the function $h(\cdot)$ in place of $r$ and hence this notation is more convenient.

\begin{lemma}\label{lemma:a-h-small} 
Suppose $\Lambda \ge  2\Lambda_0$, where $\Lambda_0 \gg 1$ is as in Lemma \ref{lemma:level-sets}. If the parameter $\epsilon$
in Lemma \ref{CreateFstarlem.lem} is chosen small enough, depending only on the structural constants, we then have
\begin{equation}\label{h:small}
\sup_{\mbf{y}\in Q_{50R}(\mbf{x}^0)} h(\mbf{y}) \le  R/(80\Lambda).
\end{equation} Moreover, if $\Lambda$ is sufficiently large, then  for every $\mbf{y}\in Q_{50R}(\mbf{x}^0)$ we have
\begin{equation}\label{ahest}
0\le \psi(h(\mbf{y});\mbf{y})-\psi(\mbf{y})\approx h(\mbf{y}) \approx \dist(\mbf{y},F),
\end{equation}
where the implicit constants depend on $\Lambda$, $\epsilon$, and the structural constants,
and
\begin{equation}\label{ahforom1lem}
\psi(h(\mbf{y});\mbf{y}) < x_0^0+300M_0\sqrt{n} R .
\end{equation}
\end{lemma}

\begin{proof}  Recall that $d=n+1$ is the homogeneous dimension of 
parabolic $\rn$.
Since $h \equiv 0$ on $F$, and  
$F\subset Q_{50R}(\mbf{x}^0)$, we deduce from the definition of $h$ and \eqref{projsizeeqbasic.eq} that
\[
h(\mbf{x}) \approx \dist(\mbf{x},F) \lesssim  \epsilon^{1/d} R\,, \quad \forall \,  \mbf{x} \in Q_{50R}(\mbf{x}^0)\,,
\]
which yields \eqref{h:small} by choice of $\epsilon$ small enough, 
depending on $\Lambda$.
We will henceforth assume that  $\epsilon$ is at least small enough to ensure that \eqref{h:small} holds with $\Lambda = 2\Lambda_0$, but with the freedom to take $\Lambda$ larger (thus, $\epsilon$ smaller, 
depending on $\Lambda$, but with each ultimately fixed).


We next prove \eqref{ahest} and \eqref{ahforom1lem}. 
Let $\mbf{y}\in Q_{50R}(\mbf{x}^0)$. Observe that the case $\mbf{y}\in F$ is trivial since $ h(\mbf{y}) \approx \dist(\mbf{y},F)=0$ and $ \psi(0;\mbf{y})=\psi(\mbf{y})$. We may then assume that $\mbf{y}\not\in F$, hence $h(\mbf{y})>0$. Pick $\widehat{\mbf{y}}\in F$ so that $\|\widehat{\mbf{y}}-\mbf{y}\|=\dist(\mbf{y}, F)\approx_n h(\mbf{y})$.
Set $\mbf{Y}:=(\psi(\mbf{y}), \mbf{y})\in\Sigma$, 
$\widehat{\mbf{Y}}:=(\psi(\widehat{\mbf{y}}), \widehat{\mbf{y}})\in F_\star $, 
and $\widetilde{\mbf{Y}}:=(\psi(h(\mbf{y});\mbf{y}), \mbf{y})\in\Omega$. 
By Lemma \ref{CreateFstarlem.lem},
	\begin{equation}\label{342rt3qf}
\frac{\omega(\Delta_\rho(\widehat{\mbf{Y}}))}{\sigma(\Delta_{\rho}
(\widehat{\mbf{Y}}))}
	\approx
	1\,, \qquad 0<\rho<M_1R/8\,,
	\end{equation}
since $\widehat{\mbf{Y}} \in F_\star$.
	Here the implicit constant depends on the constant $M$ of Lemma \ref{CreateFstarlem.lem}, and hence on $\epsilon$. 

In particular, note that 
$\Delta_{h(\mbf{y})}(\mbf{Y}) \subset \Delta_{Ch(\mbf{y})}(\widehat{\mbf{Y}}) 
\subset \Delta_{C^2h(\mbf{y})}(\mbf{Y})$ for some harmless constant $C>1$ 
depending on dimension. Thus, if $\Lambda$ is sufficiently large we can 
use \eqref{h:small} to obtain $h(\mbf{y}) \le C^{-2}R$, and then 
use \eqref{Fstarprops.eq}  and the local doubling of $\hm$ 
(see Lemma \ref{calisdoub.lem}) to conclude that
	\begin{equation}\label{average1}
 \frac{\omega(\Delta_{h(\mbf{y})}(\mbf{Y}))}{\sigma(\Delta_{h(\mbf{y})}
 (\mbf{Y}))} \approx
 \frac{\omega(\Delta_{Ch(\mbf{y})}(\widehat{\mbf{Y}}))}
 {\sigma(\Delta_{Ch(\mbf{y})}(\widehat{\mbf{Y}}))}
	\approx_M
	1\,.
	\end{equation}
Let $N$ be a large constant to be chosen momentarily.
	If $\mbf{Z}\in I_{h(\mbf{y})/N}(\mbf{Y})\cap\Omega$, we can then invoke Lemma~\ref{carlest.lem},  Lemma \ref{CFMS}, and \eqref{average1}, 
	to deduce that
	\[
	u(\mbf{Z})
	\lesssim
	\Big(\frac{\delta(\mbf{Z})}{h(\mbf{y})}\Big)^\alpha
	u(\cA^+_{h(\mbf{y})}(\mbf{Y}))
	\lesssim
	N^{-\alpha}\,h(\mbf{y})
	\frac{\omega(\Delta_{h(\mbf{y})}(\mbf{Y}))}{\sigma(\Delta_{h(\mbf{y})}(\mbf{Y}))}
	\approx
	N^{-\alpha}\,h(\mbf{y})
	<\frac12\,
	h(\mbf{y}),
	\]  
for  $N=N(n,M,M_0)$ large enough.
Consequently, using that 
$u(\widetilde{\mbf{Y}})= u(\psi(h(\mbf{y});\mbf{y}) \mbf{y})=h(\mbf{y})$ by construction (see Lemma~\ref{lemma:level-sets}), we conclude that 
$\widetilde{\mbf{Y}}\notin  I_{h(\mbf{y})/N}(\mbf{Y})$. Hence, 
\[
\psi(h(\mbf{y});\mbf{y})> \psi(\mbf{y})+3M_0\sqrt{n} h(\mbf{y})/N,
\]
that is, since $M$ depends on $\epsilon$,
\begin{equation}\label{hlowerbound}
\delta(\widetilde{\mbf{Y}})
\approx_{n,M_0} \psi(h(\mbf{y});\mbf{y})-\psi(\mbf{y}) \geq c_1
h(\mbf{y})>0\,,\qquad c_1 = c(n,\epsilon,M_0)\,.
\end{equation}

	To obtain the converse inequality we observe that what we have just obtained implies that
	\begin{multline}\label{distytildeyhat}
	\|\widetilde{\mbf{Y}}-\widehat{\mbf{Y}}\|
	\le
	|\psi(h(\mbf{y});\mbf{y})-\psi(\mbf{y})|+|\psi(\mbf{y})-\psi(\widehat{\mbf{y})}|+\|\mbf{y}-\widehat{\mbf{y}}\|
	\lesssim
	\delta(\widetilde{\mbf{Y}})+ (1+M_0)\|\mbf{y}-\widehat{\mbf{y}}\|
	\\
	=
	\delta(\widetilde{\mbf{Y}})+ (1+M_0)\dist(\mbf{y},F)
	\approx_n
	\delta(\widetilde{\mbf{Y}})+ (1+M_0)h(\mbf{y})
	\lesssim_{n,M_0}
	\delta(\widetilde{\mbf{Y}}).
	\end{multline}
	Recall that, as before,
$x_0^0=\psi(\mbf{x}^0)$.  Note that by Lemma \ref{lemma:level-sets} 
	and \eqref{h:small} we have
\[\psi(\mbf{y})< \psi(h(\mbf{y});\mbf{y}) < x_0^0+300M_0\sqrt{n}R ,\] 
so that \eqref{ahforom1lem} holds.  In particular, 
$\widetilde{\mbf{Y}}= 
(\psi(h(\mbf{y});\mbf{y}), \mbf{y}) \in I_{100R}(\mbf{X}^0)$, 
and therefore
\begin{equation}\label{tydist}
	0< 
\delta(\widetilde{\mbf{Y}})
= \delta(\psi(h(\mbf{y});\mbf{y}),\mbf{y})
	\le \|\widetilde{\mbf{Y}} - \mbf{X}^0\| \le 500M_0\sqrt{n}R
	\ll M_1R/8\,,
\end{equation}
by \eqref{crdIrest.eq}, with a possibly smaller choice of $\eta$, depending on $n$ 
and $M_0$ (see \eqref{wtMdef.eq}).  
We now claim that
	\[
	\frac{h(\mbf{y})}{\psi(h(\mbf{y});\mbf{y})-\psi(\mbf{y})} \approx
	\frac{h(\mbf{y})}{\delta(\psi(h(\mbf{y});\mbf{y}),\mbf{y})}
	=
\frac{u(\psi(h(\mbf{y});\mbf{y}),\mbf{y})}{\delta(\psi(h(\mbf{y});\mbf{y}),\mbf{y})}
	\approx
	\frac{\omega\big(\Delta_{\delta(\widetilde{\mbf{Y}})}(\widehat{\mbf{Y}} )\big)}{\sigma\big(\Delta_{\delta(\widetilde{\mbf{Y}})}(\widehat{\mbf{Y}})\big)}
	\approx 1\,,
	\]
which yields \eqref{ahest}.  Indeed, in the string of inequalities above we first used that $\Sigma$ is a Lip(1,1/2) graph so that 
$\psi(h(\mbf{y});\mbf{y})-\psi(\mbf{y}) 
\approx \delta(\psi(h(\mbf{y});\mbf{y}),\mbf{y})$,
then that $u(\psi_r(\mbf{y}),\mbf{y}) = r$ (with $r=h(\mbf{y})$), and then \eqref{distytildeyhat},
Lemma \ref{CFMS},  and the strong Harnack 
inequality applied to $u$. 
The last estimate is \eqref{342rt3qf} and \eqref{tydist}.
\end{proof}

We let $\Lambda_1:=\Lambda$,  where $\Lambda$ from now on is fixed so that Lemma \ref{lemma:a-h-small} holds. This imposes a condition on $\epsilon$, and with  this condition met, along with \eqref{projsizeeq.eq}, we have 
fixed $\epsilon$. Next, we introduce the region
\begin{equation}\label{omz.eq}
\mathcal{S}:=\big\{(r,\mbf{y})\in I_{50R}(0,\mbf{x}^0): h(\mbf{y})<r<R/\Lambda_1\big\},
\end{equation}
which can be understood as a local sawtooth relative to $F$ since $h(\mbf{y})\approx\dist(\mbf{y},F)$. By Lemma~\ref{lemma:a-h-small} we have that $h(\mbf{y})\le R/(80\Lambda_1)$, hence the 
top of $\mathcal{S}$ (at height $R/\Lambda_1$) is 
above $h(\mbf{y})$. Moreover,  
there exists $c_1\in (0,1)$ so that
$\psi(h(\mbf{y});\mbf{y})-\psi(\mbf{y})\ge c_1 h(\mbf{y})$ for every $\mbf{y}\in Q_{50R}(\mbf{x}^0)$ (see \eqref{hlowerbound}). Here $c_1$ depends on $\epsilon$, which is now fixed. We set
\begin{equation}\label{firstdom}
\Omega_\star
:=\big\{(y_0,\mbf{y})\in I_{100R}(\mbf{X}^0): 
y_0>\psi_\star(\mbf{y})\big\}, \quad \psi_{\star}: = \psi+c_1h.
\end{equation}
 We define
\begin{equation}\label{wtomegadef.eq}
\Omega_{\mathcal{S}} = \{(\psi(r;\mbf{y}), \mbf{y}): (r,\mbf{y})\in \mathcal{S}\}.
\end{equation}

\begin{lemma}\label{omnmapstar.lem}
	If $(r,\mbf{y})\in \mathcal{S}$ then $(\psi(r;\mbf{y}), \mbf{y})\in \Omega_\star$, that is, ${\Omega}_{\mathcal{S}} \subseteq \Omega_\star$.
\end{lemma}

 \begin{proof}
Let $(r,\mbf{y})\in \mathcal{S}$, that is, $\mbf{y} \in Q_{50R}(\mbf{x}^0)$ and $h(\mbf{y})<r<R/\Lambda_1$. Lemma~\ref{lemma:level-sets} yields that $r\mapsto \psi_r(\mbf{y})$ is strictly increasing for  
$\mbf{y}\in  Q_{100R}(\mbf{x}^0)$ and $0<r<R/\Lambda_0$. 
Hence, 
\eqref {hlowerbound} gives
\[
\psi(r;\mbf{y})> \psi(h(\mbf{y});\mbf{y})\ge \psi(\mbf{y})+c_1h(\mbf{y})=\psi_\star(\mbf{y}).
\]
It remains to show that $(\psi(r;\mbf{y}), \mbf{y})\in  I_{100R}(\mbf{X}^0)$.
To this end, note that since $\psi(r;\mbf{y})$ is increasing in $r$, and
since $x_0^0=\psi(\mbf{x}^0)$ by definition, it holds
\[
	\psi(r;\mbf{y})-\psi(\mbf{x}^0)
\,\le\,
\psi(R/\Lambda_1;\mbf{y})-x_0^0
 \,\le\,
\psi(R/\Lambda_0;\mbf{y})-x_0^0
\,\le\, 300M_0 \sqrt{n}R \,,
\]
where we used Lemma \ref{lemma:level-sets}, specifically \eqref{abdlevelsets.eq}.
On the other hand, since also $\psi(r;\mbf{y}) \ge \psi(\mbf{y})$, and 
\[
|\psi(\mbf{y}) - \psi(\mbf{x}^0)| \leq M_0 \|\mbf{y} - \mbf{x}^0\| <
100M_0\sqrt{n} R\,,
\] 
we find that $\psi(r;\mbf{y})-\psi(\mbf{x}^0) \geq -100M_0\sqrt{n} R$,
and hence that  $(\psi(r;\mbf{y}), \mbf{y})\in \Omega_\star$.
 \end{proof}

 \section{Square function estimates}\label{SecSq}
 Recall the region $\mathcal{S}$ introduced  in \eqref{omz.eq}. The purpose of this section is to prove the following square function estimate.
Recall that $d:=n+1$ is the homogeneous dimension of parabolic $\rn$.

\begin{proposition}\label{prop:CM-ar} We have
\begin{equation}\label{psiSbound}
\iiint_{\mathcal{S}}\big(
|r\,\partial_s \psi_r(y,s)|^2+|r\,\nabla_{y,r}^2 \psi_r(y,s)|^2+ |r^2 \,\nabla_{y,r}\partial_s \psi_r(y,s)|^2\big)\,\frac{\d r}{r}{\d y\,\d s}
\lesssim R^d.
\end{equation}
\end{proposition}

To prove Proposition \ref{prop:CM-ar} we will  use several auxiliary domains which we next introduce. Recall that the domain
\begin{equation*}
\Omega_\star
:=\big\{(y_0,\mbf{y})\in I_{100R}(\mbf{X}^0): 
y_0>\psi_\star(\mbf{y})\big\}, \quad \psi_{\star}: = \psi+c_1\,h.
\end{equation*}
was introduced in \eqref{firstdom}. $\Omega_\star$  is a pseudo-sawtooth relative to  $F_\star$ whose boundary agrees with $\Sigma$ above $F$, that is, on $F_\star$. In the following we will also use the domains
\begin{align*}
\Omega_{\star\star}
&:=\big\{(y_0,\mbf{y})\in I_{125R}(\mbf{X}^0): y_0>\psi_{\star\star}(\mbf{y})\big\}, \quad \psi_{\star\star}: = \psi+c_1h/2,\\
\Omega_{\star\star\star}
&:=\big\{(y_0,\mbf{y})\in I_{150R}(\mbf{X}^0): y_0>\psi_{\star\star\star}(\mbf{y})\big\}, \quad \psi_{\star\star\star}: = \psi+c_1h/4,\\
\Omega_{\star\star\star\star}
&:=\big\{(y_0,\mbf{y})\in I_{175R}(\mbf{X}^0): y_0>\psi_{\star\star\star\star}(\mbf{y})\big\}, \quad \psi_{\star\star\star\star}: = \psi+c_1h/8.
\end{align*}
By construction
$$\Omega_{\star}\subset\Omega_{\star\star}\subset\Omega_{\star\star\star}\subset\Omega_{\star\star\star\star}.$$
We first prove the following lemma which shows that by the construction $u$ behaves like the distance to $\Sigma$ in $\Omega_{\star\star\star\star}$.

\begin{lemma}\label{gislikedeltainsss.eq} It holds that
\begin{equation}\label{eqdistlemma1}
u(\mbf{X}) \approx \delta(\mbf{X}), \quad \forall \mbf{X} \in \Omega_{\star\star\star\star}. 
\end{equation}
Furthermore,
\begin{equation}\label{eqdistlemma2}
h(\mbf{x})\lesssim\delta(\mbf{X}), \quad \forall \mbf{X}=(x_0,\mbf{x}) \in \Omega_{\star\star\star\star}.
\end{equation}
\end{lemma}
\begin{proof} 
Since $\Sigma$ is a Lip(1,1/2) graph,  
$\delta(\mbf{X})\approx x_0 - \psi(\mbf{x}) \ge c_1h(\mbf{x})/8$, by the very definition of $\Omega_{\star\star\star\star}$, which gives \eqref{eqdistlemma2}.

To prove \eqref{eqdistlemma1}, 
consider $\mbf{X}=(x_0,\mbf{x})= (x_0,x,t) \in \Omega_{\star\star\star\star}$. Using that
$\dist(\mbf{x}, F)\approx h(\mbf{x})$, 
we deduce from \eqref{eqdistlemma2} that
\begin{equation}\label{eqdistlemma3}
\dist(\mbf{x}, F)\leq C\delta(\mbf{X})
\end{equation}
for a uniform constant $C>1$. We introduce $c := C^{-1}$.

Assume first that  $\delta(\mbf{X}) =\delta(x_0,\mbf{x})  \le c R$,  and let $\widetilde{\mbf{X}}= (\psi(\mbf{x}),\mbf{x}) \in \Sigma$ be the point on the graph below $\mbf{X}$.  Let $\mbf{y}\in F$ be such that 
$\|\mbf{x} - \mbf{y}\| = \dist(\mbf{x}, F)$, 
and let $\mbf{Y} = (\psi(\mbf{y}),\mbf{y})$ be the 
corresponding point on the graph.  
Since $M_0=\|\psi\|_{\text{Lip(1,1/2)}}+1$, by
Lemma \ref{distgraphlem.lem} we have
\begin{align*}
\|\mbf{X} - \mbf{Y}\| &\le \|\widetilde{\mbf{X}} - \mbf{Y}\| + \|\widetilde{\mbf{X}} - \mbf{X}\|
\\ & \le M_0 \|\mbf{x} - \mbf{y}\| + M_0\delta(\mbf{X})\le 2CM_0\delta(\mbf{X}) \le 2M_0R\,,
\end{align*}
provided $\delta(\mbf{X})\le cR$.
Now  using Lemma \ref{CFMS}, Harnack's inequality,
and doubling, we have
\[\frac{u(\mbf{X})}{\delta(\mbf{X})} \approx 
\frac{\hm\big(\Delta(\widetilde{\mbf{X}}, \delta(\mbf{X}))\big)}{\sigma\big(\Delta(\widetilde{\mbf{X}}, \delta(\mbf{X}))\big)} 
\approx 
\frac{\hm\big(\Delta(\mbf{Y}, \delta(\mbf{X}))\big)}{\sigma\big(\Delta(\mbf{Y}, \delta(\mbf{X}))\big)} \approx  1\,.\]
In the last step, we used that $\mbf{Y} \in F_\star$ and that
$\delta(\mbf{X}) <M_1 R/8$, hence \eqref{Fstarprops.eq} holds.

Next, assume that $\delta(\mbf{X}) =\delta(x_0,\mbf{x})  \geq c R$. In this case we don't need to use the set $F$. Since $\mbf{X} \in I_{175R}(\mbf{X}^0)$, we have $\|\mbf{X}  - \mbf{X}^0\| \le 875\sqrt{n}M_0R$ and hence
\[cR \le \delta(x_0,\mbf{x}) \le  875\sqrt{n}M_0R \le M_1R/4.\]
Let again $\widetilde{\mbf{X}}= (\psi(\mbf{x}),\mbf{x}) \in \Sigma$ be the point on the graph below $\mbf{X}$. In this case we have, using  Lemma \ref{CFMS}, the strong Harnack inequality and the doubling property of $\hm$ and $\sigma$, that
\[
\frac{u(\mbf{X})}{\delta(\mbf{X})} \approx 
\frac{\hm\big(\Delta(\widetilde{\mbf{X}}, \delta(\mbf{X}))\big)}{\sigma\big(\Delta(\widetilde{\mbf{X}}, \delta(\mbf{X}))\big)} 
\approx 
\frac{\hm\big(\Delta_{M_1R/4}(\mbf{X}^0)\big)}{\sigma\big(\Delta_{M_1R/4}(\mbf{X}^0)\big)} \approx 1.\]
Here we have used \eqref{Bourgain:aver} in the final step. 
We can justify the use of doubling up to the scale here, see Remark \ref{whnweusrhq.rmk} and Lemma \ref{calisdoub.lem}.
\end{proof}

We will make use of the following lemma. Recall that the closed, standard 
parabolic cube 
$\mathcal{J}_{r}(\mbf{X})$ (note to be confused with the vertically 
elongated open cube $I_r$) was introduced in \eqref{cuba}.

\begin{lemma}\label{binstars.lem}
There is a uniform positive constant 
$ \theta \ll 1$, such that if 
$\mathcal{J} := \mathcal{J}_\rho(\mbf{Z})\subset \Omega$, 
with $\diam(\mathcal{J}) \le \theta \dist(\mathcal{J}, \Sigma)$, 
and $\mathcal{J}\cap\Omega_{\star}\neq\emptyset$, then
\[100\mathcal{J} = \mathcal{J}_{100\rho} (\mbf{Z}) \subset \Omega_{\star\star}.\]
Furthermore, the same statement is true if the pair of domains $(\Omega_{\star},\Omega_{\star\star})$ is replaced by either  $(\Omega_{\star\star},\Omega_{\star\star\star})$ or  $(\Omega_{\star\star\star},\Omega_{\star\star\star\star})$.
\end{lemma}
\begin{proof}
We prove the statement only 
for the pair $(\Omega_{\star},\Omega_{\star\star})$, 
as the arguments for the other pairs are entirely analogous.
First note that $\diam(\mathcal{J}) \le \theta \delta(\mbf{X})$ for all $\mbf{X} \in \mathcal{J}$. Fix $\mbf{X}  \in \mathcal{J} \cap \Omega_{\star}$. Then $\mbf{X} \in I_{100R}(\mbf{X}^0)$, hence $\delta(\mbf{X}) \lesssim M_0R$, so that for
$\theta \ll M_0^{-1}$,
\[
100 \rho< 200\diam(\mathcal{J}) \leq 200\theta \delta(\mbf{X}) \ll R\,,
\]
whence it follows that
$100\mathcal{J} \subset I_{125R}(\mbf{X}^0)$. 

It remains only to show that $100\mathcal{J}$ stays above the function defining the lower boundary of $\Omega_{\star\star}$. 
To this end, for future reference, we 
note that for $\mbf{X}  \in \mathcal{J} \cap \Omega_{\star}$,
\begin{equation}\label{distcomp2}
\|\mbf{X} - \mbf{Y}\| \leq  100\diam(\mathcal{J})
\leq 100\theta \dist(\mathcal{J},\Sigma) \leq 100\theta \delta(\mbf{X})\,,
\qquad \forall\, \mbf{Y} \in 100\mathcal{J}.
\end{equation}
We consider two cases.

With $\mbf{X} = (x_0,x,t)\in
\mathcal{J} \cap \Omega_{\star}$ fixed as above, we
first suppose that $\delta(\mbf{X}) \ge h(x,t)$. 
We then have
\begin{equation}\label{starstardist}
x_0 - \psi_{\star\star}(x,t) 
= x_0 - \psi(x,t) - c_1 h(x,t)/2
\ge \delta(\mbf{X}) - c_1 h(x,t)/2 \ge \delta(\mbf{X})/2\,,
\end{equation}
since $c_1<1$.
Using Lemma \ref{distgraphlem.lem} (applied in $\Omega_{\star\star}$, so $M_0$ is replaced by 
$\|\psi_{\star\star}\|_{\text{Lip(1,1/2)}}+1\lesssim M_0 + 1$), 
along with \eqref{starstardist}, we deduce that
\begin{equation}\label{distcomp}
\dist(\mbf{X}, \Sigma') \leq \delta(\mbf{X}) \lesssim (M_0 + 1) \dist(\mbf{X}, \Sigma'),
\end{equation}
where $\Sigma':= 
\pom_{\star\star}= \{(\psi(x,t)+(c_1/2)\,h(x,t), x,t): (x,t) \in \rn \}$. 
Here, $\dist(\mbf{X}, \Sigma')\le \delta(\mbf{X})$ trivially, 
since $\Sigma'$ lies at or above $\Sigma$.  

We now claim that
$\delta'(\mbf{Y}) :=\dist(\mbf{Y}, \Sigma') > 0$,
for all $\mbf{Y} \in 100\mathcal{J}$. 
Indeed, 
since $\delta'$ is parabolically Lipschitz 
with norm 1, for $\mbf{Y} \in 100\mathcal{J}$, by \eqref{distcomp2},
we find that
\begin{equation*}
\left|\delta'(\mbf{X})
- \delta'(\mbf{Y}) \right| \leq  100\theta \delta(\mbf{X})
\lesssim \theta \delta'(\mbf{X}) \,,
\end{equation*}
where we used
\eqref{distcomp} and the definition of $\delta'$ in the last step.
Consequently,
\[
\delta' (\mbf{Y}) = \delta'(\mbf{X}) -\big(\delta'(\mbf{X})
- \delta'(\mbf{Y}) \big) \geq (1-C\theta)\delta'(\mbf{X}) 
>0\,,
\]
provided that $\theta$ is chosen small enough.  
Since  $100\mathcal{J}$ is convex and contains a point $\mbf{X}$ in $\Omega_{\star\star}$, it then follows that 
$100\mathcal{J}\subset \Omega_{\star\star}$.

Now suppose that $\delta(\mbf{X}) \le h(x,t)$. In this case, 
we first note that  
\[
x_0  - (\psi(x,t) + (c_1/2) h(x,t)) >  (c_1/2) h(x,t)\,,
\]
as $\mbf{X} \in \Omega_{\star}$.
For $\mbf{Y} = (y_0,y,s) \in 100\mathcal{J}$, 
write
\begin{multline*}
y_0 - (\psi(y,s) + c_1 h(y,s)/2) \\[4pt]
= \bigl (x_0  - (\psi(x,t) + c_1 h(x,t)/2)\bigr)
 +\bigl ((\psi(x,t) + c_1 h(x,t)/2)- (\psi(y,s) + c_1 h(y,s)/2)\bigr )
+ (y_0 - x_0),
\end{multline*}
so that, using \eqref{distcomp2}, we have
\begin{multline*}
y_0 - (\psi(y,s) + c_1 h(y,s)/2) \ge (c_1/2) h(x,t) - C\|\mbf{Y} - \mbf{X}\|
\\ \ge (c_1/2) h(x,t)  - C\theta\delta(\mbf{X})
 \ge (c_1/2) h(x,t) - C\theta h(x,t)  > 0,
\end{multline*}
provided $\theta$ is sufficiently small; i.e., we conclude that  $\mbf{Y}$ lies above the graph defining $\Omega_{\star\star}$. 
\end{proof}

\subsection{Whitney decompositions} 
We here introduce Whitney decompositions.
Let $\theta$ be as in Lemma \ref{binstars.lem}, chosen small enough that the conclusion of the Lemma holds for each pair of sub-domains $(\Omega_{\star},\Omega_{\star\star})$,
$(\Omega_{\star\star},\Omega_{\star\star\star})$, and  $(\Omega_{\star\star\star},\Omega_{\star\star\star\star})$.
 We let $\mathcal{W} = \{\mathcal{J}\}$ be a 
 (parabolic) Whitney decomposition of $\Omega$ in parabolic dyadic cubes
 (see e.g. \cite[Chapter 6]{Stein-SIOs} for the classical construction, which adapts readily to the parabolic setting) with the additional property that
\begin{equation}\label{thetapwhitdecomp.eq}
\theta^{-1} \diam(\mathcal{J}) \le  \dist(2\mathcal{J}, \Omega^c) \le \dist(\mathcal{J}, \Omega^c) \le  C_n \theta^{-1}  \diam(\mathcal{J}).
\end{equation}
This means that each $\mathcal{J} \in \mathcal{W}$ is a (closed) parabolic dyadic cube and whenever $\mathcal{J}$ and $\mathcal{J}'$ are distinct they have disjoint interiors. Moreover, taking $\theta$ smaller if need be,
we can insist that when $10 \mathcal{J} \cap 10\mathcal{J}' \neq \emptyset$ then $\ell(\mathcal{J}) \approx \ell(\mathcal{J}')$, and that the collection $\{10\mathcal{J}\}$ have bounded overlap.  

Note that by Lemma \ref{binstars.lem}, if $\mathcal{J} \in \W$ and $\mathcal{J}$ meets $\Omega_{\star}$, then $10\mathcal{J} \in \Omega_{\star\star}$. Similarly, 
if $\mathcal{J}$ meets $\Omega_{\star\star}$,  then $10\mathcal{J} \in \Omega_{\star\star\star}$ and finally, if 
$\mathcal{J}$ meets $\Omega_{\star\star\star}$, then $10\mathcal{J} \in \Omega_{\star\star\star\star}$.

\subsection{Pointwise estimates}
We now present some preliminary estimates for the Green function and to be used in the proof of Proposition \ref{prop:CM-ar}.

\begin{lemma} It holds that
\begin{equation}\label{est:G-delta-sawtooth:deriv:1}
\partial_{y_0}u(\mbf{Y})
\approx  \frac{u(\mbf{Y})}{\delta(\mbf{Y})}
\approx 1,\quad \forall \mbf{Y}\in \Omega_{\star\star\star},
\end{equation}
and
\begin{equation}\label{est:G-delta-sawtooth:deriv:2}
	\delta(\mbf{Y})^k\,|\nabla_Y^{k+1} u(\mbf{Y})|
	+
	\delta(\mbf{Y})^{k+1}\,|\nabla_Y^k \partial_s u(\mbf{Y})|
	\lesssim_k 1, \quad \forall \mbf{Y}\in \Omega_{\star\star\star},
\end{equation}
and for every $k \ge 0$. Moreover, if $\mathcal{J} \in \W$, and $\mathcal{J}$ meets $\Omega_{\star\star}$, then
\begin{multline}\label{redderivtxest.eq}
\iiint_{\mathcal{J} }
\big(
|\delta\,\partial_s u|^2
+
|\delta\,\nabla_Y^2 u|^2
+
|\delta^2\,\nabla_Y\partial_s u|^2
+
|\delta^2\,\partial_s u\, \nabla_Y^2u|^2
\big) \, \frac{\d\mbf{Y}}{\delta(\mbf{Y})}
\\[4pt]  \lesssim \,
\iiint_{4\mathcal{J} }
\big(
|\partial_s u|^2
+
|\nabla^2_Y u|^2
\big) 
\, \delta(\mbf{Y})\d\mbf{Y}
\, \approx\,
\iiint_{4\mathcal{J} }
\big(
|\partial_s u|^2
+
|\nabla^2_Y u|^2
\big)
u \, \d\mbf{Y}.
 \end{multline}
\end{lemma}

\begin{proof}
The estimate in \eqref{est:G-delta-sawtooth:deriv:1} is just a combination of \eqref{eqn:est-der-G} and \eqref{eqdistlemma1}. The proof of the remaining estimates stated in the lemma follows from standard interior estimates for the heat/adjoint heat equation, Lemma \ref{gislikedeltainsss.eq} and Lemma \ref{binstars.lem}. We omit the routine details.\end{proof}

To proceed we introduce
\[\mbf{\Psi}_r(\mbf{y}) := (\psi_r(\mbf{y}), \mbf{y}),\quad r \in (0, R/\Lambda_0),\ \mbf{y} \in Q_{100R}(\mbf{x}^0).\]
Then for any fixed $r \in (0, R/\Lambda_0)$, $\mbf{\Psi}_r(\mbf{y})$ is the point on the $r$-level set of $u$ above $\mbf{y} \in Q_{100R}(\mbf{x}^0)$, i.e.,
\begin{equation}\label{old4.17}
r = u(\mbf{\Psi}_r(\mbf{y}))  = u(\psi_r(\mbf{y}), \mbf{y}) \quad \forall r \in (0, R/\Lambda_0),\  \mbf{y} \in Q_{100R}(\mbf{x}^0).
\end{equation}
Furthermore, as the set $\mathcal{S}$ is contained in the domain of $\mbf{\Psi}_r(\mbf{y})$, Lemma \ref{omnmapstar.lem} and Lemma \ref{gislikedeltainsss.eq} imply that
\begin{equation}\label{old4.18.eq}
r = u(\mbf{\Psi}_r(\mbf{y})) \approx \delta(\mbf{\Psi}_r(\mbf{y})) \approx \psi_r(\mbf{y}) - \psi(\mbf{y}), \quad \forall (r,\mbf{y}) \in \mathcal{S},
\end{equation}
where the last equivalence comes from Lemma \ref{distgraphlem.lem} and we have used that $\Omega_{\star} \subset \Omega_{\star\star\star}$.

The following lemma allows us to relate our estimates on the normalized Green function $u$ to corresponding estimates on $\psi_r$.

\begin{lemma}\label{lemma:estimates:G-star}
Let $(r,y,s)\in\mathcal{S}$. Then,
\begin{equation}\label{eqn:estimates:G-star:1}
|\partial_s \psi_r(y,s)|\lesssim |(\partial_s u)(\mbf{\Psi}_r(y,s))|,
\quad
|\nabla_{y,r} \psi_r(y,s)|\lesssim 1,
\quad 
|\nabla_{y,r}^2 \psi_r(y,s)|\lesssim |(\nabla_Y^2 u)(\mbf{\Psi}_r(y,s))|,
\end{equation}	
\begin{equation}\label{eqn:estimates:G-star:2}
|\nabla_{y,r}\partial_s \psi_r(y,s)|
\,\lesssim\,
|(\nabla_Y\partial_s u)(\mbf{\Psi}_r(y,s))|\,+\, |(\partial_s u)(\mbf{\Psi}_r(y,s))|\,|(\nabla_Y^2 u)(\mbf{\Psi}_r(y,s))|
\end{equation}	
\begin{equation}\label{eqn:estimates:G-star:3}
r\,|\partial_s \psi_r(y,s)|\,+\,r\,|\nabla_{y,r}^2 \psi_r(y,s)|\,+\, r^2\,|\nabla_{y,r}^3 \psi_r(y,s)|\,+\,r^2\, |\nabla_{y,r}\partial_s \psi_r(y,s)|\,\lesssim 1.
\end{equation}	
\end{lemma}

\begin{proof} 
For $(r,y,s)\in\mathcal{S}$, we have $\mbf{\Psi}_r(y,s)\in\Omega_\star$, by Lemma \ref{omnmapstar.lem}.  Thus, by \eqref{est:G-delta-sawtooth:deriv:1}, 
\begin{equation}\label{uyPsi}
(\partial_{y_0} u)(\mbf{\Psi}_r(y,s)) \approx 1\,,
\qquad \forall \, (r,y,s)\in\mathcal{S}\,.
\end{equation}
Differentiating \eqref{old4.17} with respect to $s$, we obtain
\[
0
=
\partial_s\bigl (u(\psi_r(y,s), y,s)\big)
=
(\partial_s u)(\mbf{\Psi}_r(y,s))+(\partial_{y_0} u)(\mbf{\Psi}_r(y,s))\,\partial_s \psi_r(y,s).
\]
I.e.,
\begin{equation}\label{partial-s:a-r}
	\partial_s \psi_r(y,s)
=
-\frac{(\partial_s u)(\mbf{\Psi}_r(y,s))}{(\partial_{y_0} u)(\mbf{\Psi}_r(y,s))}.
\end{equation}
Similarly, if we differentiate \eqref{old4.17} with respect either to $r$, or to $y_i$ with $1\le i\le n-1$, we obtain
\begin{equation}\label{partial-r-yi:a-r}
\partial_r \psi_r(y,s)
=
\frac1{(\partial_{y_0} u)(\mbf{\Psi}_r(y,s))},
\qquad
\partial_{y_i} \psi_r(y,s)
=
-\frac{(\partial_{y_i} u)(\mbf{\Psi}_r(y,s))}{(\partial_{y_0} u)(\mbf{\Psi}_r(y,s))}.
\end{equation}
We readily obtain the first two estimates in  \eqref{eqn:estimates:G-star:1} from
 \eqref{uyPsi}
 and \eqref{est:G-delta-sawtooth:deriv:2}. The last estimate in \eqref{eqn:estimates:G-star:1}, and also \eqref{eqn:estimates:G-star:2}, 
 follow by differentiating the formulas \eqref{partial-s:a-r} and \eqref{partial-r-yi:a-r}, and by invoking again 
  \eqref{uyPsi}
 and \eqref{est:G-delta-sawtooth:deriv:2}. Finally \eqref{eqn:estimates:G-star:3} follows using the previous estimates, the same kind of arguments, and \eqref{old4.18.eq}. Details are left to the interested reader.
\end{proof}

 \subsection{Proof of Proposition \ref{prop:CM-ar}} 
 With the preceeding preliminaries in hand, we are now ready to prove Proposition \ref{prop:CM-ar}. 
By Lemma~\ref{lemma:estimates:G-star}, \eqref{est:G-delta-sawtooth:deriv:1}, and 
\eqref{old4.18.eq}, we derive
\begin{multline}\label{ibpstp1.eq}
\iiint_{\mathcal{S}}\big(
|r\,\partial_s \psi_r(y,s)|^2+|r\,\nabla_{y,r}^2 \psi_r(y,s)|^2+ |r^2 \,\nabla_{y,r}\partial_s \psi_r(y,s)|^2\big)\,\frac{\d r}{r}{\d y\,\d s}
\\  \lesssim
\iiint_{\mathcal{S}}
\big(
|\delta\,\partial_s u|^2
+
|\delta\,\nabla_Y^2 u|^2
\big)(\mbf{\Psi}_r(y,s))\,\frac{\d r\,\d y\,\d s}{\delta
(\mbf{\Psi}_r(y,s))}
\\ + \,\,
\iiint_{\mathcal{S}}
\big(
|\delta^2\,\nabla_Y\partial_s u|^2
+
|\delta^2\,\partial_s u\, \nabla_Y^2u|^2
\big)(\mbf{\Psi}_r(y,s))\,\frac{\d r\,\d y\,\d s}{\delta
(\mbf{\Psi}_r(y,s))} =:\, I+II.
\end{multline}

Next, we use the definition of ${\Omega}_{\mathcal{S}}$, see \eqref{wtomegadef.eq}, and the change of variable
 $\mbf{Y} =\mbf{\Psi}_r(y,s)$, which amounts to  
 the 1-dimensional change of variable
 $y_0=\psi_r(y,s)$, with $(y,s)$ fixed.  By 
 \eqref{uyPsi} and \eqref{partial-r-yi:a-r}, the 
 Jacobian of this change of variable is uniformly 
 bounded above and below. Hence, 
  by Lemma \ref{omnmapstar.lem},
 \begin{equation}\label{ibpstp3.eq} 
  I + II 
\lesssim \iiint_{\Omega_\star}
\big(
|\delta\,\partial_s u|^2
+
|\delta\,\nabla_Y^2 u|^2
+
|\delta^2\,\nabla_Y\partial_s u|^2
+
|\delta^2\,\partial_s u\, \nabla_Y^2u|^2
\big)(\mbf{Y})\,\frac{\d\mbf{Y}}{\delta(\mbf{Y})}\, =: III \,.
 \end{equation}

 Let $\W_\star = \{\mathcal{J} \in \W: \mathcal{J} \cap \Omega_\star \neq \emptyset\}$. As $\{10\mathcal{J}\}_{\mathcal{J} \in \W_\star}$ have bounded overlap, we can use \eqref{redderivtxest.eq} and 
  then Lemma \ref{binstars.lem} (along with 
 the definition of $\W$, 
 specifically \eqref{thetapwhitdecomp.eq}) to deduce that
 \begin{multline}\label{ibpstp4.eq}
III \,\lesssim\,
\sum_{\mathcal{J} \in \W_\star}\iiint_{\mathcal{J}}\big(
|\delta\,\partial_s u|^2
+
|\delta\,\nabla_Y^2 u|^2
+
|\delta^2\,\nabla_Y\partial_s u|^2
+
|\delta^2\,\partial_s u\, \nabla_Y^2u|^2
\big)(\mbf{Y})\,\frac{\d\mbf{Y}}{\delta(\mbf{Y})}
\\
 \lesssim
\sum_{\mathcal{J} \in \W_\star}\iiint_{4\mathcal{J}}\big(
u|\partial_s u|^2
+
u|\nabla_Y^2 u|^2
\big)(\mbf{Y})\,\d\mbf{Y}
 \lesssim
\iiint_{\Omega_{\star\star}}\big(
u|\partial_s u|^2
+
u|\nabla_Y^2 u|^2
\big)(\mbf{Y})\,\d\mbf{Y}.
\end{multline}

Combining \eqref{ibpstp1.eq}, \eqref{ibpstp3.eq} and \eqref{ibpstp4.eq}, we have reduced the proof of  Proposition \ref{prop:CM-ar} to proving that
\begin{equation}\label{ibnwstup.eq}
\iiint_{\Omega_{\star\star}}\big(
u|\partial_s u|^2
+
u|\nabla_Y^2 u|^2
\big)(\mbf{Y})\,\d\mbf{Y} \lesssim R^d.
\end{equation}
To this end, for $N \in \mathbb{N}$, with $N \gg \Lambda_1/R$, we set
\begin{equation}\label{Ndef}
\Omega_{\star\star,N}
:=\big\{(y_0,\mbf{y})\in I_{125R}(\mbf{X}^0): y_0> \psi_N(\mbf{y})\big\}, \quad \psi_N: = \psi+c_1\,h/2 + 1/N,
\end{equation}
that is, $\Omega_{\star\star,N}$ is the domain formed by pushing the lower boundary of $\Omega_{\star\star}$ up by a distance of $1/N$. By the monotone convergence theorem we see that  to prove \eqref{ibnwstup.eq}, it is enough to prove that
\begin{equation}\label{ibnwstuptrunc.eq}
\iiint_{\Omega_{\star\star,N}}\big(
u|\partial_s u|^2
+
u|\nabla_Y^2 u|^2
\big)(\mbf{Y})\,\d\mbf{Y} \lesssim R^d\,,
\end{equation}
uniformly in $N$.

Using the collection of Whitney-type cubes $\W = \{\mathcal{J}\}$, 
we form a partition of unity.  
That is, using the fact that
\begin{equation}\label{whitneypts.eq}
\delta(\mbf{X}) \approx \theta^{-1} \ell(\mathcal{J}),\quad 
\forall \, \mbf{X} \in 2\mathcal{J},\, \mathcal{J} \in \W\,,
\end{equation}
we construct
$\eta_\mathcal{J} \in C_0^\infty(2\mathcal{J})$, with the properties
\begin{equation}\label{POUderest.eq}
\ell(\mathcal{J})|\nabla \eta_\mathcal{J}| + \ell(\mathcal{J})^2|\partial_t \eta_\mathcal{J}| \lesssim 1\,,
\end{equation}
such that 
\begin{equation}\label{partunity}
\sum_{\mathcal{J} \in \W} \eta_{\mathcal{J}}(\mbf{X}) = 1,\qquad
\forall \, \mbf{X} \in \Omega\,.
\end{equation}
The well known construction
in \cite[Chapter 6]{Stein-SIOs} adapts routinely to the parabolic setting.  We omit the details.


Now we set
\begin{equation}\label{WssNdef.eq}
\W_{\star\star,N} = \{\mathcal{J} \in \W: 2\mathcal{J} \cap \Omega_{\star\star,N} \neq \emptyset\}
\end{equation}
and
\begin{equation}\label{etaNdefin.eq}
\eta := \eta_N := \sum_{\mathcal{J} \in \W_{\star\star,N}} \eta_{\mathcal{J}}.
\end{equation}
Since $\supp \eta_{\mathcal{J}} \subset 2\mathcal{J}$, and $1 = \sum_{\mathcal{J}} \eta_\mathcal{J}$ on $\Omega$,  
we have $\eta \equiv 1$ on $\Omega_{\star\star,N}$. Moreover,
by Lemma \ref{binstars.lem} and \eqref{thetapwhitdecomp.eq}, 
we see that if $\mathcal{J} \in W_{\star\star,N}$, then $10\mathcal{J} \in \Omega_{\star\star\star}$. Hence
\begin{equation}\label{etasupp.eq}
\mathbbm{1}_{\Omega_{\star\star,N}} \le \eta \le \mathbbm{1}_{\Omega_{\star\star\star}}.
\end{equation}
In particular,
\begin{equation}\label{A2Bbdsssn.eq}
\iiint_{\Omega_{\star\star,N}}\big(
u|\partial_s u|^2
+
u|\nabla_Y^2 u|^2
\big)(\mbf{Y})\,\d\mbf{Y} \leq A+B \leq A + 2B,
\end{equation}
where
\[A := \sum_{0\leq i,j\leq n-1}\iiint  u ( u_{y_i y_j} )^2\eta \,\d\mbf{Y},\qquad B:= \iiint u u_s^2\eta\, \d\mbf{Y}\,.\]
The following square function estimate is adapted from
 \cite{LeNy}, and is crucial to our ability to impose an assumption 
 {\em only} on the caloric measure, or only on the adjoint caloric measure,
 but not on both simultaneously. 
 For the convenience of the reader we 
 include a proof in the next subsection.  

\begin{lemma}[essentially, \cite{LeNy}]
\label{sqrf} Let $A$ and $B$ be defined as above. Then
\begin{equation*}
A+2B\lesssim  \iiint \big(|\nabla_Y \eta(\mbf{Y})|+\delta(\mbf{Y} )|\partial_s \eta(\mbf{Y})|\big)\, \d\mbf{Y}.
\end{equation*}
\end{lemma}
We take the lemma for granted momentarily, and 
defer the proof to Subsection \ref{sgfproof}.

We claim that
\begin{equation}\label{POUtruncIBP.eq}
\iiint \big(|\nabla_Y \eta(\mbf{Y})|
+\delta(\mbf{Y} )|\partial_s \eta(\mbf{Y})|\big)\, \d\mbf{Y} \lesssim R^d.
\end{equation}
Taking \eqref{POUtruncIBP.eq} for granted momentarily, 
we observe that
\eqref{POUtruncIBP.eq},  Lemma \ref{sqrf} and \eqref{A2Bbdsssn.eq} 
combine to give \eqref{ibnwstuptrunc.eq}. Therefore we have, modulo Lemma \ref{sqrf}, reduced matters to verifying \eqref{POUtruncIBP.eq}.

To start the proof of \eqref{POUtruncIBP.eq}, recall the definition of $\eta$ \eqref{etaNdefin.eq} and $\W_{\star\star,N}$ \eqref{WssNdef.eq}. Note that by \eqref{etasupp.eq},  if $D$ 
is any partial derivative operator (in $X$ or $t$), then 
$D \eta(\mbf{X}) = 0$ if $\mbf{X} \in \Omega_{\star\star,N}$. Hence
\[
|D \eta| \le \sum_{\substack{\mathcal{J} \in 
\W_{\star\star,N}\\ 2\mathcal{J} \cap( \Omega_{\star\star,N} )^c 
\neq \emptyset} }|D \eta_\mathcal{J}|.
\]
By definition $2\mathcal{J} \cap \Omega_{\star\star,N} \neq \emptyset$ 
for all $\mathcal{J} \in \W_{\star\star,N}$, and for such $\mathcal{J}$,
\[
2\mathcal{J} \cap( \Omega_{\star\star,N} )^c \neq \emptyset \implies 
2\mathcal{J} \cap \partial\Omega_{\star\star,N} \neq \emptyset\,.
\]
Thus, if we define
\[\mathcal{B}_N = \{\mathcal{J} \in \W_{\star\star,N}: 2\mathcal{J} \cap \partial\Omega_{\star\star,N} \neq \emptyset\},\]
then it holds that 
\[|D \eta| \le \sum_{\mathcal{J} \in \mathcal{B}_N }|D \eta_\mathcal{J}|\,.\]
Therefore, to prove \eqref{POUtruncIBP.eq} it is enough to show
\begin{equation}\label{BNgoal.eq}
\sum_{\mathcal{J} \in \mathcal{B}_N } \iiint 
\big(|\nabla_Y \eta_\mathcal{J}(\mbf{Y})|+\delta(\mbf{Y} )|\partial_s \eta_\mathcal{J}(\mbf{Y})|\big)\,\d\mbf{Y} \lesssim R^d.
\end{equation}

Let $\Sigma_N : = \{(\psi_N(x,t),(x,t)): (x,t) \in \rn\}$, where 
$\psi_N$ is defined in \eqref{Ndef}. Then
\[
\partial\Omega_{\star\star,N} 
=  (\partial\Omega_{\star\star,N} \cap \Sigma_N)  \cup \big(\partial\Omega_{\star\star,N} \cap \partial I_{125R}(\mbf{X}^0)\big),
\]
 and we can further decompose $\B_N = \B_N^{(1)} \cup \B_N^{(2)}$, where
 \[\B_N^{(1)}= \{\mathcal{J} \in \W_{\star\star,N}: 2\mathcal{J} \cap (\partial\Omega_{\star\star,N} \cap \Sigma_N)  \neq \emptyset\},\]
 and
 \[
 \B_N^{(2)}= \{\mathcal{J} \in \W_{\star\star,N}: 2\mathcal{J} \cap \big(\partial\Omega_{\star\star,N} \cap \partial I_{125R}(\mbf{X}^0)\big)
  \neq \emptyset\}.
 \]
We will handle contributions to \eqref{BNgoal.eq} coming from each of these families similarly, but with a small difference which we point out later. Let us start with $\B_N^{(1)}$. For each $\mathcal{J} \in \B_N^{(1)}$ we have that $\sigma_N(E_\mathcal{J}) \approx \ell(\mathcal{J})^d$ where 
$E_\mathcal{J}: = 4\mathcal{J} \cap \Sigma_N$ and $\sigma_N: = \cH^d_p|_{\Sigma_N}$, The sets $\{E_\mathcal{J}\}_{\mathcal{J}\in \B_N^{(1)}} $ 
have bounded overlap, since the fattened cubes in 
$\{10 \mathcal{J}\}_{\mathcal{J}\in\W}$ have 
bounded overlap. Moreover,
\[
\bigcup_{\mathcal{J} \in \B_N^{(1)}} 
4\mathcal{J} \subset I_{CR}(\mbf{X}^0),
\]
where $C$ can be taken to depend only on $n$ 
and  $M_0$. 
Thus, since $\sigma_N$ is a 
parabolic ADR measure, with constants also
depending only on $n$ and  $M_0$, 
we have
 \[
 \sum_{\mathcal{J} \in \B_N^{(1)}} \ell(\mathcal{J})^d 
 \approx \sum_{\mathcal{J} \in \B_N^{(1)}} \sigma_N(E_\mathcal{J}) 
 \approx \sigma_N\biggl(\bigcup_{\mathcal{J} \in \B_N^{(1)}} E_\mathcal{J} \biggr) \lesssim R^d.
 \]
Now using 
\eqref{whitneypts.eq} and \eqref{POUderest.eq}, we have
\begin{equation}
\sum_{\mathcal{J} \in \mathcal{B}_N^{(1)}} \iiint 
\big(|\nabla_Y \eta_\mathcal{J}(\mbf{Y})|
+\delta(\mbf{Y} )|\partial_s \eta_\mathcal{J}(\mbf{Y})|\big)\,
\d\mbf{Y} \lesssim \sum_{\mathcal{J} \in \B_N^{(1)}} 
\ell(\mathcal{J})^d \lesssim  R^d.
\end{equation}
This controls the contribution of $\B_N^{(1)}$ to \eqref{BNgoal.eq}.

For each cube $\mathcal{J}\in \B_N^{(2)}$, we have that 
$2\mathcal{J}$ meets a face of 
$\partial I_{125R}(\mbf{X}^0)$. 
We can handle the `space-time' faces (the ones not perpendicular to the time direction) in the same way as we handled the cubes in $\B_N^{(1)}$. 
In the case that $2\mathcal{J}$ meets a face perpendicular 
to the time direction, we need to
do something a little different. 

Let us sketch the argument. 
Recall that $\xbf^0:= (x_0^0,x^0,t^0)$.
Consider, e.g., the time-forward
face, call it $\tt{F}^+$, where 
$t = t_* := t^0 + (125R)^2$,  and
let $\widetilde{\B}_N$ be the collection of cubes in $\B^{(2)}_N$ 
such that $2\mathcal{J}$ meets this face. 
In a similar manner as above, we set 
$E_\mathcal{J} := 4\mathcal{J} \cap \tt{F}^+$, 
which is simply the 
time-slice cross-section of a parabolic cube, i.e., an
$n$-dimensional standard cube.
Consequently, the $n$-dimensional Lebesgue measure of 
$E_\mathcal{J}$ satisfies 
$|E_\mathcal{J}|_n = 
\ell(\mathcal{J})^n =\ell(\mathcal{J})^{d -1}$. 
Note that for $\mbf{X} = (x_0,x,t_*) \in E_{\mathcal{J}}$, 
it holds that 
$\ell(\mathcal{J}) \approx \theta\delta(\mbf{X})$, and furthermore,
for all $\mathcal{J} \in \B_N$, we have
$4\mathcal{J} \subset I_{CR}(\mbf{X}^0)$,
and therefore $\ell(\mathcal{J}) \lesssim R$.
Then by
\eqref{POUderest.eq}, since $d=n+1$ by definition,
\begin{multline*}
\sum_{\mathcal{J} \in \widetilde{\B}_N} 
\iiint \big(|\nabla_Y \eta_\mathcal{J}(\mbf{Y})|+\delta(\mbf{Y} )|\partial_s \eta_\mathcal{J}(\mbf{Y})|
\big)\,\d\mbf{Y} \lesssim \sum_{\mathcal{J} \in \widetilde{\B}_N} \ell(\mathcal{J})^d
= \sum_{\mathcal{J} \in \widetilde{\B}_N} \ell(\mathcal{J}) 
|E_{\mathcal{J}}|_n
\\
\lesssim \, R \sum_{\mathcal{J} \in \widetilde{\B}_N} 
|E_{\mathcal{J}}|_n \, \lesssim \, R^{n+1}\, =\, R^d\,,
\end{multline*}
where in the last inequality, we have used
that the sets $\{E_\mathcal{J}\}$ have bounded overlap, and are all 
contained in ${\tt F}^+$, which is an $n$-dimensional standard cube of 
side length $\ell({\tt F}^+) \approx R$.

The time backwards face, where $t=t^0 - (125R)^2$, can be handled by 
essentially the same argument.
This proves \eqref{BNgoal.eq} and concludes the proof of Proposition \ref{prop:CM-ar} modulo Lemma \ref{sqrf}.

\subsection{Proof of Lemma \ref{sqrf}}\label{sgfproof}
We follow \cite{LeNy}. Recall that $u$ and its derivatives solve the adjoint equation (we will just say `the equation' below), that is, $\partial_su+\Delta u =0= \partial_su_{y_j}+\Delta u_{y_j}$.  Here and below we use $f_{y_j}$ to denote the partial derivative of a function $f$ with respect to $y_j$ and we will $f_s$ to denote the $s$-derivative of $f$. 
We shall use the standard summation convention for repeated indices. 
Integrating by parts and using the equation, we have
\begin{multline*}
A=-\iiint u_{y_i}u_{y_j}u_{y_iy_j}\eta \, \d\mbf{Y}
-\iiint uu_{y_j}\Delta u_{y_j}\eta \, \d\mbf{Y}
-\iiint uu_{y_j}u_{y_iy_j}\eta_{y_i} \, \d\mbf{Y}
\\[4pt]
=-\frac 1 2 \iiint \frac {\partial}{\partial y_j}|\nabla_Y u|^2u_{y_j}\eta \, \d\mbf{Y}
+\frac 1 2 \iiint u\frac {\partial}{\partial s}|\nabla_Y u|^2\eta \, \d\mbf{Y}
-\iiint uu_{y_j}u_{y_iy_j}\eta_{y_i} \, \d\mbf{Y}.
\end{multline*}
Using integration by parts once more, as well as the equation,
\begin{multline*}
A=\frac 1 2 \iiint |\nabla_Y u|^2\Delta u\,\eta \, \d\mbf{Y}+
\frac 1 2 \iiint |\nabla_Y u|^2u_{y_j}\eta_{y_j} \, \d\mbf{Y}
+\frac 1 2 \iiint \Delta u|\nabla_Y u|^2\eta \, \d\mbf{Y}
\\[4pt] 
- \frac 1 2 \iiint u|\nabla_Y u|^2\eta_s \, \d\mbf{Y}
-\iiint uu_{y_j}u_{y_iy_j}\eta_{y_i} \, \d\mbf{Y}
\\[4pt]
= \iiint |\nabla_Y u|^2\Delta u\,\eta \, \d\mbf{Y} \,+\, {\tt E}\,,
\end{multline*}
where ${\tt E}$ represents a sum of error terms satisfying
\begin{equation}\label{errorterm}
|{\tt E}| \lesssim \iiint \left(|\nabla_{Y}u|^3 |\nabla_Y \eta|
\,+\, u|\nabla_{Y}u|^2 |\eta_s| \,
+\,u |\nabla_Y u|\,|\nabla_Y^2 u|\, |\nabla_Y\eta|\right)
\d\mbf{Y}\,.
\end{equation}

We similarly manipulate the term $B$:
\begin{multline*}
B=
\iiint u u_s^2\eta\, \d\mbf{Y}=\iiint u (\Delta u)^2\eta\, \d\mbf{Y}\\
=-\iiint |\nabla_Y u|^2\Delta u\eta\, \d\mbf{Y}-
\iiint u u_{y_i}\Delta u_{y_i}\eta\, \d\mbf{Y}
-\iiint u u_{y_i}\Delta u\eta_{y_i}\, \d\mbf{Y}.
\end{multline*}
By further manipulations, integration by parts and using the equation satisfied by $u$,
\begin{multline*}
B=-\iiint |\nabla_Y u|^2\Delta u\,\eta\, \d\mbf{Y}+\frac 1 2
\iiint u \frac {\partial}{\partial s}|\nabla_Y u|^2\eta\, \d\mbf{Y}
-\iiint u u_{y_i}\Delta u\eta_{y_i}\, \d\mbf{Y}
\\[4pt]
=-\iiint |\nabla_Y u|^2\Delta u\,\eta\, \d\mbf{Y}+\frac 1 2
\iiint \Delta u |\nabla_Y u|^2\eta\, \d\mbf{Y} \,+\, {\tt E}
\\[4pt]
=-\frac12\iiint |\nabla_Y u|^2\Delta u\,\eta\, \d\mbf{Y}
\,+\, {\tt E}\,,
\end{multline*}
where ${\tt E}$ is a sum of two error terms, again controlled by \eqref{errorterm}.
Adding, we conclude that
\[
A+ 2B = {\tt E}\,.
\]
Hence, using \eqref{est:G-delta-sawtooth:deriv:1}, \eqref{est:G-delta-sawtooth:deriv:2}, and \eqref{etasupp.eq}, we deduce from the nature of the error terms that
\begin{equation}\label{a2bid.eq}
A+2B\lesssim  \iiint \big(|\nabla_Y \eta(\mbf{Y})|+\delta(\mbf{Y} )|\partial_s \eta(\mbf{Y})|\big)\ \, \d\mbf{Y}.
\end{equation}

\section{Proof of Theorem \ref{mainthrm.thrm}}\label{final}

As mentioned at the beginning of Section \ref{mainthrmsetup.sect}, to prove Theorem \ref{mainthrm.thrm} it suffices to prove, for the fixed cube $Q_R=Q_R(\mbf{x}^0)$, that there exists $N_\star$, depending only on the structural constants, such that
\begin{equation}\label{JS-lemma}
	\inf_C \big|\big\{
\mbf{y}\in Q_R: |\mathcal{D}_t \psi(\mbf{y})-C|>N_\star
\big\}\big|
\le
\,(1/4)|Q_R|.
\end{equation}
With this goal in mind, recall that $\Pi(\mbf{y}):=(0,y,s)$ for each $\mbf{y}=(y_0,y,s)\in\ree$,  and that the closed set  $F=\Pi(F_\star)\subset Q_{50 R}(\mbf{x}^0)$ was introduced in \eqref{proja+}. By \eqref{projsizeeq.eq} and Chebyshev's inequality,
\[
\big|\big\{
\mbf{y}\in Q_R: |\mathcal{D}_t \psi(\mbf{y})-C|>N_\star
\big\}\big|
\le
\frac1{N_\star} \iint_{F\cap Q_R} |\mathcal{D}_t \psi(\mbf{Y})-C|\,\d\mbf{y}
+
(1/8)|Q_R|.
\]
We claim that
\begin{equation}\label{BMO-q}
\frac1{|Q_R|} \iint_{F\cap Q_R} |\mathcal{D}_t \psi(\mbf{y})-C|\,\d\mbf{y}
\le
C_1,
\end{equation}
where $C_1$ depends only on the structural constants.  We then choose $N_\star = 1/(8C_1)$ to deduce \eqref{JS-lemma} 
from \eqref{BMO-q}. 
Our goal is therefore to obtain \eqref{BMO-q}. We divide the proof into steps.

\subsection{Step 1: Localization} \label{ss6.1}

Let $\varphi \in C_c^\infty(\re)$ be
an even function with $\mathbbm{1}_{(-9,9)}\le \varphi \le \mathbbm{1}_{(-10,10)}$ and set $\Phi(x,t) := \varphi(x_1)\varphi(x_2)\dots \varphi(x_{n-1})
\varphi(t/10)$ for $\mbf{x}= (x,t) = (x_1,\dots x_{n-1}, t) \in\re^n$. We set $\Phi_R(\mbf{x}):= \Phi(x/R,t/R^2)$ for $\mbf{x}= (x,t) \in\re^n$.
Recalling the definition of $\mathrm{I_P}$ in \eqref{def-IP}, we write $V_R=\Phi_R\,V$ for the locally truncated kernel, and we consider the localized 
parabolic fractional integral
\begin{equation}\label{IPRdef.eq}
\mathrm{I}_\mathrm{P}^R h(\mbf{x})
:=
\iint_{\re^n} V_R(\mbf{x}-\mbf{y})\,h(\mbf{y})\,\d\mbf{y}
=
\iint_{\re^n} \Phi_R(\mbf{x}-\mbf{y})\,V(\mbf{x}-\mbf{y})\,h(\mbf{y})\,\d\mbf{y}.
\end{equation}
We can then define the localized half-order time derivative
\begin{multline}\label{localdtdef}
	\mathcal{D}_t^R \psi(\mbf{x})
:=
\partial_t \circ\mathrm{I}_\mathrm{P}^R  \psi(\mbf{x})
=
\iint_{\re^n} K^R (\mbf{x}-\mbf{y})\,\psi(\mbf{y})\,\d\mbf{y}
: =
\iint_{\re^n} \partial_t\big(V_R(\mbf{x}-\mbf{y})\big)\,\psi(\mbf{y})\,\d\mbf{y}
\\
=
\iint_{\re^n} \partial_t\big(V(\mbf{x}-\mbf{y})\,\Phi_R(\mbf{x}-\mbf{y})\big)\,\psi(\mbf{y})\,\d\mbf{y},
\qquad
\mbf{x}=(x,t)\in\re^n.
\end{multline}
This operator should be viewed as a principal value operator, or $\partial_t$ should be considered in the weak sense. Let $\mathcal{E}^R:=\mathcal{D}_t-\mathcal{D}_t^R$ and set $K_R := \partial_t V - K^R$. Thus
\[
\mathcal{E}^R \psi(\mbf{x})
=
\iint_{\re^n} \partial_t\big(V(\mbf{x}-\mbf{y})\,(1-\Phi_R(\mbf{x}-\mbf{y}))\big)\,\psi(\mbf{y})\,\d\mbf{y}
=
\iint_{\re^n} K_R (\mbf{x}-\mbf{y})\,\psi(\mbf{y})\,\d\mbf{y}
\]
for all $\mbf{x}=(x,t)\in\re^n$.
Recalling that $d=n+1$, we observe that
\begin{align}
	\label{KR:estimates}
	\begin{split}
	& |K_R(\mbf{x})|\lesssim \|\mbf{x}\|^{-d-1}\,\mathbbm{1}_{(Q_{9R}(0))^c}(\mbf{x}),
	\\[2pt]
	&|K_R(\mbf{x})-K_R(\mbf{x}')|\lesssim \frac{R}{(\|\mbf{x}\|+R)^{d+2}}, \qquad\text{if }\|\mbf{x}-\mbf{x}'\|\lesssim R,
	\\[2pt]
	&\iint_{\re^n} K_R(\mbf{x})\,\d\mbf{x}=0, \qquad\text{(i.e., $\mathcal{E}^R 1=0$).}
	\end{split}
\end{align}
These estimates imply that if $\|\mbf{x}-\mbf{x'}\|\lesssim R$, then
\begin{multline*}
	|\mathcal{E}^R \psi(\mbf{x})-\mathcal{E}^R \psi(\mbf{x}')|
=
\Big|
\iint_{\re^n} \big(K_R(\mbf{x}-\mbf{y})-K_R(\mbf{x}'-\mbf{y})\big)\,\big(\psi(\mbf{y})-\psi(\mbf{x})\big)\,\d\mbf{y}
\Big|
\\
\lesssim\,
\iint_{\re^n} \frac{R}{(\|\mbf{x}-\mbf{y}\|+R)^{d+2}}\,\|\mbf{y}-\mbf{x}\|\,\d\mbf{y}
\, \lesssim\,
\iint_{\re^n} \frac{R}{(\|\mbf{x}-\mbf{y}\|+R)^{d+1}}\,\d\mbf{y}
\lesssim 1.
\end{multline*}
Thus, if we pick $C=\mathcal{E}^R \psi(\mbf{x}^0)$, then
\begin{multline*}
	\frac1{|Q_R|} \iint_{F\cap Q_R} |\mathcal{D}_t \psi(\mbf{y})-C|\,\d\mbf{y}
	\\ =\,
	\frac1{|Q_R|} \iint_{F\cap Q_R} \big| \mathcal{D}_t^R \psi(\mbf{y})
+ \mathcal{E}^R \psi(\mbf{y})
	 -\mathcal{E}^R \psi(\mbf{x}^0)\big|\,\d\mbf{y}
	\, \lesssim\, 
	\frac1{|Q_R|} \iint_{F\cap Q_R} \big|\mathcal{D}_t^R \psi(\mbf{y}) \big|\,\d\mbf{y} \,+\, 1.
\end{multline*}
As a result, \eqref{BMO-q} follows once we can prove its localized version
\begin{equation}\label{BMO-q:localized}
	\frac1{|Q_R|} \iint_{F\cap Q_R} |\mathcal{D}_t^R \psi(\mbf{y})|\,\d\mbf{y}
	\le
	C_2,
\end{equation}
for a constant $C_2$ depending only on the structural constants. Having reduced matters to proving \eqref{BMO-q:localized}, we make a further reduction in the next step.

\subsection{Step 2: Replacing $\psi$ by $\psi$ along a contour}
To begin, for $\mbf{x} \in Q_{50R}(\mbf{x}^0)$ we define
\begin{equation}\label{ahdef.eq}
\psi^h(\mbf{x}):=\psi(h(\mbf{x});\mbf{x}),
\end{equation}
where $h(\mbf{x}) \approx \dist(\mbf{x}, F)$ is as introduced in Lemma \ref{regdistlem.lem} and $\psi(r;\mbf{x})$ is defined in \eqref{asemicoldef.eq}. Note that the function $\psi^h(\mbf{x})$ is well defined. Indeed, by Lemma \ref{lemma:a-h-small} it holds $h(\mbf{x}) \le R/(80\Lambda_1) \le R/\Lambda_0$ for $\mbf{x} \in Q_{50R}(\mbf{x}^0)$, so that Lemma \ref{lemma:level-sets} allows us to plug-in $r = h(\mbf{x})$ into $\psi_r(\mbf{x})$.

Recalling that $K^R(\mbf{x})=\partial_t\big(V_R(\mbf{x})\big)=\partial_t\big(V(\mbf{x})\,\Phi_R(\mbf{x})\big)$ is the 
kernel of $\mathcal{D}_t^R$, that $h\equiv 0$ in $F$, and that 
$K^R(\mbf{z}) \lesssim \|\mbf{z}\|^{-d-1}\,\mathbbm{1}_{Q_{10R}(0)}(\mbf{z})$, we see that if $\mbf{x}\in Q_R  \cap F 
=Q_R(\mbf{x}^0)\cap F$, then
\begin{multline}\label{awrfgase-}
|\mathcal{D}_t^R \psi(\mbf{x})-\mathcal{D}_t^R \psi^h(\mbf{x})|
\\ =\,
\Big|\iint_{\re^n} K^R(\mbf{x}-\mbf{y})\,\big(\psi(\mbf{y})-\psi^h(\mbf{y}) \big)\,\d\mbf{y}\Big|
\,\lesssim\,
\iint_{Q_{10R}(\mbf{x})\setminus F} \frac{\big|\psi(\mbf{y})-\psi(h(\mbf{y});\mbf{y}) \big|}{\|\mbf{x}-\mbf{y}\|^{d+1}\,}\,\d\mbf{y}.
\end{multline}
Furthermore, using \eqref{ahest} we can estimate the last term and deduce that
 for $\mbf{x}\in Q_R\cap F =Q_R(\mbf{x}^0)\cap F$,
\begin{equation}\label{awrfgase}
|\mathcal{D}_t^R \psi(\mbf{x})-\mathcal{D}_t^R \psi^h(\mbf{x})|\,\lesssim \,\iint_{Q_{10R}(\mbf{x})\setminus F}\frac{\dist(\mbf{y},F)}{\|\mbf{x}-\mbf{y}\|^{d+1}\,}\,\d\mbf{y} 
\,\le\, \iint_{Q_{20R}(\mbf{x}^0)\setminus F}\frac{\dist(\mbf{y},F)}{\|\mbf{x}-\mbf{y}\|^{d+1}\,}\,\d\mbf{y} \,, 
\end{equation}
which is a parabolic Marcinkiewicz integral; thus,
since $\dist(\mbf{y},F) \leq \|\mbf{x}-\mbf{y}\|$ for
 $\mbf{x}\in F$, 
 \begin{equation*}
\iint_{Q_R\cap F} |\mathcal{D}_t^R 
\psi(\mbf{x})-\mathcal{D}_t^R \psi^h(\mbf{x})|\, \d\mbf{x}\,
\lesssim \, \iint_{Q_{20R}(\mbf{x}^0)}\iint_{\|\mbf{x}-\mbf{y}\|\ge \dist(\mbf{y},F)}
\frac{\dist(\mbf{y},F)}{\|\mbf{x}-\mbf{y}\|^{d+1}\,}\,
 \d\mbf{x}\d\mbf{y}\\
\,\lesssim \,  |Q_R|\,.
\end{equation*}

Having previously reduced matters to proving \eqref{BMO-q:localized}, and using Cauchy-Schwarz, we can now conclude that it suffices to prove that
\begin{equation}\label{BMO-contour}
\frac1{|Q_R|} \iint_{Q_R} |\mathcal{D}^R_t \psi^h(\mbf{y})|^2\d\mbf{y}
\le
C_3,
\end{equation}
for a constant $C_3$ depending only on the structural constants.

\subsection{Step 3: Proof of \eqref{BMO-contour}}

 We let $\zeta \in C_0^\infty([-1,1])$ be an even function with $\mathbbm{1}_{[-1/2,1/2]} \le \zeta \le \mathbbm{1}_{[-1,1]}$ and set
\begin{equation}\label{prdef1}
p(x,t) = c_n\zeta(x_1)\zeta(x_2)\dots \zeta(x_{n-1})\zeta(t), \quad \forall \mbf{x} = (x,t) = (x_1,\dots,x_{n-1},t) \in \rn,
\end{equation}
where $c_n$ is chosen so that $\iint_{\rn} p(x,t) \, \d x \, \d t = 1$.  We define, for $(x,t) \in \rn$,
\begin{equation}\label{prdef2}
P_r f(x,t) := (p_r \ast f)(x,t),
\end{equation}
where $p_r(x,t) := r^{-d}p(x/r,t/r^2)$,
Thus, $P_r$ is a nice parabolic approximation to the identity on $\rn$.

Since $h$ is Lip(1,1/2) with $\|h\|_{\Lip{(1,1/2)}} \lesssim_n 1$, there exists $\gamma \ll_n 1/100$ such that
\begin{equation*}
|h(\mbf{x}) - P_{\gamma r}h(\mbf{x})| \le r/4,\quad \forall r > 0, \mbf{x} \in \rn,\]
and therefore
\begin{equation}\label{eqprh1}
r + P_{\gamma r}h(\mbf{x}) \ge 3r/4 + h(\mbf{x}) > h(\mbf{x}), \quad \forall \, r > 0, \mbf{x} \in \rn.
\end{equation}
Moreover, by \eqref {h:small},  recalling that $\Lambda$ in Lemma \ref{lemma:a-h-small} has been fixed equal to
 $\Lambda_1$, we have
\begin{equation}\label{eqhbound}
h(\mbf{x}) \leq R/(80\Lambda_1)\,, \quad \forall \, \mbf{x} \in Q_{50R}(\mbf{x}^0)\,.
\end{equation}
Consequently, for $r < R/(10\Lambda_1)$, 
\begin{equation}\label{eqprh2}
r + P_{\gamma r}h(\mbf{x}) \leq 5r/4 + h(\mbf{x}) \leq
R/(8\Lambda_1) + R/(80\Lambda_1) < R/\Lambda_1\,.
\end{equation}
In particular, if we set
\begin{equation}\label{tsdef}
\widetilde{\mathcal{S}}= \{(r,\mbf{x}),\, r \in (0,R/(10\Lambda_1)),\, \mbf{x} \in Q_{20R}(\mbf{x}^0)\},
\end{equation}
then \eqref{eqprh1}-\eqref{eqprh2} imply that
\begin{equation}\label{om1om0apcont.eq}
(r,\mbf{x}) \in \widetilde{\mathcal{S}} \implies \big(r + P_{\gamma r}h(\mbf{x}),\mbf{x}\big) \in \mathcal{S},
\end{equation}
where $\mathcal{S}$ is defined in \eqref{omz.eq}. Hence, by Lemma \ref{omnmapstar.lem},
\begin{equation}\label{om1omstar.eq}
(r,\mbf{x}) \in \widetilde{\mathcal{S}} \implies 
\big(\psi(r +P_{\gamma r}h(\mbf{x});\mbf{x}),\mbf{x}\big) \in \Omega_\star.
\end{equation}
We now define
\begin{equation}\label{atilddef.eq}
\tilde{\psi}(r;\mbf{x}): = \psi(r +P_{\gamma r}h(\mbf{x});\mbf{x})
\end{equation}
creating a map 
$ (r,\mbf{x}) \mapsto (\tilde{\psi}(r;\mbf{x}),\mbf{x})$
from $\widetilde{\mathcal{S}}$ to  $\Omega_\star$. We also introduce
\begin{equation}\label{Rprimedef}
\lambda := 1/(1000\Lambda_1)
\end{equation}
Thus, if  $r \in (0, 100\lambda R)$, then $r < R/(10\Lambda_1)$, which is the condition appearing in $\widetilde{\mathcal{S}}$.

Based on Lemma \ref{regdistlem.lem}, 
see also Remark \ref{remdist}, we refer to 
\cite[Lemma 2.8]{HL96} 
for the proof of the following lemma.

\begin{lemma}\label{PgrhLPprop.lem} Define $P_r$ as in 
\eqref{prdef1}-\eqref{prdef2}.  We have
\begin{align*}\bigl| \frac{1}{r} (P_r - I) h(\mbf{x})  \bigr| + \bigl|\partial_r P_r h(\mbf{x}) \bigr| + \bigl|r^{2j + m - 1}\nabla_{x,r}^m \partial_t^j P_{r} h(\mbf{x})\bigr| \lesssim 1,
\end{align*}
where the implicit constant depends at most on $(n,m,j)$. Furthermore,  the following estimates holds for all $Q \subset \rn$,
\begin{align*}
\mathrm{(i)}&\quad  \int_{0}^{\ell(Q)} \iint_{Q}\left|\frac{1}{r} (P_r - I) h(\mbf{x})  \right|^2\, \d\mbf{x} \, \frac{\d r}{r} \lesssim |Q|,\\
\mathrm{(ii)}&\quad \int_{0}^{\ell(Q)} \iint_{Q} \left|\partial_r P_r h(\mbf{x})  \right|^2\, \d\mbf{x} \, \frac{\d r}{r}\lesssim |Q|,\\
\mathrm{(iii)}&\quad \int_{0}^{\ell(Q)} \iint_{Q}  \left|r^{2j + m - 1}\nabla_{x,r}^m \partial_t^j P_{r} h(\mbf{x})  \right|^2 \, \d\mbf{x} \, \frac{\d r}{r} \lesssim |Q|,
\end{align*}
where in $\mathrm{(iii)}$ we require that $j,m\geq 0$, and that either 
$j \ge 1$,  or that $ m \ge 2$. Again the implicit constants depend at most on $(n,m,j)$.
\end{lemma}

The following lemma is really the heart of the matter and it is a consequence of Proposition \ref{prop:CM-ar} and Lemma \ref{PgrhLPprop.lem}. The proof of the  lemma is postponed to the next subsection.
\begin{lemma}\label{atildeprop.lem} Let $\tilde \psi$ be defined as 
in \eqref{atilddef.eq}, and $\lambda$ as in \eqref{Rprimedef}. Then
\begin{equation}\label{L3est2.eq}
|r\partial_r^2 \tilde{\psi}(r;\mbf{x})| + |r\partial_t \tilde{\psi}(r; \mbf{x})|  + |r^2 \partial_t \partial_r \tilde{\psi}(r;\mbf{x})| \lesssim 1,
\quad \forall \, r\in (0, 10\lambda R),\, \mbf{x}\in Q_{20R}\,,
\end{equation}
and
\begin{equation}\label{L3est1.eq}
\int_{0}^{10\lambda R}\! \iint_{Q_{20R}} 
\left[|r\partial_r^2 \tilde{\psi}(r;\mbf{x})|^2 
+ |r\partial_t \tilde{\psi}(r; \mbf{x})|^2  +
 |r^2 \partial_t \partial_r \tilde{\psi}(r;\mbf{x})|^2 \right]  
  \d\mbf{x}  \frac{\d r}{r} \,\lesssim\, |Q_R|.
\end{equation}
Here the implicit constants depend only on the structural constants.
\end{lemma}

Recall (see \eqref{IPRdef.eq}-\eqref{localdtdef}) that 
$\mathcal{D}_t^R := \partial_t \circ\mathrm{I}_\mathrm{P}^R$, where
$\mathrm{I}_\mathrm{P}^R$ is the localized parabolic fractional integral.
Armed with Lemma \ref{atildeprop.lem}, we shall prove \eqref{BMO-contour} by showing, 
for $f \in C_0^\infty(Q_R)$ with $\|f\|_{L^2}\leq 1$, that
\begin{equation}\label{BMO2cont.eq}
\biggl |\iint_{Q_R} \mathcal{D}_t^R \psi^h(\mbf{x}) f(\mbf{x}) \, \d\mbf{x}\biggr | 
\,=:\, \biggl |\iint \mathrm{I}_\mathrm{P}^R \psi^h(\mbf{x})\, \partial_t f(\mbf{x}) \, \d\mbf{x}\biggr | \,\lesssim\, |Q_R|^{1/2}\,.
\end{equation}
To this end, let $f \in C_0^\infty(Q_R)$ be as stated. 
Note that 
$\tilde{\psi}(0; \mbf{x}) = \psi(h(\mbf{x}); \mbf{x})
= \psi^h(\mbf{x})$, by definition of $\tilde{\psi}$, since $P_r$ is an approximate identity. Integrating by parts twice vertically, and once in $t$, we have 
\begin{multline}\label{IBPfordthalf.eq}
- \iint \mathrm{I}_\mathrm{P}^R \psi^h(\mbf{x})\, \partial_t f(\mbf{x})
 \d\mbf{x}  \,=\, - \iint
\mathrm{I}_\mathrm{P}^R \tilde{\psi}(0; \mbf{x})\, \partial_t f(\mbf{x})
 \d\mbf{x}
\\ 
=- \int_0^{\lambda R} \iint_{Q_{5R}} 
\partial_r[\mathcal{D}_t^R \tilde{\psi}(r; \mbf{x}) 
P_{\gamma r} f(\mbf{x})]\, \d\mbf{x} \, \d r \,+\, b_1
\\  = \, \int_0^{\lambda R} \iint_{Q_{5R}} 
\partial_r^2[\mathcal{D}_t^R \tilde{\psi}(r; \mbf{x}) 
P_{\gamma r} f(\mbf{x})]  \, \d\mbf{x}\, r \d r \,+\, b_1 - b_2
\, =:\, I + b_1 - b_2,
\end{multline}
where we may justify integration by parts in $t$ in the second line using 
\eqref{L3est2.eq} to make sense of
$\mathcal{D}_t^R \tilde{\psi}(r; \mbf{x})$ for $r>0$,
and where the boundary terms $b_1$ and $b_2$ are defined by
\begin{equation}\label{b1err.eq}
b_1 := \iint_{Q_{5R}} \mathcal{D}_t^R \tilde{\psi}(\lambda R; \mbf{x}) P_{\gamma \lambda R} f(\mbf{x})\, \d\mbf{x},
\end{equation}
and
\begin{equation}\label{b2err.eq}
b_2 :=  \iint_{Q_{5R}} \partial_r[\mathcal{D}_t^R \tilde{\psi}(r; \mbf{x}) P_{\gamma r} f(\mbf{x})]\big|_{r = \lambda R}\, \lambda R \, \d\mbf{x}.
\end{equation}
Note that in \eqref{IBPfordthalf.eq}, we used that $\supp P_{\gamma r} f(\cdot) \subseteq Q_{5R}$, whenever $r \le \lambda R$, since $f$ is supported in $Q_{R}$ and the kernel of
$P_{\gamma r}$ is supported in $Q_R(0)$ whenever $r < \lambda R$ (recall $\gamma \ll _n 1/100$). 
We claim, and will prove in 
subsection \ref{tech} below, that
\begin{equation}\label{b12errest.eq}
|b_1| + |b_2| \lesssim |Q_R|^{1/2}.
\end{equation}
This leaves the contribution of the main term
$I$, in which we distribute the $r$-derivatives:
\begin{multline}\label{Iintothree.eq}
I =  \int_0^{\lambda R}\! \iint_{Q_{5R}} \mathcal{D}_t^R  \partial_r^2\tilde{\psi}(r; \mbf{x}) P_{\gamma r} f(\mbf{x}) \d\mbf{x}\,r  \d r 
 \,+\, 2 \int_0^{\lambda R}\! \iint_{Q_{5R}}  \mathcal{D}_t^R \partial_r \tilde{\psi}(r; \mbf{x})  \partial_r P_{\gamma r} f(\mbf{x}) \d\mbf{x}\, r \d r
 \\
 +\,\, \int_0^{\lambda R} \!\iint_{Q_{5R}}  \mathcal{D}_t^R \tilde{\psi}(r; \mbf{x}) \partial_r^2 P_{\gamma r} f(\mbf{x})   \d\mbf{x}\, r  \d r
\,=:\, I_1 + I_2 + I_3.
\end{multline}
By parabolic Littlewood Paley theory\footnote{Estimate
\eqref{IRPLPTest.eq} may be proved via 
Plancherel's theorem.  We omit the standard argument.} 
we have
\begin{equation}\label{IRPLPTest.eq}
\int_0^\infty\! \iint_{\rn} |\mathcal{Q}^{(1)}_rf(\mbf{x})|^2 + |\mathcal{Q}^{(2)}_rf(\mbf{x})|^2 + |\mathcal{Q}^{(3)}_rf(\mbf{x})|^2 \, \d\mbf{x} \, \frac{\d r}{r} \,\lesssim\,  \|f\|_{L^2}^2 \,\le\, 1,
\end{equation}
where $\mathcal{Q}^{(j)}_r = 
\mathcal{Q}^{(j,R)}_r $ are defined by
\[
\mathcal{Q}^{(1)}_r \,=\,  r \,\mathcal{D}_t^R P_{\gamma r}, 
\quad \mathcal{Q}^{(2)}_r \,=\,  \mathrm{I}_P^R \,\partial_r P_{\gamma r} , 
\quad \mathcal{Q}^{(3)}_r \,= \,r \,\mathrm{I}_P^R \,\partial_r^2 P_{\gamma r}\,,
\]
and where the implicit constant in \eqref{IRPLPTest.eq} is independent of $R$.
Here we recall that $\mathrm{I}_P^R$ is the smoothly truncated fractional integral operator of order $1$ (see \eqref{IPRdef.eq}), and 
is self-adjoint, and that
 $\mathcal{D}_t^R = \partial_t \circ \mathrm{I}_P^R$.
Note that the cancellation for $\mathcal{Q}^{(1)}_r $ comes from this $t$-derivative.

Now we estimate $I_1$.  Using Lemma \ref{atildeprop.lem} and \eqref{IRPLPTest.eq}, and the fact that $\mathcal{D}_t^R$ is localized,
we have
\begin{multline*}
|I_1| = \left|\int_0^{\lambda R} \!\iint_{Q_{20R}} r \partial_r^2\tilde{\psi}(r; \mbf{x}) \mathcal{Q}^{(1)}_r f(\mbf{x})  \, \d\mbf{x}\,  \frac{\d r}{r}\right|
\\  
\le \left( \int_0^{\lambda R}\! \iint_{Q_{20R}} 
|r \partial_r^2\tilde{\psi}(r; \mbf{x})|^2 \, \d\mbf{x}\,  
\frac{\d r}{r}\right)^{1/2} \left(\int_0^{\lambda R} \!
\iint_{Q_{20R}} |\mathcal{Q}^{(1)}_r f(\mbf{x})|^2  \, \d\mbf{x}\,  \frac{\d r}{r}\right)^{1/2}
 \lesssim |Q_R|^{1/2}.
\end{multline*}
Similarly, using Lemma \ref{atildeprop.lem} and \eqref{IRPLPTest.eq}, 
we find that
\[|I_2| + |I_3| \lesssim  |Q_R|^{1/2}.\]
Combing our estimates for $I_1,I_2,I_3$, $b_1,b_2$, we can conclude that \eqref{BMO2cont.eq} holds.  This proves \eqref{BMO-contour} and hence Theorem \ref{mainthrm.thrm} modulo, 
Lemma \ref{atildeprop.lem}, and the claimed bounds 
for $b_1,b_2$ in \eqref{b12errest.eq}. 
The proof of these claims are given in the next two subsections.

\subsection{Proof of Lemma \ref{atildeprop.lem}} 
We note that by the
definition of $\lambda$ (see \eqref{Rprimedef}), we are 
always working with $(r,\mbf{x}) \in \widetilde{\mathcal{S}}$
(see \eqref{tsdef}).
We start by proving \eqref{L3est2.eq}. We will only handle the first term to the left in \eqref{L3est2.eq} as the rest of the terms can be handled analogously. We have
\[
\partial_r \tilde{\psi}(r;\mbf{x}) = 
\partial_r \left[
\psi(r +P_{\gamma r} h(\mbf{x});\mbf{x})\right] = 
(\partial_r \psi)(r + 
P_{\gamma r}h(\mbf{x}) ;\mbf{x})
(1 + \partial_r P_{\gamma r}h(\mbf{x})),
\]
and hence
\begin{equation}\label{dr2tpsi}
r\partial_r^2 \tilde{\psi}(r;\mbf{x})  
= r(\partial_r^2 \psi)(r + P_{\gamma r}h(\mbf{x}) ;\mbf{x}) \,
\left(1 + \partial_r P_{\gamma r}h(\mbf{x})\right)^2
\, +\, (\partial_r \psi)(r + P_{\gamma r}h(\mbf{x}) ;\mbf{x}) 
\, r\left(\partial_r^2 P_{\gamma r}h(\mbf{x})\right).
\end{equation}
The bound for $|r\partial_r^2 \tilde{\psi}(r;\mbf{x})|$ now 
follows from Lemma \ref{lemma:estimates:G-star}, specifically 
\eqref{eqn:estimates:G-star:1} and 
\eqref{eqn:estimates:G-star:3},  and 
Lemma \ref{PgrhLPprop.lem}. In particular, 
to apply 
Lemma \ref{lemma:estimates:G-star},
we use \eqref{om1om0apcont.eq}, 
and make the observation that 
$r \le r + P_{\gamma r}h(\mbf{x})$,
since $h$ is non-negative.

Next we turn our attention to \eqref{L3est1.eq}, which is a little more delicate. Again, we will only handle the first term (in the integral) as the others terms can be handled analogously. First, we control a 
closely related expression. We observe, using Proposition \ref{prop:CM-ar}, that
\begin{multline*}
\int_{0}^{10\lambda R}\iint_{Q_{20R}} \!
|r\partial_r^2 \psi(r +h(\mbf{x});\mbf{x})|^2 \, \d\mbf{x}  
\frac{\d r}{r}
  = \int_{0}^{10\lambda R}\!\iint_{Q_{20R}} 
  |\partial_r^2 \psi(r+ h(\mbf{x});\mbf{x})|^2\, r \d\mbf{x}  \d r
\\  = \int_{h(\mbf{x})}^{h(\mbf{x})+ 
10\lambda R}\iint_{Q_{20R}} 
|\partial_r^2 \psi(r;\mbf{x})|^2 \,\big(r - h(\mbf{x})\big) \d\mbf{x}  \d r
\\  \le \int_{h(\mbf{x})}^{h(\mbf{x})+ 10\lambda R}
\iint_{Q_{20R}} |\partial_r^2 \psi(r;\mbf{x})|^2 \,r \d\mbf{x} 
 \d r
\\  \le  \iiint_{\mathcal{S}} 
|\partial_r^2 \psi(r;\mbf{x})|^2 r \d\mbf{x} \d r 
\lesssim R^d\,,
\end{multline*}
where in the last step we used \eqref{psiSbound},
 and in the next-to-last inequality, we used 
\eqref{eqhbound} and 
\eqref{Rprimedef} to see 
that $h(\mbf{x})+ 10\lambda R \le R/(80\Lambda_1) 
+ R/(10\Lambda_1) < R/\Lambda_1$, so that we may
change the domain of integration to the region $\mathcal{S}$. With the preceeding estimate in hand, we see that to obtain the bound
\[\int_{0}^{10\lambda R}\iint_{Q_{20R}} |r\partial_r^2 \tilde{\psi}(r;\mbf{x})|^2\, \d\mbf{x} \, \frac{\d r}{r} \lesssim |Q_R|,\]
it is enough to prove
\begin{equation}\label{roughhsmoothhbd.eq}
\int_{0}^{10\lambda R}\iint_{Q_{20R}} 
|r\partial_r^2[\tilde{\psi}(r;\mbf{x}) - \psi(r +h(\mbf{x});\mbf{x})]|^2 \, \d\mbf{x} \, \frac{\d r}{r} \lesssim |Q_R|.
\end{equation}
To do this, we first note that by \eqref{dr2tpsi}, Lemma 
\ref{lemma:estimates:G-star}, and Lemma \ref{PgrhLPprop.lem}, 
\[
r\partial_r^2\tilde{\psi}(r;\mbf{x}) =
r(\partial_r^2 \psi)(r + P_{\gamma r}h(\mbf{x}) ;\mbf{x}) 
\,+\, O\left(
\left|\partial_r P_{\gamma r}h(\mbf{x})\right|\,+\,r\left|\partial_r^2 P_{\gamma r}h(\mbf{x})\right|\right)\,,
\]
and observe further that the ``big-$O$" term 
may be handled via Lemma \ref{PgrhLPprop.lem}.
To treat the contribution to \eqref{roughhsmoothhbd.eq}
of the remaining term, we 
use the mean value theorem:  
 for $(r,\mbf{x}) \in (0,10\lambda R) \times {Q_{20R}}$, 
 there exists $\tilde{r}$ between  $r + h(\mbf{x})$ and $r + P_{\gamma r}h(\mbf{x})$, such that
 \begin{equation}\label{eesta}
|(\partial_r^2 \psi)(r + P_{\gamma r}h(\mbf{x}) ;\mbf{x}) 
\,-\,\partial_r^2  \psi(r +h(\mbf{x});\mbf{x})| 
= |(P_{\gamma r} - I) h(\mbf{x})| 
|(\partial_r^3\psi)(\tilde{r};\mbf{x})|
 \lesssim r^{-2}|(P_{\gamma r} - I) h(\mbf{x})|.
\end{equation}
The use of \eqref{eqn:estimates:G-star:3} to 
derive the last inequality 
in this display may be justified 
by the fact that $\tilde{r}$ is between $r + h(\mbf{x})$ and 
$r + P_{\gamma r}h(\mbf{x})$; in particular,
$(\tilde{r},\mbf{x})$ lies in $\mathcal{S}$. 
Multiplying \eqref{eesta} by $r$,
and using Lemma \ref{PgrhLPprop.lem} (i), we obtain \eqref{roughhsmoothhbd.eq}, thus completing the proof of the lemma.

\subsection{Proof of the bounds for $b_1$, $b_2$} \label{tech} 
Recall that $\|f\|_{L^2} \leq 1$, and that, for $r>0$, the operator $P_{r}$ is bounded on $L^2$, uniformly in $r$.  Hence,
by the definition of $b_1$ in \eqref{b1err.eq},
\begin{multline}\label{b1bdst.eq}
|b_1|\,\le  \,\left( \iint_{Q_{5R}} |\mathcal{D}_t^R 
\tilde{\psi}(\lambda R; \mbf{x})|^2 \, \d\mbf{x}\right)^{1/2} \left(\iint_{Q_{5R}} |P_{\gamma \lambda R} f(\mbf{x})|^2\, \d\mbf{x}\right)^{1/2}
\\
 \lesssim \, \left( \iint_{Q_{5R}} |\mathcal{D}_t^R 
 \tilde{\psi}(\lambda R; \mbf{x})|^2 \, \d\mbf{x}\right)^{1/2} 
\,  \lesssim \,\|\mathcal{D}_t^R \tilde{\psi}(\lambda R; \mbf{x})\|_{L^\infty} 
 \, |Q_R|^{1/2}\,.
\end{multline}
Next, recall that $\mathcal{D}_t^R$ is a convolution operator with kernel $K^R = \partial_t V_R$, where $V_R$ is the truncated version of the kernel of
$\mathrm{I}_P$, localized at scale $R$
(see subsection \ref{ss6.1}). 
Hence,  passing the $t$-derivative onto $\tilde{\psi}(\lambda R; \mbf{x})$,
and then using \eqref{L3est2.eq}, we have
\begin{align*}
|\mathcal{D}_t^R \tilde{\psi}(\lambda R; \mbf{x})| &= |V_R \ast \partial_t \tilde{\psi}(\lambda R; \mbf{x})|\lesssim \frac{1}{R}\iint_{Q_{10R}(\mbf{x})}\|\mbf{x} - \mbf{y}\|^{1-d} \, \d \mbf{y} \lesssim 1\,.
\end{align*}
Plugging the latter bound into \eqref{b1bdst.eq} gives 
$|b_1| \lesssim |Q_R|^{1/2}$, as desired. 
The proof in the case of $b_2$ proceeds analogously 
and we omit the details.

\appendix
\section{Proof of Lemma \ref{lemma:CS-3}} \label{appendixA}

Here we give the proof of Lemma \ref{lemma:CS-3}. 
For the most part, we follow the ideas of \cite[Chapter 13]{CS} (see also
\cite{ACS}), but with some simplifications: see Remark \ref{remarkA} below.

Before proceeding, we recall that the ``vertically elongated" box $I_R$ is defined in \eqref{eq3.1}, and that
for $\mbf{X}\in \Sigma$, the time forward corkscrew point $\cA_R^+(\mbf{X})$ and the subdomain
 $\Omega_{2R}(\mbf{X})$ are defined in \eqref{CSpm} and \eqref{CS2}.
 We fix $\mbf{X} =(x_0,x,t)\in\Sigma$, set
$\Omega_{2R}= \Omega_{2R}(\mbf{X})$,  and let
	$\partialp \Omega_{2R}$ denote the parabolic boundary of $ \Omega_{2R}$, i.e.,
	\begin{equation}\label{pbdry}
	\partialp  \Omega_{2R}
	:=
	\partial \Omega_{2R} \setminus\big\{ 
	s=t+ (2R)^2 \big\}.
	\end{equation}
We split the parabolic boundary of 
$\Omega_{2R}$ into
\begin{equation}\label{eq-boundarysplit}
\partialp \Omega_{2R} = \mathcal{B} \cup \mathcal{S}\,,
\end{equation}
where $ \mathcal{B} := \overline{\Omega_{2R}}\cap \{(y_0,y,s): s=t-(2R)^2\}$ is the bottom boundary of $\Omega_{2R}$, 
and $ \mathcal{S} :=\partialp \Omega_{2R} \setminus  \mathcal{B}$ is the lateral boundary.

With $\Omega_{3R}= \Omega_{3R}(\mbf{X})$, we also set
\begin{equation}\label{ostar}
 \Omega_{*}	:= \Omega_{3R} \cap \big\{ \mbf{Y}=(y_0,y,s): y_0<x_0+6M_0\sqrt{n} R\big\}
 \end{equation}
(i.e., $\Omega_*$ is a ``shortened" version of $\Omega_{3R}$, for which 
the ``ceiling", with $y_0\equiv  6M_0\sqrt{n} R$, overlaps that of $\Omega_{2R}$).  
Thus
	\[
	\partialp  \Omega_{*}
	=
	\partial \Omega_{*} \setminus\big\{ 
	s=t+ (3R)^2 \big\},
	\]
and
\begin{equation}\label{starsplit}
\partialp \Omega_{*} = \mathcal{B}_* \cup \mathcal{S}_*\,,
\end{equation}
where $ \mathcal{B}_* := \overline{\Omega_{*}}\cap \{(y_0,y,s): s=-(3R)^2\}$ is the bottom boundary of $\Omega_{*}$, 
and $ \mathcal{S}_* :=\partialp \Omega_{*} \setminus  \mathcal{B}_*$ is the lateral boundary.	
We further define
\begin{equation}\label{sprime}
\s' := \s_* \cap \big\{(y_0,y,s):\, y_0 = x_0+6M_0\sqrt{n}R,\, -8R^2 <s-t<-5R^2\big\}
\end{equation}
Let $\hm_*$ denote caloric measure for
$\Omega_*$.

We then have the Bourgain-type estimate
\begin{equation}\label{bourgain}
\hm_*^{\mbf{Y}}\, (\s') 
\geq c_0 \,,
  \quad  \forall \,\mbf{Y}\in  \s'':=\Omega_{2R} \cap \big\{(y_0,y,s):\, y_0 = x_0+5M_0\sqrt{n}R\big\}\,.
\end{equation}
for some uniform
constant $c_0=c_0(n,M_0)>0$.
Estimate \eqref{bourgain} is well known; see e.g., \cite[Lemma 2.2]{GH} for a more general result.

Let us recall the boundary Harnack principle.

\begin{lemma}[Boundary Harnack inequality {\cite[Chapter 3, Lemma 6.5]{LewMur}}]\label{BdryHarn.lem}
Let $u$ and $v$ satisfy the hypothesis of the previous Lemma \eqref{carlest.lem} . Then there exists a constant $C \ge 1$ depending only on dimension and the $Lip(1,1/2)$ constant of the graph function $a$ such that
\[\frac{u(\mbf{Y})}{v(\mbf{Y})} \le C \frac{u(\cA^+_R(\mbf{X}))}{v(\cA^-_R(\mbf{X}))},\]
whenever $\mbf{Y} \in  I_{R/2}(\mbf{X}) \cap \Omega$.
\end{lemma}
An alternative, streamlined proof of Lemma \ref{BdryHarn.lem} can be found in \cite{DeSS}.

Next, we prove a backwards Harnack inequality for certain solutions vanishing on a surface box on $\Sigma$.
Let $\B, \s$ be as in \eqref{eq-boundarysplit}.
\begin{lemma}\label{backharnacklemma} Fix $\mbf{X}\in \Sigma$.
Let $u$ be a positive, bounded caloric function in $\Omega_{2R}=\Omega_{2R}(\mbf{X})$, vanishing continuously on
$\Sigma \cap \s$, with $u\equiv 1$ on $\B$, and $\|h\|_\infty \leq 1$.  Then for all  $r<R/2$, and for every
$\mbf{Z}\in \Delta_R(\mbf{X})$,
we have the strong Harnack inequality
\begin{equation}\label{backharnack}
u\left(\cA_r^-(\mbf{Z}\right)\lesssim u\left(\cA_r^+(\mbf{Z})\right) \lesssim u\left(\cA_r^-(\mbf{Z})\right)\,,
\end{equation}
where the implicit constants depend only upon $n$ and $M_0$.
\end{lemma}

\begin{remark}\label{remarkA1}
For solutions vanishing on all of $\s$, i.e., on the entire lateral boundary  of
$\Omega_{2R}$, 
the result appeared previously as \cite[Theorem 4]{FGS}.  
The point here is that we {\em do not require vanishing on} $\s\setminus \Sigma$.
\end{remark}

\begin{remark}\label{remarkA}
A similar result appears as \cite[Theorem 13.7]{CS} (see also
\cite{ACS}), without the restrictions that 
$\|h\|_\infty\leq 1$, and that
$h\equiv 1$ on $\B$ (or even that $h\approx  1$ on $\B$, which would work just as well), 
but without those restrictions one obtains worse dependence for
the implicit constant
in the right hand inequality in \eqref{backharnack}.  Specialized to our setting,
the result in 
\cite[Theorem 13.7]{CS} yields ours, but the
proof that we present is rather different to that of \cite[Theorem 13.7]{CS}, 
and is shorter and a bit simpler.
\end{remark}

\begin{proof}[Proof of Lemma \ref{backharnacklemma}]
The left hand inequality is simply the standard parabolic Harnack inequality. To prove the right hand inequality, we 
proceed as follows.

Fix $\mbf{X} \in \Sigma$.  
Define $\Omega_*,\B_*,\s_*$ and $\s'$ as in \eqref{ostar}, \eqref{starsplit} and \eqref{sprime}.
We define two positive auxiliary solutions: let $v$ be caloric in $\Omega_{2R}$, with $v\equiv 1$ on $\B$, $v\equiv 0$ on $\s$, and set
\[
w(\mbf{Y}):= \omega_*^{\mbf{Y}}(\s')\,,
\]
Then by the maximum principle,
\begin{equation}\label{mpone}
w\leq v \leq u \,\, \text{ in } \, \Omega_{2R}\,,
\end{equation}
and by \eqref{bourgain},
we have
\[
v(\mbf{Y}) \geq w(\mbf{Y}) \geq c_0\,, \quad  \forall \,\mbf{Y}\in \s''\,.
\]
Consequently, by the maximum principle, 
\[
u(\mbf{Y}) \lesssim v(\mbf{Y})\,, \quad \forall\, \mbf{Y}\in\Omega'':= \Omega_{2R}\cap \big\{(y_0,y,s):\, y_0< x_0+5M_0 \sqrt{n} R \big\}\,.
\]
where the implicit constants depend only on $n$ and $M_0$.  Combining the latter bound with
\eqref{mpone}, and using that by \cite[Theorem 4]{FGS}, the estimate \eqref{backharnack} applies to $v$ 
(since the latter vanishes on all of $\s$), we obtain
\[
 u\left(\cA_r^+(\mbf{Z})\right) \lesssim v\left(\cA_r^+(\mbf{Z})\right) \lesssim v\left(\cA_r^-(\mbf{Z})\right) \leq u\left(\cA_r^-(\mbf{Z})\right)\,,
\]
for all  $r<R/2$, and for every
$\mbf{Z}\in \Delta_R(\mbf{X})$, since in that case $\cA_r^\pm(\mbf{Z})\in \Omega''$.
\end{proof}

We continue to define 
$\B,\,\s$ as in \eqref{eq-boundarysplit}.
As a consequence of Lemma \ref{backharnacklemma}, we have the following.
\begin{lemma}\label{comparisonlemma} Fix $\mbf{X}\in \Sigma$.
Let $u,v$ be a positive, bounded caloric functions in $\Omega_{2R}=\Omega_{2R}(\mbf{X})$, each vanishing continuously on
$\Sigma \cap \s$, satisfying the strong Harnack inequality \eqref{backharnack} for all
$r<R/2$, and for every
$\mbf{Z}\in \Delta_R(\mbf{X})$.  We then have
\begin{equation}\label{bhp1}
\frac{u(\mbf{Y})}{v(\mbf{Y})}\, \approx \, \frac{u \left(\cA_R^+(\mbf{X})\right) }{v \left(\cA_R^+(\mbf{X})\right) }\,,\qquad 
\forall\, \mbf{Y} \in \Omega_{R/2}(\mbf{X})\,,
\end{equation}
and  for all
$r<R/2$, and for every
$\mbf{Z}\in \Delta_{R/2}(\mbf{X})$,
\begin{equation}\label{bhp2}
\left|\frac{u(\mbf{Y})}{v(\mbf{Y})} -  \frac{u(\mbf{Y'})}{v(\mbf{Y'})} \right|\, \lesssim \, \left(\frac{r}{R}\right)^\alpha\,
\frac{u \left(\cA_R^+(\mbf{X})\right) }{v \left(\cA_R^+(\mbf{X})\right) }\,,\qquad 
\forall\, \mbf{Y}, \mbf{Y}' \in \Omega_{r}(\mbf{Z})\,,
\end{equation}
where the implicit constants and $\alpha$ depend only upon $n$, $M_0$, and the implicit constants in
\eqref{backharnack}.
\end{lemma}

Estimate  \eqref{bhp1} follows from Lemma \ref{BdryHarn.lem} and the strong Harnack inequality. The estimates \eqref{bhp1} and \eqref{bhp2} are both stated without proof (and more generally, in parabolic NTA domains) 
in \cite{HLN2} as Lemma 3.18 and Lemma 3.19, respectively.
As observed in
\cite{HLN2}, the strong Harnack inequality \eqref{backharnack}
 allows one to repeat the proof given in the 
 elliptic case in \cite[Theorem 5.1 and Theorem 7.9]{JKNTA}, {\em mutatis mutandis}; 
 in particular, \eqref{bhp2} follows from
 \eqref{bhp1} and a standard iteration argument.  
For solutions vanishing continuously on all of $\s$, these results appeared previously in \cite{FGS} (see also \cite{FSY}).  
 
From this point onward, we follow very closely the proofs in
\cite{CS} (or \cite{ACS} or \cite{N2006}).  We include the remaining
arguments for the reader's convenience.

\begin{lemma}\label{lemma:CS-1}
	Let $\mbf{X}\in\Sigma$ and $R>0$. 
	Assume that $0\le u\in W^{1,2}(I_{2\,R}(\mbf{X})\cap\Omega)\cap C(\overline{I_{2\,R}(\mbf{X})\cap\Omega})$ 
	satisfies $\partial_tu -\mathcal{L}u=0$ in $I_{2\,R}(\mbf{X})\cap\Omega$ with 
	$u\equiv 0$ in $\Delta_{2\,R}(\mbf{X})$. 
Suppose further that $u$ satisfies the strong Harnack inequality \eqref{backharnack}
for all  $r<R/2$, and for every
$\mbf{Z}\in \Delta_{R}(\mbf{X})$.		
	If  $\partial_{y_0} u\ge 0$ in $I_{2\,R}(\mbf{X})\cap\Omega$, then 
	\begin{equation}\label{derivative}
	\partial_{y_0} u(\mbf{Y})\approx \frac{u(\mbf{Y})}{\delta(\mbf{Y})},
	\qquad \mbox{for every }\mbf{Y}\in I_{R}(\mbf{X})\cap\Omega.
	\end{equation}
\end{lemma}

\begin{proof}
	Let  $\mbf{X}=(\psi(\mbf{x}),\mbf{x})\in\Sigma$ and $\mbf{Y}=(y_0,\mbf{y})\in I_{R}(\mbf{X})\cap\Omega$ where $\mbf{y}=(y,t)$. 
	Choose $r>0$ 
	so that $y_0=\psi(\mbf{y})+10M_0\sqrt{n}r$, i.e.,
	$r:=(y_0-\psi(\mbf{y}))/(10M_0\sqrt{n})$. 
	Note that $\delta(\mbf{Y})\approx r$.  Since $\mbf{Y}\in I_{R}(\mbf{X})$, we see that $0<r<R/2$. Indeed:
	\[
10M_0\sqrt{n}r
	=
	y_0-\psi(\mbf{y})
	=
	y_0-\psi(\mbf{x})+ \psi(\mbf{x})-\psi(\mbf{y})
	<
	3M_0\sqrt{n}R+ M_0\,\|\mbf{x}-\mbf{y}\|
	<
	5M_0\sqrt{n}R.
	\]

	Set $\mbf{y}^-=(y,t^-):=(y,t-r^2)$, and given $0\leq\mu\le M_0$, 
	let $\mbf{Y}^-(\mu)=(\psi(\mbf{y}^-)+\mu r, \mbf{y}^-)$ so that 
$\delta(\mbf{Y}^-(\mu))\approx \mu r$. 
	Note that since $\mbf{Y}\in I_{R}(\mbf{X})$ we have
	\[
	|y_i-x_i|
	<R\,,\,\, 1\leq i\leq n, 
	\]
	and
	\[
	|t^- -t(\mbf{x})|^{\frac12}
	=
	|t-r^2-t(\mbf{x})|^{\frac12}
	\le
	r+|t-t(\mbf{x})|^{\frac12}
	<
	r+R
	<
	3R/2,
	\]
	so that in particular, $\mbf{y}^-\in Q_{3R/2}(\mbf{x})$,
	and $\|\mbf{y}^--\mbf{x}\| \leq \sqrt{n}R + 3R/2$.
	Hence for $0<\mu\le M_0$, we have
	\[
	|\psi(\mbf{y}^-)+\mu r-\psi(\mbf{x})|
	<
	\mu r+M_0\|\mbf{y}^--\mbf{x}\|
	<
	\mu r+3M_0\sqrt{n}R
	<
	4M_0\sqrt{n}R\,,
	\]
and therefore 
$\mbf{Y}^-(\mu)\in I_{3 R/2}(\mbf{X})\cap\Omega$ for 
every $0<\mu\le M_0$.

 	Applying \eqref{carlhcest.eq} (which combines
	Hölder continuity at the boundary 
with Carleson's estimate), with center 
$\mbf{Y}^-(0)$ in place of $\mbf{X}$ and at scale $r/16$ instead of $R$,  and using
Harnack's inequality and the backward Harnack inequality in \eqref{backharnack}, we obtain 
	\begin{align*}
		u(\mbf{Y}^-(\eta))
		\leq C
		\Big(\frac{\delta(\mbf{Y}^-(\eta))}{r}\Big)^\alpha\,u(\psi(\mbf{y})+r,\mbf{y})
		\leq C
		\eta^\alpha\,u(\mbf{Y}^-(M_0))
		<
		\frac12\,u(\mbf{Y}^-(M_0))
	\end{align*}  
	provided we fix $\eta=\eta(n,M_0)\ll 1$ sufficiently small. 
Thus, by the backward Harnack inequality, 
	\[
	u(\mbf{Y})
	\lesssim \frac12
	u(\mbf{Y}^-(M_0))
	<
	u(\mbf{Y}^-(M_0))- u(\mbf{Y}^-(\eta))
	=
	r\,
	\int_{\eta}^{M_0} \partial_{y_0} u(\mbf{Y}^-(\mu))\,d\mu. 
	\]
	Using next that $(\partial_t -\mathcal{L})(\partial_{y_0} u)=0$ in $I_{2\,R}(\mbf{X})\cap\Omega$, since $\partial_{y_0} u\ge 0$, we may apply 
 Harnack's inequality 
 to conclude that 
	\begin{align*}
		u(\mbf{Y})
		\lesssim
		r\,\partial_{y_0} u(\mbf{Y})
		\approx
		\delta(\mbf{Y})\,\partial_{y_0} u(\mbf{Y}).
	\end{align*}
	
	To prove the opposite inequality in \eqref{derivative},
set $\mbf{y}^+=(y,t+r^2)$ and 
	let $\mbf{Y}^+:=(\psi(\mbf{y}^+)+r, \mbf{y}^+)$. 
Then by a standard interior estimate for derivatives of caloric functions, 
	combined with the backward Harnack inequality in
\eqref{backharnack}, we have
\[
\partial_{y_0} u(\mbf{Y}) \lesssim \delta(\mbf{Y})^{-1}  u(\mbf{Y}^+) 
\lesssim \delta(\mbf{Y})^{-1} u(\mbf{Y})\,.
\]  
\end{proof}

\begin{lemma}\label{lemma:CS-2}
	Let $\mbf{X}\in\Sigma$ and $R>0$. Assume that 
	$0\le u\in W^{1,2}(I_{2R}(\mbf{X})\cap\Omega)\cap C(\overline{I_{2R}(\mbf{X})\cap\Omega})$ satisfies 
	$\partial_tu -\mathcal{L}u=0$ in $I_{2R}(\mbf{X})\cap\Omega$, with 
	$u\equiv 0$ in $\Delta_{2R}(\mbf{X})$. 
	Suppose further that $u$ satisfies the strong Harnack inequality \eqref{backharnack}
for all  $r<R/2$, and for every
$\mbf{Z}\in \Delta_{R}(\mbf{X})$.	
	Then, there exists $\eta = \eta(n,M_0)\in (0,1/16)$ such that 
	\begin{equation}\label{posderiv}
	\partial_{y_0} u(\mbf{Y})\ge 0, 
	\qquad \text{for every } \mbf{Y}\in I_{R/2}(\mbf{X})\cap\Omega \text{ with } \delta(\mbf{Y})<\eta R.
	\end{equation}
\end{lemma}

\begin{proof} 
We follow very closely the proof of \cite[Lemma 11.12]{CS}.  The strong Harnack inequality \eqref{backharnack} allows 
one to follow the elliptic argument.

	Write $\mbf{X}=(x_0,x,t)=(\psi(x,t),x,t)\in\Sigma$. 
	As above, we set 
	$\Omega_{2R}:=I_{2R}(\mbf{X})\cap\Omega$, and let 
$\partialp\Omega_{2R}$ denote the parabolic boundary of $\Omega_{2R}$ (see \eqref{pbdry}).  
As in \eqref{eq-boundarysplit}, we make the splitting $\partialp\Omega_{2R}=\B \cup \s$.
We set $F=\partialp \Omega_{2R}\setminus \Sigma$, and let $h$ be the solution of the initial-Dirichlet problem for the heat equation
in $\Omega_{2R}$, with data $h\equiv 1$ on $F$, and $h\equiv 0$ on  $\partialp \Omega_{2R}\cap \Sigma=\Delta_{2R}(\mbf{X})$.
Thus, Lemma \ref{backharnacklemma} applies to $h$.
We note that 
\begin{equation}\label{lowerbound}
\partial_{y_0} h(\mbf{Y}) \geq 0\,,\quad \forall \, \mbf{Y}\in \Omega_{2R}\,,
\end{equation}
as may be seen by the fact that for $0<\rho<1$,
\[
 h_\rho(y_0,\mbf{y}):= h(y_0 -\rho, \mbf{y}) \leq h(y_0, \mbf{y})\,,\quad \text {if }\,\, (y_0, \mbf{y}),(y_0 -\rho, \mbf{y}) \in \Omega_{2R}\,,
\]
by the maximum principle. 
	
	Set $v:= ch$, where $c$ is a positive constant chosen so that
$u \left(\cA_R^+(\mbf{X})\right) = v \left(\cA_R^+(\mbf{X})\right)$.  Of course, Lemma \ref{backharnacklemma} 
and \eqref{lowerbound} apply also to $v$.
Since both $u$ and $v$ satisfy the strong Harnack inequality \eqref{backharnack}, we may apply Lemma \ref{comparisonlemma}
to the ratio $u/v$.  In particular, for any fixed $\mbf{Y} \in \Omega_{R/2}(\mbf{X})$, with
$\delta(\mbf{Y})=: r<R/2$, and for 
$\|\mbf{Z}-\mbf{Y}\|\leq r/2$, we have
\begin{equation}\label{bhp2a}
\left| u(\mbf{Z}) \,- \, \frac{u(\mbf{Y})}{v(\mbf{Y})}\, v(\mbf{Z})\right|\, \lesssim \, \left(\frac{r}{R}\right)^\alpha\, v(\mbf{Z}) \,
\lesssim \,  r^{\alpha +1} R^{-\alpha} 
\,\partial_{z_0} v(\mbf{Z})\,, 
\end{equation}
where in the last inequality we have used Lemma \ref{lemma:CS-1} applied to $v$.  Set $\Theta=\Theta(\mbf{Y})=u(\mbf{Y})/v(\mbf{Y})$.  Then by \eqref{bhp1}, and our normalization,
$\Theta \approx 1$.  Combining
\eqref{bhp2a} with standard interior estimates for spatial derivatives of the
caloric function $u(\cdot) - \Theta v(\cdot)$, we see that for $\|\mbf{Z}-\mbf{Y}\|\leq r/2$ and $r<\eta R$, 
\[
|\partial_{z_0} u(\mbf{Z})\, -\,\Theta\, \partial_{z_0} v(\mbf{Z})| \,\lesssim \,\eta^\alpha \partial_{z_0} v(\mbf{Z})\,.
\]
In particular, the latter bound holds with $\mbf{Z}=\mbf{Y}$, so that
\[
\partial_{y_0} u(\mbf{Y})\,\geq  \,\left(\Theta\,  -\, C\eta^\alpha\right) \partial_{y_0} v(\mbf{Y})\,.
\]
Choosing $\eta$ small enough, we obtain \eqref{posderiv}.
\end{proof}

Lemmas \ref{lemma:CS-1} and \ref{lemma:CS-2} can be combined to give a version of the former on which the non-negativity of $\partial_{y_0} u$ needs not be assumed:

\begin{lemma}\label{lemma:CS-3a}
	Let $\mbf{X}\in\Sigma$ and $R>0$. Assume that $0\le u\in W^{1,2}(I_{2\,R}(\mbf{X})\cap\Omega)\cap C(\overline{I_{2\,R}(\mbf{X})\cap\Omega})$ satisfies $\partial_tu -\mathcal{L}u=0$ in $I_{2\,R}(\mbf{X})\cap\Omega$ with 
	$u\equiv 0$ in $\Delta_{2\,R}(\mbf{X})$. 
		Suppose further that $u$ satisfies the strong Harnack inequality \eqref{backharnack}
for all  $r<R/2$, and for every
$\mbf{Z}\in \Delta_{R}(\mbf{X})$.	
	Then, there exists $\eta=\eta(n,M_0)>0$ such that 
	\begin{equation}\label{eq3.25appendix}
	\partial_{y_0} u(\mbf{Y})\approx \frac{u(\mbf{Y})}{\delta(\mbf{Y})},
	\qquad \text{for every $\mbf{Y}\in I_{R/4}(\mbf{X})\cap\Omega$ with $\delta(\mbf{Y})<\eta\,R$}.
	\end{equation}
\end{lemma}

\begin{remark}\label{remark:CS-3}
Estimate \eqref{eq3.25appendix} holds also for solutions of the adjoint caloric equation, that is for
$u$ as above, but satisfying $\partial_tu +\mathcal{L}u=0$. This follows immediately from Lemma \ref{lemma:CS-3} and the change of variable
$t \to -t$. 

\end{remark}

\begin{proof}[Proof of Lemma \ref{lemma:CS-3a}]
  Let $\eta_0$ denote the constant $\eta$ in Lemma~\ref{lemma:CS-2}.
  Observe that the current hypotheses are identical to those of Lemma \ref{lemma:CS-2}, so 
  $\partial_{y_0} u \geq 0$ in  $I_{R/2}(\mbf{X})\cap\{\mbf{Y}\in\Omega: \delta(\mbf{Y})<\eta_0 R\}$. Hence, for each $\mbf{X'}\in\Delta_{R/4}(\mbf{X})$, 
  we may apply Lemma \ref{lemma:CS-1} in $I_{2\lambda R}(\mbf{X'})\cap \Omega$, with $\lambda R\approx \eta_0 (M_0\sqrt{n})^{-1} R$, to obtain
  \eqref{eq3.25appendix} with $\eta \approx \eta_0$.
\end{proof}	

\begin{proof}[Proof of Lemma \ref{lemma:CS-3}]
The proof follows immediately 
from Lemma \ref{lemma:CS-3a} (and its adjoint caloric version).
\end{proof}

\section{Parabolic SIO bounds on Lip(1,1/2) graphs 
imply parabolic uniform rectifiability.}\label{appendixB}

As noted in the introduction, all sufficiently
nice parabolic singular integral operators (SIOs) are $L^2$ 
bounded on any parabolic uniformly rectifiable set \cite[Corollary 4.9]{BHHLN-BP}.  The converse remains open, in general, but
in this appendix, we obtain a restricted version of the converse in the setting of the present paper; i.e., we observe that $L^2$ boundedness of SIOs on a Lip(1,1/2) graph, implies that the graph is {\em regular} Lip(1,1/2).   We remark that we actually require only $L^2$ bounds for
parabolic SIOs with homogeneous kernels.

\begin{definition}\label{defnice}
We shall say that $K=K(X,t)$ is a ``nice" (homogeneous) parabolic C-Z kernel 
(of homogeneous dimension $d=n+1$)
if it satisfies the following properties:

\begin{enumerate}
\item[(i)] Parabolic Homogeneity:  $K(\lambda X, \lambda^2 t) =\lambda^{-d} K(X,t)$, for all $\lambda >0$.
\smallskip
\item[(ii)] Smoothness:  $K \in C^\infty\big(\ree\setminus \{0\}\big)$.
\smallskip
\item[(iii)] Oddness in spatial variables:  $K(X,t) = -K(-X,t)$, for each $(X,t) \in \rn\times\re$.
\end{enumerate}
\end{definition}

Corresponding to any such ``nice" kernel, and given a Lip(1,1/2) graph 
$\Sigma$ with surface measure $\sigma$, 
we define for each $\eps>0$ the corresponding truncated SIO
on $\Sigma$ in the usual way:
\[ T_\eps f(x,t) = T_\eps(K)f(x,t):= \int_{\|(x-y,t-s)\|>\eps} K(x-y,t-s)\, f(y,s) \, \d\sigma(y,s)\,.
\]

\begin{proposition}\label{prop1}
Suppose that $\Sigma\subset \ree$ 
is a Lip(1,1/2) graph, on which, for every ``nice" kernel $K$ as in Definition \ref{defnice},
the corresponding truncated
SIOs $T_\eps(K)$ are bounded on $L^2(\Sigma)$, uniformly in $\eps>0$, i.e.,
\begin{equation}\label{eq2}
\sup_{\eps>0} \|T_\eps(K) f\|_{L^2(\Sigma)} \leq C_K  \| f\|_{L^2(\Sigma)}\,.
\end{equation}
Then $\Sigma$ is a {\bf regular} Lip(1,1/2) graph.
\end{proposition}

To prove the proposition, we shall make use of the following fact, established in \cite{H95}.
\begin{theorem}[\cite{H95}]\label{t3}
Let $T_\psi$ denote the parabolic Calder\'on commutator
\begin{equation}\label{defcomm}
T_\psi f(x,t) :=  p.v.
\int_{\rn} \frac{\psi(x,t)-\psi(y,s)} {(t-s)^{1+(n/2)} }\,\exp\left(\frac{-|x-y|^2}{4(t-s)}\right)1_{\{t>s\}}\, f(y,s) \, \d y\d s.
\end{equation}
Then $T: L^2(\rn) \to L^2(\rn)$ if and only if $\psi$ is a {\em regular} Lip(1,1/2) function.
\end{theorem}

\noindent{\bf Remark.} Up to a multiplicative constant, $T_\psi =[H^{1/2},\psi]$, where
$H=\partial_t -\Delta$ is the usual heat operator in $\rn$ (see \cite{H95} for details).

\smallskip

It will be convenient to let $\rn_{sp}$ denote spatial $\rn$.

\begin{proof}[Proof of Proposition \ref{prop1}] Let $\Phi \in C_0^\infty (\re)$, with
$0\leq \Phi \leq 1$, $\Phi(r)\equiv 1$ if $|r|\leq 1$, and $\Phi(r) \equiv 0,$ if $r \geq 2$.
Given a positive number $M<\infty$, set $\Phi_M:= \Phi(\cdot/M)$.
Given a unit vector $\nu \subset S^{n-1} \subset \rn_{sp}$, 
and a point $X\in \rn$, set
\[ x_\nu:= (X\cdot \nu)\nu\,,\quad x_\nu^\perp:= X- x_\nu\,.
\]
For each such $\nu$ and $M$, we define the kernels
\[ K^{\nu,M} (X,t) := \frac{x_\nu}{t^{1+(n/2)}} \exp\left(\frac{-|x_\nu^\perp|^2}{4t}\right) 1_{\{t>0\}}\,
 \Phi_M\left(\frac{x_\nu}{\|(x_\nu^\perp,t)\|}\right)\,.
\]
Note that $K^{\nu,M}$ is a ``nice" kernel in the sense of Definition \ref{defnice}, for each $\nu$ and $M$, 
with quantitative bounds that are uniform in $\nu$, but of course depend upon $M$.
Thus, by hypothesis, \eqref{eq2} holds for every $\nu,M$, for the truncated SIO 
$T_\eps^{\nu,M} := T_\eps(K^{\nu,M})$, with a quantitative bound
$C_K$ that is uniform in $\nu$, and depends quantitatively on $M$.

After a possible rotation of the spatial co-ordinates, we may suppose that
 $\Sigma=\{(\psi(x,t),x,t)\}$,
  where $\psi$ is a Lip(1,1/2) function defined on
 $\rn$.  Thus, for some positive constant $M_\psi<\infty$, we have
 \[|\psi(x,t) -\psi(y,s)|\,\leq \,M_\psi\, \|(x-y,t-s)\|\,,\quad \forall \, (x,t),(y,s) \in \rn\,.
 \]
 Let $\nu_0:=(1,0,...,0)$ denote the unit basis vector in the $x_0$ direction in 
 $\rn_{sp}$.  Choosing $\nu=\nu_0$, and $M=M_\psi$, and merging the surface area element
 $\sqrt{1+|\nabla_y\psi(y,s)|^2}$ into $f(y,s)$, we find that
 $T_\eps^{\nu_0,M_\psi}$, defined on $L^2(\Sigma)$ and written in the graph co-ordinates, is 
 merely the truncated version of the parabolic Calder\'on commutator $T_\psi$ defined in \eqref{defcomm}.
 By hypothesis, these truncations are uniformly bounded on $L^2(\rn)$, and thus by Theorem \ref{t3}, we 
 obtain that $\psi$ is a regular Lip(1,1/2) function, equivalently, that $\Sigma$ is a regular Lip(1,1/2) graph.
\end{proof}

\newcommand{\etalchar}[1]{$^{#1}$}




\end{document}